\documentclass[11pt]{scrartcl}

\newif\ifcomment
\commenttrue

\usepackage[utf8]{inputenc}
\usepackage[T1]{fontenc}
\usepackage{lmodern}

\usepackage{ifthen}
\usepackage{amssymb}
\usepackage{microtype}
\usepackage{amsmath}
\usepackage{amssymb}
\usepackage{amsfonts}
\usepackage{mathtools}
\usepackage{xcolor}
\usepackage{xstring}

 \usepackage[hyphens]{url}
 \usepackage{amsthm}
 \usepackage{hyperref}
 \usepackage{cleveref}
 \usepackage{tikz}
 \usepackage{subcaption}

 \usepackage[inline]{enumitem}
 \usepackage{titling}
 \usepackage{xspace}
 \usepackage{dsfont}

 \setlist[itemize]{topsep=0pt,partopsep=0pt,itemsep=0pt,parsep=0pt}
 \setlist[itemize,1]{label=-}
 \setlist[itemize,2]{label=---}
 \setlist[itemize,3]{label=*}
 \setlist[enumerate]{topsep=0pt,partopsep=0pt,itemsep=0pt,parsep=0pt}
 \setlist[enumerate,1]{label=\roman*)}
 \setlist[enumerate,2]{label=\alph*)}
 \setlist[enumerate,3]{label=\arabic*)}
 
\usetikzlibrary{shapes.geometric,hobby,calc} % für ellipse
\usetikzlibrary{decorations}
\usetikzlibrary{decorations.pathmorphing}
\usetikzlibrary{decorations.text}
\usetikzlibrary{shapes.misc}
\usetikzlibrary{decorations,shapes,snakes}

\usepackage{macros}

\usepackage{stackengine}
\stackMath

\renewcommand{\tilde}{\widetilde}
\renewcommand{\emptyset}{\varnothing}
{\begin{adjustwidth}{2em}{}\hspace{-2em}\textbf{Case #1.} }%
{\end{adjustwidth}}

\hypersetup{
	colorlinks=true,
	linkcolor=myBlue,
	citecolor=myBlue,
	urlcolor=myLightBlue,
	bookmarksopen=true,
	bookmarksnumbered,
	bookmarksopenlevel=2,
	bookmarksdepth=3
}
\title{Colouring Non-Even Digraphs\thanks{This research has been supported by DFG-GRK 2434 and the ERC consolidator grant DISTRUCT-648527.}
\thanks{An extended abstract of this paper was accepted at EUROCOMB 2019.}}
\predate{}
\date{}
\postdate{}

\preauthor{}
\DeclareRobustCommand{\authorthing}{
\begin{center}
\begin{tabular}{p{.3\textwidth}p{.3\textwidth}p{.3\textwidth}}
Marcelo Garlet Millani & Raphael Steiner & Sebastian Wiederrecht\\
\emph{m.garletmillani@tu-berlin.de} & \emph{steiner@math.tu-berlin.de} & \emph{sebastian.wiederrecht@tu-berlin.de}\\
\noalign{\smallskip}
\multicolumn{3}{c}{Technische Universität Berlin\quad}
\end{tabular}
\end{center}}
\author{\authorthing}
\postauthor{}

\begin{document}
\maketitle

\begin{abstract}
	A colouring of a digraph as defined by Neumann-Lara \cite{ErdNeuLara} in 1982 is a vertex-colouring such that no monochromatic directed cycles exist. The minimal number of colours required for such a colouring of a loopless digraph is defined to be its \emph{dichromatic number}. This quantity has been widely studied in the last decades and can be considered as a natural directed analogue of the chromatic number of a graph.
	A digraph $D$ is called \emph{even} if for every $0$-$1$-weighting of the edges it contains a directed cycle of even total weight.
	We show that every non-even digraph has dichromatic number at most $2$ and an optimal colouring can be found in polynomial time.
	We strengthen a previously known NP-hardness result \cite{bojannp} by showing that deciding whether a directed graph is $2$-colourable remains NP-hard even if it contains a feedback vertex set of bounded size.
	
\end{abstract}

\section{Introduction}

Graphs in this paper are considered simple, that is, without loops and multiple edges, while digraphs have no loops or parallel edges, but are allowed to have antiparallel pairs of edges (digons).
An undirected edge with \emph{endpoints} $u$ and $v$ will be denoted by $uv$, or $vu$ symmetrically, while a directed edge with \emph{tail} $u$ and \emph{head} $v$ will be denoted as $\Brace{u,v}$.
A digraph $D$ is called strongly connected if for every pair of vertices $u,v\in\Fkt{V}{D}$ there is a directed path from $u$ to $v$ and from $v$ to $u$.
The \emph{girth} of $D$ is the minimum length of a directed cycle in $D$.
We call a set $X\subseteq\Fkt{V}{D}$ \emph{acyclic}, if $\InducedSubgraph{D}{X}$ is acyclic.

A \emph{colouring} of a digraph $D$ with $k$ colours is a function $c\colon\Fkt{V}{D}\rightarrow\Set{0,\dots,k-1}$.
A colouring is called \emph{proper} if $\Fkt{c^{-1}}{i}$ is acyclic for every $i\in\Set{0,\dots,k-1}$.
The \emph{dichromatic number} $\Dichromatic{D}$ is the smallest integer $k$ such that $D$ has a proper colouring with $k$ colours.

One of the arguably most influential problems in graph theory was the Four-Colour-Conjecture, answered positively by Appel and Haken in 1976.
As a directed version of this famous theorem, the Two-Colour-Conjecture posed by Erd\H{o}s and Neumann-Lara and independently by Skrekovski (see \cite{bokal2004circular,ErdNeuLara}) still stands open.
A digraph $D$ is called \emph{oriented} if its underlying undirected graph is simple.

\begin{conjecture}\label{con:twocolours}
	Every oriented planar digraph $D$ is $2$-colourable.
\end{conjecture}

Although this conjecture has an easy formulation, there seems to be a lack of methods for attacking it. The strongest partial result proved so far is due to Mohar and Li \cite{mohar4}, who showed the following:
\begin{theorem}
	Every oriented planar digraph of girth at least $4$ is $2$-colourable.
\end{theorem}

In the undirected case, $2$-colourability is very well understood and the class of bipartite graphs can be characterised in many different ways.
For one, bipartite graphs are exactly the graphs without cycles of odd length, on the other hand the famous theorem by \Konig can also be used to characterise bipartite graphs.

\begin{theorem}[\Konig \cite{konig1931grafok}]\label{thm:konig}
	A graph $G$ is bipartite if and only if for all subgraphs $G'\subseteq G$ the size of a maximum matching of $G'$ equals the size of a minimum vertex cover.
\end{theorem}

Matchings and vertex covers can be generalised to digraphs as well.
A \emph{transversal}, or \emph{feedback vertex set}, in a digraph $D$ is a set $T$ of vertices which intersects every directed cycle in $D$, i.e.,\@ $D-T$ is acyclic.
A \emph{cycle packing} is a collection $\mathcal{C}$ of pairwise (vertex-) disjoint cycles.
The cardinality of a minimum transversal of $D$ is denoted by $\VCNum{D}$ and the cardinality of a maximum cycle packing of $D$ is denoted by $\MatNum{D}$.
We say that $D$ has the \emph{\Konig property} if $\MatNum{D'}=\VCNum{D'}$ for all subdigraphs $D'\subseteq D$.

An edge $\Brace{u,v}$ in a digraph $D$ is \emph{butterfly contractible} if it is the only outgoing edge of $u$ or the only incoming edge of $v$.
The \emph{butterfly contraction} of a butterfly contractible edge $\Brace{u,v}$ which is the only outgoing edge of $u$ is obtained from $D$ by adding the edge $\Brace{x,v}$ for every edge $\Brace{x,u}$ in $D$ (if it does not yet exist) and then deleting the vertex $u$.
Analogously, if $\Brace{u,v}$ is the only incoming edge of $v$, we obtain the butterfly contraction of $\Brace{u,v}$ by adding the edge $\Brace{u,x}$ for every edge $\Brace{v,x}$ in $D$ (if it does not yet exist) and then deleting $v$.
A digraph $D'$ is a \emph{butterfly minor} of $D$ if it can be obtained by butterfly contractions from a subdigraph of $D$.

For an undirected graph $G$, the digraph obtained from $G$ by replacing every undirected edge $xy$ with the two directed edges $\Brace{x,y}$ and $\Brace{y,x}$ is called the \emph{bidirected} graph $\Bidirected{G}$.
If $G$ is a cycle we call $\Bidirected{G}$ a \emph{bicycle}.

\begin{figure}[h!]
	\begin{center}
		\begin{tikzpicture}[scale=0.7]
		
		\pgfdeclarelayer{background}
		\pgfdeclarelayer{foreground}
		
		\pgfsetlayers{background,main,foreground}
		
		%%%%% Vertex Styles %%%%%
		\tikzstyle{v:main} = [draw, circle, scale=0.5, thick,fill=black]
		\tikzstyle{v:tree} = [draw, circle, scale=0.3, thick,fill=black]
		\tikzstyle{v:border} = [draw, circle, scale=0.75, thick,minimum size=10.5mm]
		\tikzstyle{v:mainfull} = [draw, circle, scale=1, thick,fill]
		\tikzstyle{v:ghost} = [inner sep=0pt,scale=1]
		\tikzstyle{v:marked} = [circle, scale=1.2, fill=CornflowerBlue,opacity=0.3]
		%%%%% %%%%% %%%%%
		
		%%%%% Edge Styles %%%%%
		\tikzset{>=latex} 
		\tikzstyle{e:marker} = [line width=9pt,line cap=round,opacity=0.2,color=DarkGoldenrod]
		\tikzstyle{e:colored} = [line width=1.2pt,color=BostonUniversityRed,cap=round,opacity=0.8]
		\tikzstyle{e:coloredthin} = [line width=1.1pt,opacity=0.8]
		\tikzstyle{e:coloredborder} = [line width=2pt]
		\tikzstyle{e:main} = [line width=1pt]
		\tikzstyle{e:extra} = [line width=1.3pt,color=LavenderGray]
		%%%%% %%%%% %%%%%
		
		\begin{pgfonlayer}{main}
		
		%%%%% Centered Ghost Vertices %%%%%
		\node (C) [] {};
		
		%%%%% Left Center %%%%%
		\node (C1) [v:ghost, position=180:25mm from C] {};
		%\node (U1) [v:ghost, position=90:50mm from C1] {};
		%		\node (L1) [v:ghost, position=270:27mm from C1,align=center] {};
		
		%%%%% Center %%%%%
		\node (C2) [v:ghost, position=0:0mm from C] {};
		%\node (U2) [v:ghost, position=90:50mm from C2] {};
		%		\node (L2) [v:ghost, position=270:27mm from C2,align=center] {};
		
		%%%%% Right Center %%%%%
		\node (C3) [v:ghost, position=0:25mm from C] {};
		%\node (U3) [v:ghost, position=90:50mm from C3] {};
		%		\node (L3) [v:ghost, position=270:27mm from C3,align=center] {};
		%%%%% %%%%% %%%%%

		%%%%% Vertices %%%%%
		
		%%%%% Left Center %%%%%

		%%%%% %%%%% %%%%%
		
		%%%%% Center %%%%%
		
		\node (v1) [v:main,position=90:15mm from C2] {};
		\node (v2) [v:main,position=141.4:15mm from C2] {};
		\node (v3) [v:main,position=192.8:15mm from C2] {};
		\node (v4) [v:main,position=243.2:15mm from C2] {};
		\node (v5) [v:main,position=294.6:15mm from C2] {};
		\node (v6) [v:main,position=346.1:15mm from C2] {};
		\node (v7) [v:main,position=37.5:15mm from C2] {};
		
		%%%%% %%%%% %%%%%
		
		%%%%% Right Center %%%%%
		
		%%%%% %%%%% %%%%%
		
		%%%%% %%%%% %%%%%

		%%%%% Edges %%%%%
		
		%%%%% Left Center %%%%%

		%%%%% %%%%% %%%%%
		
		%%%%% Center %%%%%
		
		\draw (v1) [e:main,->,bend right=15] to (v2);
		\draw (v2) [e:main,->,bend right=15] to (v3);
		\draw (v3) [e:main,->,bend right=15] to (v4);
		\draw (v4) [e:main,->,bend right=15] to (v5);
		\draw (v5) [e:main,->,bend right=15] to (v6);
		\draw (v6) [e:main,->,bend right=15] to (v7);
		\draw (v7) [e:main,->,bend right=15] to (v1);
		
		\draw (v1) [e:main,->] to (v6);
		\draw (v2) [e:main,->] to (v7);
		\draw (v3) [e:main,->] to (v1);
		\draw (v4) [e:main,->] to (v2);
		\draw (v5) [e:main,->] to (v3);
		\draw (v6) [e:main,->] to (v4);
		\draw (v7) [e:main,->] to (v5);
		
		%%%%% %%%%% %%%%%
		
		%%%%% Right Center %%%%%

		%%%%% %%%%% %%%%%
		
		%%%%% %%%%% %%%%%
		
		\end{pgfonlayer}
		
		%%%%% %%%%% %%%%%

		%%%%% Background %%%%%
		\begin{pgfonlayer}{background}
		
		\end{pgfonlayer}	
		%%%%% %%%%% %%%%%
		
		%%%%% Foreground %%%%%
		\begin{pgfonlayer}{foreground}

		\end{pgfonlayer}
		%%%%% %%%%% %%%%%
		\end{tikzpicture}
	\end{center}
	\caption{The digraph $F_7$.}
	\label{fig:F7}
\end{figure}
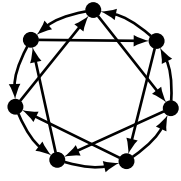

Similar to \cref{thm:konig} the digraphs with the \Konig-property can be described by forbidding odd bicycles and a single digraph called $F_7$ (illustrated in \cref{fig:F7}).
Surprisingly, this class turns out to be closed under butterfly minors.

\begin{theorem}[Guenin and Thomas \cite{guenin2011packing}]\label{thm:cyclepacking}
	A digraph $D$ has the \Konig-property if and only if it does not contain $F_7$ or an odd bicycle as a butterfly minor.
\end{theorem}

The odd bicycles also appear in another context. 
Namely, the so-called \emph{non-even digraphs} extend the class of digraphs described by \cref{thm:cyclepacking} and were helpful in the study of structural bipartite matching theory as well as in the solution of the famous \emph{even cycle problem} for digraphs.
A digraph $D$ is called \emph{even} if for every edge weighting $w\colon\Fkt{E}{D}\rightarrow\Set{0,1}$ there exists a directed cycle of even total weight in $D$.

\begin{theorem}[Seymour and Thomassen \cite{seymour1987characterization}]\label{thm:noneven}
	A directed graph is non-even if and only if it does not contain an odd bicycle as a butterfly minor.
\end{theorem}

Non-even digraphs and their recognition problem naturally correspond to a famous problem from structural matching theory.
An undirected graph $G$ is called \emph{matching covered} if $G$ is connected and for every edge $e \in \Fkt{E}{G}$ there is some $M \in \Perf{G}$ with $e \in M$, where $\Perf{G}$ denotes the set of all perfect matchings of $G$.
A set $S \subseteq \V{G}$ of vertices is called \emph{conformal} if $G-S$ has a perfect matching.
A subgraph $H\subseteq G$ is \emph{conformal}  if $\V{H}$ is a conformal set and $H$ has a perfect matching.
A cycle $C$ in $G$ is called \emph{$M$-alternating} if it alternately uses edges from $M$ and $E(G) \setminus M$. Clearly, the conformal cycles of $G$ are exactly the cycles occurring as an alternating cycle in at least one perfect matching.

Counting the number of perfect matchings in a given graph (also known as the \emph{dimer problem}) is an important and well-known task which is known to be $\#P$-hard on general graphs \cite{valiant1979complexity}.
However, there is a rather rich class of graphs for which the number of perfect matchings can be expressed as the permanent of a well-known matrix and can thus be computed in polynomial time \cite{kasteleyn1967graph,little1975characterization,thomas2006survey}, known as the \emph{Pfaffian} graphs:

A graph $G$ is called Pfaffian if there exists an orientation $\Oriented{G}$ such that every conformal cycle of $G$ contains an odd number of directed edges going in one and an odd number of directed edges going in the other direction in $\Oriented{G}$.
Such an orientation is also called \emph{Pfaffian}. 
It is well-known that any planar graph is Pfaffian (see \cite{kasteleyn1967graph}).
Since edges that are not contained in a perfect matching do not contribute to a pfaffian orientation in any way, one usually just considers matching covered graphs in this context.
Similar to non-even digraphs, bipartite matching covered Pfaffian graphs can be described by forbidden minors.
To state the complete theorem, we need a connection between directed graphs and bipartite graphs with perfect matchings, as well as the definition for minors in the context of matching covered graphs.

Let $G$ be a matching covered graph and let $v_0$ be a vertex of $G$ of degree two incident to the edges $e_1=v_0v_1$ and $e_2=v_0v_2$.
Let $H$ be obtained from $G$ by contracting both $e_1$ and $e_2$ and deleting all resulting parallel edges.
We say that $H$ is obtained from $G$ by \emph{bicontraction} or \emph{bicontracting the vertex $v_0$}.
We say that $H$ is a \emph{matching minor} of $G$ if $H$ can be obtained from a conformal subgraph of $G$ by repeatedly bicontracting vertices of degree two.
Similar to how topological minors specialise graph minors, there is the following specialisation of matching minors:
A \emph{bisubdivision} of an edge is a subdivision, i.\@{}e.\@ replacing the edge by a path joining its endpoints, with an even number (possibly $2$) of vertices.
We call $H_2$ a \emph{bisubdivision of $H_1$} if $H_1$ is a matching covered graph and $H_2$ can be obtained by bisubdividing the edges of $H_1$.
If a matching covered graph $G$ contains a conformal bisubdivision of a matching covered graph $H$, then $H$ is a matching minor of $G$, but the converse is not true. If $G$ contains no conformal bisudivision of $H$, it is called \emph{$H$-free}.

\begin{definition}
	\label{def:Mdirection}
	Let $G=\Brace{A\cup B, E}$ be a bipartite graph and let $M\in\Perf{G}$ be a perfect matching of $G$. 
	The \emph{$M$-direction} $\DirM{G}{M}$ of $G$ is defined as follows.
	Let $M = \Set{a_1b_1,\dots,a_{\Abs{M}}b_{\Abs{M}}}$
	with $a_i\in A, b_i\in B$ for $1\le i\le \Abs{M}$.
	Then,
	\begin{enumerate}
		\item $\Fkt{V}{\DirM{G}{M}}\coloneqq\Set{v_1,\dots,v_{\Abs{M}}}$ and
		
		\item $\Fkt{E}{\DirM{G}{M}}\coloneqq\CondSet{\Brace{v_i,v_j}}{a_ib_j\in\Fkt{E}{G}, i \neq j}$.	
	\end{enumerate}
\end{definition}

Note furthermore that the above operation is reversible and that every digraph $D$ is the $M$-direction of its bipartite splitting-graph equipped with the canonical perfect matching. 

The $M$-directions of a bipartite matching covered graph $G$ inherit some of the properties of $G$. Most importantly, the directed cycles in an $M$-direction are in bijection with the $M$-alternating cycles of $G$.
Another relation is about connectivity.
A graph $G$ is called \emph{$k$-extendable} if it is connected, has at least $2k+2$ vertices and every matching of size $k$ is contained in a perfect matching of $G$.
The following statement is folklore.

\begin{theorem}\label{thm:exttoconn}
	Let $G$ be a bipartite matching covered graph and $M$ a perfect matching of $G$.
	Then $G$ is $k$-extendable if and only if $\DirM{G}{M}$ is strongly $k$-(vertex-)connected.
\end{theorem}

\begin{lemma}[McCuaig \cite{mccuaig2000even}]\label{lemma:mcguigmatminors}
	Let $G$ and $H$ be bipartite matching covered graphs.
	Then $H$ is a matching minor of $G$ if and only if there exist  perfect matchings $M\in\Perf{G}$ and $M'\in\Perf{H}$ such that $\DirM{H}{M'}$ is a butterfly minor of $\DirM{G}{M}$.
\end{lemma}

The problem of describing and recognising bipartite Pfaffian graphs has given rise to a wide range of different results.
For a good overview on the topic consult the outstanding work by McCuaig \cite{mccuaig2004polya}.
For us, an important contribution is the theorem of Little \cite{little1975characterization}, which characterises bipartite Pfaffian graphs by excluding the single graph $K_{3,3}$ as a matching minor.

\begin{theorem}\label{thm:pfaffian}
	Let $G$ be a bipartite graph with a perfect matching $M$.
	The following statements are equivalent.
	\begin{enumerate}
		\item $G$ is Pfaffian.
		
		\item $G$ does not contain $K_{3,3}$ as a matching minor.
		
		\item $\DirM{G}{M}$ is non-even.
		
		\item $\DirM{G}{M}$ does not contain an odd bicycle as a butterfly minor.
	\end{enumerate}
\end{theorem}

Please note the huge discrepancy between the single forbidden minor $K_{3,3}$ in the matching setting opposed to the infinite antichain that needs to be excluded for digraphs.
We will later encounter a similar phenomenon in the proof of our main theorem.

\begin{figure}[h!]
	\begin{center}
		\begin{tikzpicture}[scale=0.7]
		
		\pgfdeclarelayer{background}
		\pgfdeclarelayer{foreground}
		
		\pgfsetlayers{background,main,foreground}
		
		%%%%% Vertex Styles %%%%%
		\tikzstyle{v:main} = [draw, circle, scale=0.5, thick,fill=black]
		\tikzstyle{v:tree} = [draw, circle, scale=0.3, thick,fill=black]
		\tikzstyle{v:border} = [draw, circle, scale=0.75, thick,minimum size=10.5mm]
		\tikzstyle{v:mainfull} = [draw, circle, scale=1, thick,fill]
		\tikzstyle{v:ghost} = [inner sep=0pt,scale=1]
		\tikzstyle{v:marked} = [circle, scale=1.2, fill=CornflowerBlue,opacity=0.3]
		%%%%% %%%%% %%%%%
		
		%%%%% Edge Styles %%%%%
		\tikzset{>=latex} 
		\tikzstyle{e:marker} = [line width=9pt,line cap=round,opacity=0.2,color=DarkGoldenrod]
		\tikzstyle{e:colored} = [line width=1.2pt,color=BostonUniversityRed,cap=round,opacity=0.8]
		\tikzstyle{e:coloredthin} = [line width=1.1pt,opacity=0.8]
		\tikzstyle{e:coloredborder} = [line width=2pt]
		\tikzstyle{e:main} = [line width=1pt]
		\tikzstyle{e:extra} = [line width=1.3pt,color=LavenderGray]
		%%%%% %%%%% %%%%%
		
		\begin{pgfonlayer}{main}
		
		%%%%% Centered Ghost Vertices %%%%%
		\node (C) [] {};
		
		%%%%% Left Center %%%%%
		\node (C1) [v:ghost, position=180:32mm from C] {};
		%\node (U1) [v:ghost, position=90:50mm from C1] {};
		\node (L1) [v:ghost, position=270:32mm from C1,align=center] {$R$};
		
		%%%%% Center %%%%%
		\node (C2) [v:ghost, position=0:0mm from C] {};
		%\node (U2) [v:ghost, position=90:50mm from C2] {};
		%		\node (L2) [v:ghost, position=270:27mm from C2,align=center] {};
		
		%%%%% Right Center %%%%%
		\node (C3) [v:ghost, position=0:32mm from C] {};
		%\node (U3) [v:ghost, position=90:50mm from C3] {};
		\node (L3) [v:ghost, position=270:32mm from C3,align=center] {$\Bidirected{C_5}$};
		%%%%% %%%%% %%%%%

		%%%%% Vertices %%%%%
		
		%%%%% Left Center %%%%%
		
		\node (r1) [v:main,position=90:25mm from C1] {};
		\node (r4) [v:main,position=180:5mm from C1] {};
		\node (r5) [v:main,position=0:5mm from C1] {};
		\node (r2) [v:main,position=135:15mm from C1] {};
		\node (r3) [v:main,position=45:15mm from C1] {};
		\node (r6) [v:main,position=225:15mm from C1] {};
		\node (r7) [v:main,position=315:15mm from C1] {};
		\node (r8) [v:main,position=270:25mm from C1] {};
		
		%%%%% %%%%% %%%%%
		
		%%%%% Center %%%%%

		%%%%% %%%%% %%%%%
		
		%%%%% Right Center %%%%%
		
		\node (c1) [v:main,position=90:15mm from C3] {};
		\node (c2) [v:main,position=162:15mm from C3] {};
		\node (c3) [v:main,position=234:15mm from C3] {};
		\node (c4) [v:main,position=306:15mm from C3] {};
		\node (c5) [v:main,position=18:15mm from C3] {};
		
		%%%%% %%%%% %%%%%
		
		%%%%% %%%%% %%%%%

		%%%%% Edges %%%%%
		
		%%%%% Left Center %%%%%
		
		\draw (r1) [e:main,->,bend right=15] to (r2);
		\draw (r1) [e:main,->,bend right=65] to (r6);
		
		\draw (r2) [e:main,->,bend right=15] to (r3);
		\draw (r2) [e:main,->] to (r4);
		\draw (r2) [e:main,->,bend right=10] to (r6);
		\draw (r2) [e:main,->,bend right=5] to (r8);
		
		\draw (r3) [e:main,->,bend right=15] to (r1);
		\draw (r3) [e:main,->,bend right=15] to (r2);
		
		\draw (r4) [e:main,->,bend right=15] to (r5);
		\draw (r4) [e:main,->] to (r6);
		
		\draw (r5) [e:main,->,bend right=15] to (r4);
		\draw (r5) [e:main,->] to (r3);
		
		\draw (r6) [e:main,->,bend right=15] to (r7);
		\draw (r6) [e:main,->,bend right=15] to (r8);
		
		\draw (r7) [e:main,->,bend right=65] to (r1);
		\draw (r7) [e:main,->,bend right=10] to (r3);
		\draw (r7) [e:main,->] to (r5);
		\draw (r7) [e:main,->,bend right=15] to (r6);
		
		\draw (r8) [e:main,->,bend right=5] to (r3);
		\draw (r8) [e:main,->,bend right=15] to (r7);
		
		%%%%% %%%%% %%%%%
		
		%%%%% Center %%%%%
		
		%%%%% %%%%% %%%%%
		
		%%%%% Right Center %%%%%
		
		\draw (c1) [e:main,->,bend right=20] to (c2);
		\draw (c1) [e:main,->,bend right=20] to (c5);
		
		\draw (c2) [e:main,->,bend right=20] to (c3);
		\draw (c2) [e:main,->,bend right=20] to (c1);
		
		\draw (c3) [e:main,->,bend right=20] to (c4);
		\draw (c3) [e:main,->,bend right=20] to (c2);
		
		\draw (c4) [e:main,->,bend right=20] to (c5);
		\draw (c4) [e:main,->,bend right=20] to (c3);
		
		\draw (c5) [e:main,->,bend right=20] to (c1);
		\draw (c5) [e:main,->,bend right=20] to (c4);
		
		%%%%% %%%%% %%%%%
		
		%%%%% %%%%% %%%%%
		
		\end{pgfonlayer}
		
		%%%%% %%%%% %%%%%

		%%%%% Background %%%%%
		\begin{pgfonlayer}{background}
		
		\end{pgfonlayer}	
		%%%%% %%%%% %%%%%
		
		%%%%% Foreground %%%%%
		\begin{pgfonlayer}{foreground}

		\end{pgfonlayer}
		%%%%% %%%%% %%%%%
		\end{tikzpicture}
	\end{center}
	\caption{The non-planar non-even digraph $R$ and the planar even digraph $\Bidirected{C_5}$.}
	\label{fig:planarity}
\end{figure}
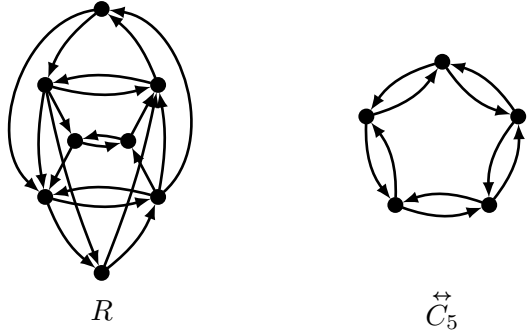

Since every matching minor of a graph is also an ordinary minor, from \cref{thm:pfaffian} it becomes clear that every planar, bipartite and matching covered graph is Pfaffian, which was known before.
However, there are also non-planar Pfaffian graphs with non-planar $M$-directions which still are non-even (for an example,  consider the graph $R$ in \cref{fig:planarity}).
On the other hand, every non-Pfaffian bipartite graph must be non-planar, but the operation of contracting a perfect matching to obtain the $M$-direction does not preserve non-planarity.
In particular, all odd bicycles are indeed planar.

Therefore, an answer to the question whether all non-even digraphs are $2$-colourable is no answer to the Two-Colour-Conjecture.
However, the class of non-even digraphs and the class of planar oriented graphs have a non-trivial intersection (see section \ref{sec:polychromatic}). For these digraphs, our main result as stated below, which however is much more general, yields a proof of \cref{con:twocolours}.

\begin{theorem}
\label{thm:mainthm}
	Let $D$ be a non-even digraph. Then $\Dichromatic{D}\leq 2$.
\end{theorem}

Given a matching covered graph $G$ and a perfect matching $M\in \Perf{G}$, an \emph{$M$-colouring} of $G$ with $k$ colours is a function $c\colon M\rightarrow\Set{0,\dots,k-1}$.
An $M$-colouring is called \emph{proper} if there is no $M$-alternating cycle whose matching edges are all of the same colour, i.e., $\Fkt{c^{-1}}{i}$ is the unique perfect matching of the subgraph of $G$ induced by the endpoints of the edges in $\Fkt{c^{-1}}{i}$ for all $i$.
The \emph{M-chromatic number} $\Mchromatic{G}{M}$ of $G$ is the smallest integer $k$ such that $G$ has a proper $M$-colouring with $k$ colours.

By the correspondence of $M$-alternating cycles in $G$ and directed cycles in $\mathcal{D}(G,M)$, we have $\chi(G,M)=\vec{\chi}(\mathcal{D}(G,M))$ for any bipartite graph $G$ with a perfect matching $M$.

From \cref{thm:mainthm,thm:pfaffian} we immediately derive the following corollary.

\begin{corollary}\label{cor:matchinghadwiger}
	Let $G$ be a bipartite graph with a perfect matching $M$.
	If $\Mchromatic{G}{M}\geq 3$, then $G$ contains $K_{3,3}$ as a matching minor.
\end{corollary}

This corollary can also be stated in the language of digraphs, where we obtain some odd bicycle as a butterfly minor instead.

Hadwiger \cite{hadwiger1943klassifikation} conjectured for the undirected chromatic number that, if $\Fkt{\chi}{G}\geq k$, $G$ would contain $K_k$ as a minor.
The case $k=5$ has been shown by Wagner \cite{wagner1937eigenschaft} to be equivalent to the Four-Colour-Theorem and, in this sense, our Main Theorem might be regarded as a directed and matching theoretic analogue of this case.

In the context of $M$-colourings of graphs one can identify certain subsets of perfect matchings, namely the \emph{forcing sets}.
Given a perfect matching $M$ of a graph, a subset $S \subseteq M$ of edges is called \emph{forcing} if $M$ is the unique perfect matching containing $S$.
The \emph{forcing number} $f(G,M)$ of a perfect matching $M$ denotes the size of a smallest forcing set for $M$.
This notion arises from resonance theory in chemistry and has attracted wide interest in the last three decades. We refer to \cite{Che-11} for a comprehensive survey on this topic. 

For any partial matching $S \subseteq M$ of a perfect matching $M$ in a graph $G$, it is clear that $S$ is forcing if and only if there is no $M$-alternating cycle with vertices in $V(G) \setminus V(S)$.
Consequently, an $M$-colouring with $k$ colours corresponds to a partition $M=S_1 \cup \dots \cup S_k$ such that for any $i$, $M \setminus S_i$ is forcing. We may thus reformulate \cref{cor:matchinghadwiger} as follows:

\begin{corollary}\label{cor:forcingdecomposition}
	Every perfect matching $M$ of a Pfaffian bipartite graph $G$ can be partitioned into two disjoint forcing sets. 
\end{corollary}

This directly yields the following corollary.

\begin{corollary} \label{cor:forcingnumber}
	For any Pfaffian bipartite graph $G$ and every perfect matching $M$ of $G$, we have $f(G,M) \leq \frac{|M|}{2}=\frac{|V(G)|}{4}$.
\end{corollary}

This generalises Theorem 2.9 in \cite{Che-11} from bipartite graphs without $K_{3,3}$ as an ordinary minor to bipartite graphs without $K_{3,3}$ as a matching minor, which is a weaker condition.

In \Cref{sec:pfaff} we consider a generalisation of the above results to non-bipartite matching covered graphs.
As these graphs bare a much more complicated structure than their bipartite cousins, we are not able to extend our colouring results in their full strength to the non-bipartite world.
Even in the planar case there are graphs with perfect matchings that are not $2$-colourable.
A smallest example of such a graph is found in the \emph{triangular prism}, which is the complement of $C_6$.
However, we are able to bring down the planar case to exactly this graph in the sense of conformal bisubdivisions and matching minors.

\begin{theorem}\label{thm:planarprismfree2colours}
	Let $G$ be a planar and matching covered graph, and $M$ a perfect matching of $G$.
	If $\Mchromatic{G}{M}\geq 3$, then $G$ contains a conformal bisubdivision of $\Complement{C_6}$, and thus has $\Complement{C_6}$ as a matching minor.
\end{theorem}

On non-bipartite graphs, to the best of our knowledge, very little is known on the forcing number (see \cite{Che-11}).
However, \cref{thm:planarprismfree2colours} implies that we can partition every perfect matching of a planar, $\Complement{C_6}$-free matching covered graph into two forcing sets and thus we obtain the following corollary.

\begin{corollary} \label{cor:nonbipforcingnumber}
	For any planar matching covered graph $G$ without a $\Complement{C_6}$ matching minor and every perfect matching $M$ of $G$, we have $f(G,M) \leq \frac{|M|}{2}=\frac{|V(G)|}{4}$.
\end{corollary}

\section{$2$-Colourings of Non-Even Digraphs} \label{sec:noneven}

This section is dedicated to the proof of \cref{thm:mainthm}.
The key idea of our proof is to consider a minimal (with respect to the number of vertices) non-$2$-colourable non-even digraph.
We introduce a number of local reductions of digraphs transporting $2$-colourability while ensuring that the reduced digraph is still non-even and prove that for any non-even digraph with at least $3$ vertices one of our reductions is applicable.

Each of our reductions can be applied in polynomial time and thus this technique implies a polynomial time algorithm for $2$-colouring a non-even digraph.

We start with two splitting operations, reducing the $2$-colouring problem to the strongly $2$-connected non-even digraphs.

\begin{definition} \label{01sums}
	Let $D$, $D_1$ and $D_2$ be digraphs.
	Then $D$ is called a \emph{$0$-sum} of $D_1$ and $D_2$ if there is a partition of $\Fkt{V}{D}$ into non-empty sets $X$ and $Y$ such that no edge of $D$ has its head in $X$ and its tail in $Y$, and $D_1=\InducedSubgraph{D}{X}$, $D_2=\InducedSubgraph{D}{Y}$.
	
	We call a strongly connected digraph $D$ the \emph{$1$-sum} of $D_1$ and $D_2$ \emph{at a vertex $v\in\Fkt{V}{D}$} if there is a partition of $\Fkt{V}{D}\setminus\Set{v}$ into non-empty sets $X$ and $Y$ such that no edge in $D$ has its head in $X$ and its tail in $Y$, and such that $D_1$ arises from $D$ by identifying $Y\cup\Set{v}$ into a single vertex and $D_2$ arises by identifying $X\cup\Set{v}$ into a single vertex. In both cases, we unify possible multiple occurences of parallel edges into single edges. 
\end{definition}

In the context of perfect matchings in bipartite graphs, the described reduction of $D$ to $D_1$ and $D_2$ corresponds to a so-called \emph{tight cut contraction}. 
Let $G$ be an undirected graph and $X\subseteq\Fkt{V}{G}$. The \emph{cut} around $X$, denoted by $\Cut{}{X}$, is the set of all edges in $G$ with exactly one endpoint in $X$.
If $G$ is matching covered and $\Abs{\Cut{}{X}\cap M}=1$ for every perfect matching $M\in\Perf{G}$, we call $\Cut{}{X}$ a \emph{tight cut}.
If $\Cut{}{X}$ is a tight cut and $\Abs{X}\geq 2$, it is \emph{non-trivial}.
Identifying the \emph{shore} $X$ of a non-trivial tight cut $\Cut{}{X}$ into a single vertex is called a \emph{tight cut contraction} and the resulting graph $G'$ can easily be seen to be matching covered again.
Among many other things, tight cut contractions can be used to produce reductions of Pfaffian graphs as shown by Vazirani and Yannakakis.

\begin{theorem}[\cite{vazirani1989pfaffian}, Theorem 4.2]\label{thm:pfaffiantcc}
	Let $G$ be a matching covered graph, $X\subseteq\Fkt{V}{G}$ such that $\Cut{}{X}$ is a non-trivial tight cut and $G_1$, $G_2$ the two graphs obtained by the tight cut contractions of $X$ and $\Complement{X}$ in $G$ respectively.
	Then $G$ is Pfaffian if and only if $G_1$ and $G_2$ are Pfaffian.
\end{theorem}

To combine the theory of tight cuts and digraphs we need to be able to translate between the two more smoothly.
Given a bipartite graph $G=\Brace{A\cup B,E}$ and a set $X\subseteq\Fkt{V}{G}$ such that $\Abs{X\cap A}<\Abs{X\cap B}$, we call $A$ the \emph{minority} and $B$ the \emph{majority} of $X$,
and analogously if the roles of $A$ and $B$ are reversed.
Consider the following characterisation of tight cuts in bipartite graphs.

\begin{lemma}[\cite{ThinEdges}, Proposition 5]\label{lemma:tightcutmajority}
	Let $G=\Brace{A\cup B,E}$ be a bipartite matching covered graph and $X\subseteq\Fkt{V}{G}$ of odd size.
	Then $\Cut{}{X}$ is tight if and only if $\big\lvert{\Abs{X\cap A}-\Abs{X\cap B}}\big\rvert=1$ and no vertex of the minority of $X$ has a neighbour in $\Complement{X}$.
\end{lemma}

In a digraph $D$ we call $\Brace{X,Y}$ a \emph{directed separation} if $X\cup Y=\Fkt{V}{D}$ and there is no edge with tail in $Y\setminus X$ and head in $X\setminus Y$.
The \emph{order} of the separation is $\Abs{X\cap Y}$.
The following is folklore, but we provide a proof for completeness.

\begin{lemma}\label{lemma:tight1sums}
	Let $G=\Brace{A\cup B,E}$ be a bipartite matching covered graph, $M$ a perfect matching in $G$ and let $X\subseteq\Fkt{V}{G}$.
	Moreover let $M_Y\coloneqq\Brace{\Fkt{E}{\InducedSubgraph{G}{Y}}\cup\Cut{}{X}}\cap M$ for $Y\in\Set{X,\Complement{X}}$ and let $v_e$ for $e\in M$ denote the vertex of the $M$-direction of $G$ corresponding to the edge $e$.
	Then $\Cut{}{X}$ is tight if and only if $\Brace{\CondSet{v_e}{e\in M_X},\CondSet{v_e}{e\in M_{\Complement{X}}}}$ or $\Brace{\CondSet{v_e}{e\in M_{\Complement{X}}},\CondSet{v_e}{e\in M_X}}$ is a directed separation of order $1$ in $\DirM{G}{M}$.
\end{lemma}

\begin{proof}
	First suppose $\Cut{}{X}$ is tight.
	By \cref{lemma:tightcutmajority} no vertex of the minority of $X$ has a neighbour in $\Complement{X}$.
	By symmetry, we may assume that $B\cap X$ is the minority of $X$.
	The $M$-direction of $G$ must be strongly connected, however there cannot exist an edge in $\DirM{G}{M}$ with head $v_e$ and tail $v_{e'}$ where $e\subseteq X$ and $e'\subseteq\Complement{X}$ since such an edge would link a vertex of $X\cap B$ to a vertex of $\Complement{X}\cap A$.
	Hence every directed path from $v_e'$ to $v_e$ must contain the vertex $v_f$ where $f$ is the unique edge of $M$ in $\Cut{}{X}$.
	Thus $\Brace{\CondSet{v_e}{e\in M_X},\CondSet{v_e}{e\in M_{\Complement{X}}}}$ is a directed separation and $v_f$ is the unique vertex in the intersection of the two sets.
	
	For the other direction let $\Brace{\CondSet{v_e}{e\in M_X},\CondSet{v_e}{e\in M_{\Complement{X}}}}$ be a directed separation of order $1$ in $\DirM{G}{M}$.
	The other case follows analogously.
	Let $f$ be the unique matching edge corresponding to the cut vertex. 
	Then every directed cycle in $\DirM{G}{M}$ must contain $v_f$ and has exactly one edge with endpoints in $\CondSet{v_e}{e\in M_X}\setminus\Set{v_f}$ and $\CondSet{v_e}{e\in M_{\Complement{X}}}\setminus\Set{v_f}$. This means that every $M$-alternating cycle in $G$ contains exactly two edges of $\Cut{}{X}$, namely $f$ and one non-matching edge.
	We know that $|\Cut{}{X} \cap M|=|\{f\}|=1$, and so to prove that $\Cut{}{X}$ is tight, we must show that any other perfect matching $M'$ of $G$ has the same number of edges on $\Cut{}{X}$ as $M$. For this, observe that the symmetric difference $M \Delta M'$ decomposes into a vertex-disjoint union of cycles $C_1,\ldots , C_t$ which are simultaneously $M$- and $M'$-alternating. Consequently, exchanging matching with non-matching edges for each $C_i$ one after the other (``flipping'') transforms $M$ into $M'$. Clearly, this operation can change the number of matching edges on $\Cut{}{X}$ only if a cycle containing vertices of both $X$ and $\Complement{X}$ is flipped, but according to the above, each such cycle must contain $f$, and so at most one $C_j$ can intersect $\Cut{}{X}$, and $E(C_j) \cap \Cut{}{X}=\{f,f'\}$ for a non-matching edge $f'$. Flipping $C_j$ now makes $f'$ into a matching and $f$ into a non-matching edge. In any case, after having performed the sequence of flips, we thus obtain that $M' \cap \Cut{}{X}$ consists of a single edge, and, hence, $\Cut{}{X}$ must be tight.
\end{proof}

From \cref{thm:pfaffiantcc,lemma:tight1sums} we obtain the following corollary.

\begin{corollary}\label{lemma:01sumnoneven}
	Let $D$ be a digraph and $i\in\Set{0,1}$ such that $D$ is the $i$-sum of the digraphs $D_1$ and $D_2$.
	Then $D$ is non-even if and only if $D_1$ and $D_2$ are non-even.
\end{corollary}
\begin{proof}
	For $i=0$, this can be seen directly from the definition of an even digraph: $D$ is non-even if and only if there is a subset $A \subseteq E(D)$ of edges intersected an odd number of times by each directed cycle. However, the set of directed cycles in $D$ consists of the directed cycles in $D[X]=D_1$ and $D[Y]=D_2$ for a partition $(X,Y)$ as in \cref{01sums}, because no directed cycle can pass trough $X$ and $Y$ at the same time. Thus, the above is the same as saying that there are edge sets $A_i \subseteq E(D_i)$, $i=1,2$, intersecting each directed cycle in $D_i$ an odd number of times, which is the same as saying that $D_1,D_2$ are non-even.
	
	For $i=1$, this is a direct consequence of \cref{lemma:tight1sums,thm:pfaffiantcc}.
\end{proof}

So $0$- and $1$-sums preserve non-eveneness. Next, we need to make sure we can obtain a $2$-colouring of $D$ from $2$-colourings of its sumands $D_1$ and $D_2$.

\begin{lemma}\label{lemma:splitting}
	Let $D$ be a non-even digraph and $D_1$, $D_2$ digraphs such that $D$ is the $i$-sum of $D_1$ and $D_2$ for $i\in\Set{0,1}$.
	If $D_1$ and $D_2$ are $2$-colourable, so is $D$.
\end{lemma}

\begin{proof}
	Assume first that $D$ is the $0$-sum of $D_1=\InducedSubgraph{D}{X}$, $D_2=\InducedSubgraph{D}{Y}$ for a partition $X,Y$ of $\Fkt{V}{D}$.
	Then the directed cycles in $D$ are exactly the directed cycles in $D_1$ together with the directed cycles in $D_2$, and thus any proper $2$-colouring of $D_1$ joined with a proper $2$-colouring of $D_2$ yields a proper $2$-colouring of $D$. 
	
	Now assume $D$ is the $1$-sum of $D_1$ and $D_2$ at $v$, let $v_1$ be the vertex of $D_1$ obtained from identifying $Y\cup\Set{v}$, and let $v_2$ be the vertex in $D_2$ identifying $X\cup\Set{v}$.
	For $i\in\Set{1,2}$ let $c_i\colon\Fkt{V}{D_i}\rightarrow\Set{0,1}$ be a proper $2$-colouring of $D_i$.
	By possibly exchanging $0$ and $1$ in $c_2$, we may assume that $\Fkt{c_1}{v_1}=\Fkt{c_2}{v_2}$.
	We define a colouring $c$ for $D$ as follows.
	\begin{align*}
	\Fkt{c}{u}\coloneqq\ThreeCases{\Fkt{c_1}{u}}{u\in X}{\Fkt{c_1}{v_1}=\Fkt{c_2}{v_2}}{u=v}{\Fkt{c_2}{u}}{u\in Y}
	\end{align*}
	To see that this defines a proper $2$-colouring of $D$, assume towards a contradiction that $C$ is a monochromatic directed cycle in $D$.
	If $C$ stays within $X \cup \Set{v}$ or $Y \cup\Set{v}$, then it also appears as a directed cycle in $D_1$, or $D_2$ respectively, contradicting the feasibility of the $2$-colourings $c_1$ and $c_2$.
	Otherwise, $C$ traverses vertices of both $X$ and $Y$ and thus, as there are no edges starting in $X$ and ending in $Y$, $C$ also contains $v$.
	Moreover, $C-v$ can be decomposed into exactly two directed paths $P_1$ and $P_2$, one contained in $X$ and the other in $Y$.
	Hence $C$ corresponds to the directed cycles $C_i = P_i+v_i$ in $D_i$ for each $i\in\Set{1,2}$ and both $C_i$ must be monochromatic under their respective colourings $c_i$.
	This again violates the feasibility of the $c_i$.
	Consequently, $c$ defines a colouring of $D$ as desired.
\end{proof}

Robertson et.\@ al.\@ \cite{robertson1999permanents} defined in total five different sum operations which they used to prove a generation theorem for non-even digraphs.
From this the following result follows.

\begin{theorem}[\cite{thomas}, Corollary 5.4]\label{thm:nonevenedges}
	Let $D$ be a strongly $2$-connected and non-even digraph on at least two vertices. Then $\Abs{\Fkt{E}{D}}\leq3\Abs{\Fkt{V}{D}}-4$.
\end{theorem}

\begin{corollary} \label{cor:degreetwo}
	Any strongly $2$-connected, non-even digraph $D$ on at least three vertices contains at least two vertices of out-degree $2$.
\end{corollary}

\begin{proof}
	Let $n:=|V(D)|$. By \cref{thm:nonevenedges} we have $|E(D)| < 3(n-1)$. If at most one vertex in $D$ had out-degree less than $3$ we would have $|E(D)|=\sum_{v \in V(D)}{\deg^{\text{out}}(v)} \ge 0+3(n-1)$, a contradiction, and so there are at least two vertices of out-degree at most, and thus, because $D$ is strongly 2-connected, exactly two. 
\end{proof}

Besides edge deletions, butterfly contractions and $0$- and $1$-sums, we will use another special operation in order to reduce our digraphs.
A bidirected $K_2$ is called a \emph{digon}.
If we encounter an out-degree $2$ vertex $v$ in a digraph $D$ such that $v$ is contained in at most one digon, we will need to delete some edges incident with $v$ in order to create a butterfly contractible edge.
However, if $v$ is contained in two different digons, we will directly contract the three digon vertices, namely $v$ and the two vertices with which $v$ forms a digon each, into a single vertex.
While this is not a standard butterfly contraction, it is natural in the context of our proof and it preserves the property of being non-even, which we show later by using matching theory.

Note that bicontractions in matching covered graphs are a special case of tight cut contractions.
To see this, consider $X$ as the set of size $3$ containing a degree $2$ vertex $v$ together with its two neighbours.
Then $\Cut{}{X}$ is tight since every perfect matching must match $v$ to one of its neighbours and thus exactly one matching edge can and must leave $X$.
Thus one can derive the following corollary from \cref{thm:pfaffiantcc} or, alternatively, \cref{thm:pfaffian}.

\begin{corollary}\label{cor:pfaffianmatminors}
	Let $G$ be a Pfaffian matching covered graph.
	Then every matching minor of $G$ is Pfaffian.
\end{corollary}

\begin{lemma} \label{lemma:3vertexcontraction}
	Let $D$ be a non-even digraph with a vertex $v\in\Fkt{V}{D}$ with $N^\text{out}(v)=\Set{v_1,v_2}$ such that $v$ induces a digon together with $v_i$ for both $i\in\Set{1,2}$.
	Then the digraph $\Star{D},$ obtained by first deleting all edges of the form $\Brace{u,v}$ with $u\notin\Set{v_1,v_2}$ as well as all edges between the vertices $v,v_1,v_2$, and then identifying $v_1$, $v$ and $v_2$ into a single vertex (and identifying occurring parallel edges into single edges afterwards), is non-even as well.
\end{lemma}

\begin{proof}
	Let $D$ be the digraph together with the vertices $v$, $v_1$, and $v_2$ as in the assertion.
	By \cref{thm:noneven}, when deleting all incoming edges of $v$ with tails other than $v_1$ or $v_2$ we obtain a subdigraph $D'$ which is non-even as well.
	Moreover, by \cref{lemma:01sumnoneven}, $D'$ is non-even if and only if every strongly connected component of $D'$ is non-even.
	Since $v$, $v_1$ and $v_2$ are contained in two digons sharing a vertex, they all must appear in the same strong component of $D'$, say, $D_0'$.
	It suffices to show that the contraction of the three vertices into one in $D_0'$ preserves non-eveness.
	
	With $D_0'$ being strongly connected, there exists a bipartite matching covered graph $G$ together with a perfect matching $M\in\Perf{G}$ such that $D_0'=\DirM{G}{M}$.
	We identify the vertices $v, v_1$ and $v_2$ of $D_0'$ as the edges $e_v$, $e_{v_1}$ and $e_{v_2}$, respectively, in $M$.
	Additionally let $A$ and $B$ be the two colour classes of $G$.
	Then $a_x$ is the vertex of $e_x$ in $A$ and $b_x$ the vertex in $B$ for all $x\in\Set{v,v_1,v_2}$.
	Since $v$ and $v_1$ form a digon in $D_0'$, the edges $a_vb_{v_1}$ and $a_{v_1}b_v$ exist in $G$ and, thus, together with $e_v$ and $e_{v_1}$ they form a conformal cycle of length $4$.
	Therefore we can obtain a new perfect matching from $M$ as follows.
	\begin{align*}
	M'\coloneqq \Brace{M\setminus\Set{e_v,e_{v_1}}}\cup\Set{a_vb_{v_1},a_{v_1}b_v}
	\end{align*}
	Now consider $G-e_v$ and note that it still has $M'$ as a perfect matching and that it is a matching minor of $G$ (see \cref{fig:3vertexcontraction} for an illustration).
	By our assumptions, $v$ has exactly two out- and two in-neighbours in $D_0'$ and therefore the two vertices $a_v$ and $b_v$ must be of degree $2$ in $G-e_v$.
	Hence we can bicontract these two vertices and identify $b_{v_1}$, $a_v$, and $b_{v_2}$ into $b_{v_1vv_2}$ and the other three vertices into $a_{v_1vv_2}$ respectively. Let us call the resulting graph $\Star{G}$ and denote the edge $a_{v_1vv_2}b_{v_1vv_2}$ by $e_{v_1vv_2}$.
	One can easily check that $\Star{G}$ still is matching covered and since it is a matching minor of $G$ it must be Pfaffian by \cref{cor:pfaffianmatminors}.
	Moreover, the strongly connected digraph $\Star{D}_0\coloneqq\DirM{\Star{G}}{\Star{M}}$ must be non-even.
	Since $\Star{M}\setminus\Set{e_{v_1vv_2}}=M'\setminus\Set{a_vb_{v_1},a_{v_1}b_v,e_{v_2}}=M\setminus\Set{e_{v_1},e_v,e_{v_2}}$ and the two edges $e_{v_1}$ and $e_{v_1vv_2}$ can be identified (again see \cref{fig:3vertexcontraction}) $\Star{D}_0$ is isomorphic to the digraph obtained from $D_0'$ identifying the three vertices $v$, $v_1$, and $v_2$ into one, and so the latter has to be non-even as well. From this we deduce that all strong components of $\Star{D}$ are non-even, proving the assertion.
\end{proof}

\begin{figure}[h!]
	\begin{center}
		\begin{tikzpicture}[scale=0.7]
		
		\pgfdeclarelayer{background}
		\pgfdeclarelayer{foreground}
		
		\pgfsetlayers{background,main,foreground}
		
		%%%%% Vertex Styles %%%%%
		%~ \tikzstyle{v:main} = [draw, circle, scale=0.5, thick,fill=black]
		%~ \tikzstyle{v:tree} = [draw, circle, scale=0.25, thick,fill=black]
		%~ \tikzstyle{v:border} = [draw, circle, scale=0.75, thick,minimum size=10.5mm]
		%~ \tikzstyle{v:mainfull} = [draw, circle, scale=1, thick,fill]
		%~ \tikzstyle{v:ghost} = [inner sep=0pt,scale=1]
		%~ \tikzstyle{v:marked} = [circle, scale=1.2, fill=CornflowerBlue,opacity=0.3]
		%%%%% %%%%% %%%%%
		
		%%%%% Edge Styles %%%%%
		%~ \tikzset{>=latex} 
		%~ \tikzstyle{e:marker} = [line width=9pt,line cap=round,opacity=0.2,color=DarkGoldenrod]
		%~ \tikzstyle{e:colored} = [line width=1.5pt,color=myGreen,cap=round,opacity=1]
		%~ \tikzstyle{e:coloredthin} = [line width=0.45pt,opacity=1,color=DarkGoldenrod]
		%~ \tikzstyle{e:coloredborder} = [line width=2.1pt]
		%~ \tikzstyle{e:main} = [line width=0.8pt]
		%~ \tikzstyle{e:extra} = [line width=1.3pt,color=LavenderGray]
		%~ \tikzstyle{e:shiftedright} = [decoration={sl, raise=0.7pt},  decorate]
		%~ \tikzstyle{e:shiftedleft}  = [decoration={sl, raise=-0.7pt}, decorate]
		%%%%% %%%%% %%%%%

		%~ \pgfdeclaredecoration{sl}{initial}{
			%~ \state{initial}[width=\pgfdecoratedpathlength-1sp]{
			 %~ \pgfmoveto{\pgfpointorigin}
			%~ }
			%~ \state{final}{
			 %~ \pgflineto{\pgfpointorigin}
			%~ }
		%~ }
		
		\begin{pgfonlayer}{main}
		
		%%%%% Centered Ghost Vertices %%%%%
		\node (C) [] {};
		
		%%%%% Center 1 %%%%%
		\node (C1) [v:ghost, position=180:75mm from C] {};
		%\node (U1) [v:ghost, position=90:50mm from C1] {};
		\node (L1) [v:ghost, position=270:20mm from C1,align=center,font=\small] {$D_0'=\DirM{G}{\textcolor{myGreen}{M}}$};
		
		%%%%% Center 2 %%%%%
		\node (C2) [v:ghost, position=180:25mm from C] {};
		%\node (U2) [v:ghost, position=90:50mm from C2] {};
		\node (L2) [v:ghost, position=270:20mm from C2,align=center,font=\small] {$G$ and $\textcolor{myGreen}{M},\textcolor{BostonUniversityRed}{M'}\in\Perf{G}$};
		
		%%%%% Center 3 %%%%%
		\node (C3) [v:ghost, position=0:25mm from C] {};
		%\node (U3) [v:ghost, position=90:50mm from C3] {};
		\node (L3) [v:ghost, position=270:19mm from C3,align=center,font=\small] {$\Star{G}$ and $\Star{M}$};
		%%%%% %%%%% %%%%%
		
		%%%%% Center 3 %%%%%
		\node (C4) [v:ghost, position=0:63mm from C] {};
		%\node (U4) [v:ghost, position=90:50mm from C4] {};
		\node (L4) [v:ghost, position=270:20.5mm from C4,align=center,font=\small] {$\Star{D}_0=\DirM{\Star{G}}{\Star{M}}$};
		%%%%% %%%%% %%%%%

		%%%%% Vertices %%%%%
		
		%%%%% Center 1 %%%%%
		
		\node (v-1) [v:main,position=0:0mm from C1,fill=myGreen] {};
		\node (v1-1) [v:main,position=150:9mm from C1,fill=myGreen] {};
		\node (v2-1) [v:main,position=30:9mm from C1,fill=myGreen] {};
		
		\node(nv1) [v:ghost,position=240:7mm from v-1] {};
		\node(nv2) [v:ghost,position=270:7mm from v-1] {};
		\node(nv3) [v:ghost,position=300:7mm from v-1] {};
		
		\node(iv1-1) [v:ghost,position=120:7mm from v1-1] {};
		\node(ov1-1) [v:ghost,position=240:7mm from v1-1] {};
		
		\node(iv2-1) [v:ghost,position=60:7mm from v2-1] {};
		\node(ov2-1) [v:ghost,position=300:7mm from v2-1] {};
		
		\node(Lv-1) [v:ghost,position=90:4mm from v-1,font=\small] {$v$};
		\node(Lv1-1) [v:ghost,position=180:5mm from v1-1,font=\small] {$v_1$};
		\node(Lv2-1) [v:ghost,position=0:5mm from v2-1,font=\small] {$v_2$};
		
		%%%%% %%%%% %%%%%
		
		%%%%% Center 2 %%%%%
		
		\node (va-2) [v:main,position=90:4.5mm from C2,fill=white] {};
		\node (vb-2) [v:main,position=270:4.5mm from C2] {};
		
		\node (v1b-2) [v:main,position=180:11mm from va-2] {};
		\node (v1a-2) [v:main,position=180:11mm from vb-2,fill=white] {};
		
		\node (v2b-2) [v:main,position=0:11mm from va-2] {};
		\node (v2a-2) [v:main,position=0:11mm from vb-2,fill=white] {};
		
		\node(iv1-2) [v:ghost,position=120:7mm from v1b-2] {};
		\node(ov1-2) [v:ghost,position=240:7mm from v1a-2] {};
		
		\node(iv2-2) [v:ghost,position=60:7mm from v2b-2] {};
		\node(ov2-2) [v:ghost,position=300:7mm from v2a-2] {};
		
		\node(Lv-2) [v:ghost,position=0:3.5mm from C2,font=\small] {$e_v$};
		\node(Lv1-2) [v:ghost,position=180:15mm from C2,font=\small] {$e_{v_1}$};
		\node(Lv2-2) [v:ghost,position=0:16mm from C2,font=\small] {$e_{v_2}$};
		
		%%%%% %%%%% %%%%%
		
		%%%%% Center 3 %%%%%
		
		\node (vb-3) [v:main,position=90:4.5mm from C3] {};
		\node (va-3) [v:main,position=270:4.5mm from C3,fill=white] {};
		
		\node(iv1-3) [v:ghost,position=120:7mm from vb-3] {};
		\node(ov1-3) [v:ghost,position=240:7mm from va-3] {};
		
		\node(iv2-3) [v:ghost,position=60:7mm from vb-3] {};
		\node(ov2-3) [v:ghost,position=300:7mm from va-3] {};
		
		\node(Lv-3) [v:ghost,position=0:8.5mm from C3,font=\small] {$e_{v_1vv_2}$};
		
		%%%%% %%%%% %%%%%
		
		%%%%% Center 4 %%%%%
		
		\node(v-4) [v:main,fill=myGreen,position=0:0mm from C4] {};
		\node(vtiny-4) [v:tree,fill=BostonUniversityRed,position=0:0mm from C4] {};
		
		\node(iv1-4) [v:ghost,position=120:7mm from v-4] {};
		\node(ov1-4) [v:ghost,position=240:7mm from v-4] {};
		
		\node(iv2-4) [v:ghost,position=60:7mm from v-4] {};
		\node(ov2-4) [v:ghost,position=300:7mm from v-4] {};
		
		\node(Lv-4) [v:ghost,position=0:9mm from C4,font=\small] {$u_{v_1vv_2}$};
		
		%%%%% %%%%% %%%%%
		
		%%%%% %%%%% %%%%%

		%%%%% Edges %%%%%
		
		%%%%% Center 1 %%%%%
		
		\draw (v-1) [e:main,->,bend right=25] to (v1-1);
		\draw (v-1) [e:main,->,bend left=25] to (v2-1);
		
		\draw (nv1) [e:main,color=gray,->,bend left=15] to (v-1);
		\draw (nv2) [e:main,color=gray,->] to (v-1);
		\draw (nv3) [e:main,color=gray,->,bend right=15] to (v-1);
		
		\draw (iv1-1) [e:main,->,bend left=15] to (v1-1);
		\draw (v1-1) [e:main,->,bend left=15] to (ov1-1);
		\draw (v1-1) [e:main,->,bend right=25] to (v-1);
		
		\draw (iv2-1) [e:main,->,bend right=15] to (v2-1);
		\draw (v2-1) [e:main,->,bend right=15] to (ov2-1);
		\draw (v2-1) [e:main,->,bend left=25] to (v-1);
		
		%%%%% %%%%% %%%%%
		
		%%%%% Center 2 %%%%%
		
		\draw (v2b-2) [e:main] to (va-2);
		\draw (v2a-2) [e:main] to (vb-2);
		
		\draw (v1b-2) [e:main] to (iv1-2);
		\draw (v1a-2) [e:main] to (ov1-2);
		
		\draw (v2b-2) [e:main] to (iv2-2);
		\draw (v2a-2) [e:main] to (ov2-2);
		
		%%%%% %%%%% %%%%%
		
		%%%%% Center 3 %%%%%
		
		\draw (vb-3) [e:main] to (iv1-3);
		\draw (va-3) [e:main] to (ov1-3);
		
		\draw (vb-3) [e:main] to (iv2-3);
		\draw (va-3) [e:main] to (ov2-3);
		
		%%%%% %%%%% %%%%%
		
		%%%%% Center 4 %%%%%
		
		\draw (iv1-4) [e:main,->,bend left=15] to (v-4);
		\draw (v-4) [e:main,->,bend left=15] to (ov1-4);
		
		\draw (iv2-4) [e:main,->,bend right=15] to (v-4);
		\draw (v-4) [e:main,->,bend right=15] to (ov2-4);
		
		%%%%% %%%%% %%%%%
		
		%%%%% %%%%% %%%%%
		
		\end{pgfonlayer}
		
		%%%%% %%%%% %%%%%

		%%%%% Background %%%%%
		\begin{pgfonlayer}{background}
		
		%%%%% Center 1 %%%%%
		
		%%%%% %%%%% %%%%%
		
		%%%%% Center 2 %%%%%
		
		%\draw (va-2) [e:coloredborder,opacity=0.3] to (vb-2);
		%\draw (va-2) [e:colored,opacity=0.3] to (vb-2);
		\draw (va-2) [e:matching1,opacity=0.3] to (vb-2);
		
		%\draw (v1a-2) [e:coloredborder] to (v1b-2);
		%\draw (v1a-2) [e:colored] to (v1b-2);
		\draw (v1a-2) [e:matching1, cap=round] to (v1b-2);
		
		%\draw (v2a-2) [line width=5pt,cap=round] to (v2b-2);
		\draw (v2a-2) [e:matching1, e:shiftedright] to (v2b-2);
		%\draw (v2a-2) [e:coloredborder, dashed] to (v2b-2);
		\draw (v2a-2) [e:matching2border, e:shiftedleft] to (v2b-2);
		%%
		% TODO: I'm not sure if a white border around the arc is better or not
		%I like it
		%%
		
		%\draw (v1b-2) [e:coloredborder, dashed] to (va-2);
		\draw (v1b-2) [e:matching2] to (va-2);
		
		%\draw (v1a-2) [e:coloredborder] to (vb-2);
		\draw (v1a-2) [e:matching2] to (vb-2);
		
		%%%%% %%%%% %%%%%
		
		%%%%% Center 3 %%%%%
		
		\draw (va-3) [e:matching1, e:shiftedright] to (vb-3);
		\draw (va-3) [e:matching2border, e:shiftedleft] to (vb-3);		
		%\draw (va-3) [e:coloredborder] to (vb-3);
		%\draw (va-3) [e:colored,color=BostonUniversityRed, color=black, line width=0.35pt, e:shiftedright, dashed, double=BostonUniversityRed, double distance=0.8pt] to (vb-3);
		
		%%%%% %%%%% %%%%%
		
		%%%%% Center 4 %%%%%
		
		%%%%% %%%%% %%%%%
		
		\end{pgfonlayer}	
		%%%%% %%%%% %%%%%
		
		%%%%% Foreground %%%%%
		\begin{pgfonlayer}{foreground}
		
		%%%%% Center 1 %%%%%
		
		%%%%% %%%%% %%%%%
		
		%%%%% Center 2 %%%%%
		
		%%%%% %%%%% %%%%%
		
		%%%%% Center 3 %%%%%
		
		%%%%% %%%%% %%%%%
		
		%%%%% Center 4 %%%%%
		
		%%%%% %%%%% %%%%%
		
		\end{pgfonlayer}
		%%%%% %%%%% %%%%%
		\end{tikzpicture}
	\end{center}
	\caption{The four steps of the contraction of $v$, $v_1$, and $v_2$ in \cref{lemma:3vertexcontraction}. The matching $M'$ is given by dashed edges while the edges of $M$ are thicker.}
	\label{fig:3vertexcontraction}
\end{figure}
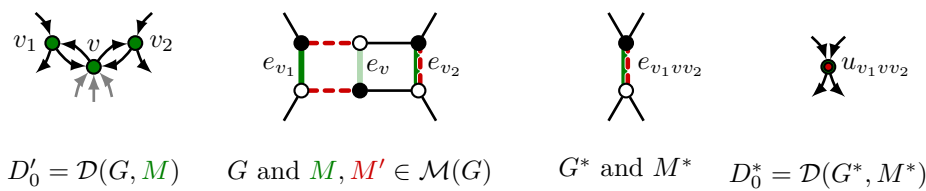

We are now ready to prove our main theorem, concluding this section.

\begin{proof}[Proof (of \cref{thm:mainthm})]
	Assume towards a contradiction that there is a non-even digraph $D$ that is not $2$-colourable.
	Furthermore, let us assume $D$ to be minimal (with respect to $\Abs{\Fkt{V}{D}}$) with this property.
	Clearly $\Abs{\Fkt{V}{D}}\geq 3$.
	
	First observe that, due to \cref{lemma:splitting}, $D$ is neither a $0$-sum nor a $1$-sum of some other non-even digraphs $D_1$ and $D_2$.
	Hence, $D$ does not have a directed cut or a cut vertex, and must therefore be strongly $2$-connected.
	By \cref{cor:degreetwo} there exists a vertex $v\in\Fkt{V}{D}$ with $\Outdeg{}{v}=2$.
	Let $e_1=\Brace{v,v_1}$ and $e_2=\Brace{v,v_2}$ be the two outgoing edges of $v$.
	We now distinguish two cases:
	
	\emph{Case 1}: Both edges $e_1$ and $e_2$ are contained in digons.
	
	If $e_1$ and $e_2$ are contained in digons, we can construct a non-even digraph $\Star{D}$ from $D$ by applying the operation from \cref{lemma:3vertexcontraction} on $v$ and its two out-neighbours.
	First, we delete all incoming edges of $v$ except $\Brace{v_1,v}$ and $\Brace{v_2,v}$ from the graph and then contract $v_1$, $v$, and $v_2$ into a single vertex.
	Since $\Abs{\Fkt{V}{\Star{D}}}=\Abs{\Fkt{V}{D}}-2$ and $\Star{D}$ is non-even, by the minimality of $D$, $\Star{D}$ admits a proper $2$-colouring $\Star{c}\colon\Fkt{V}{\Star{D}}\rightarrow\Set{0,1}$.
	Denote by $u_{v_1vv_2}$ the vertex of $\Star{D}$ into which $v_1$, $v$ and $v_2$ were identified.
	We now define a $2$-colouring for the vertices $x\in\Fkt{V}{D}$ as follows.
	\begin{align*}
	\Fkt{c}{x}\coloneqq\ThreeCases{\Fkt{\Star{c}}{u_{v_1vv_2}}}{x\in\Set{v_1,v_2}}{1-\Fkt{\Star{c}}{u_{v_1vv_2}}}{x=v}{\Fkt{\Star{c}}{x}}{\text{otherwise}}
	\end{align*}
	By assumption, $D$ is not $2$-colourable and thus there must be a directed cycle $C$ whose vertices receive the same colour from $c$.
	Moreover, $C$ must avoid $v$, since any directed cycle in $D$ containing $v$ must either contain $v_1$ or $v_2$ and thus, by the definition of $c$, cannot be monochromatic.
	Consequently, $C$ must be contained in $D-v$.
	By identifying possible occurrences of $v_1$ or $v_2$ with $u_{v_1vv_2}$, the existence of a closed directed monochromatic walk $\Star{C}$ in $\Star{D}$ follows.
	Note that $v_1$ and $v_2$ do not form a digon, as otherwise $v$, $v_1$ and $v_2$ would be an odd bicycle in $D$, contradicting the assumption that $D$ is non-even.
	Hence, the walk $\Star{C}$ must contain a directed cycle which, in turn, must also be monochromatic with respect to $\Star{c}$.
	However, the existence of such a cycle contradicts the choice of $\Star{c}$.
	
	\emph{Case 2}: At least one of the edges $e_1$ or $e_2$ is not contained in a digon.
	
	Without loss of generality assume $e_1$ to not be part of a digon in $D$.
	We now delete all edges with endpoints $v$ and $v_2$, thereby obtaining a non-even digraph in which $v$ has a single out-going edge, which is $e_1$.
	With this, $e_1$ is now butterfly contractible.
	Let $D'$ be the digraph obtained by contracting $e_1$ and let $w$ be the contraction vertex.
	Butterfly contractions are very special cases of $1$-sums, where one of the two digraphs $D_1$ and $D_2$ is a digraph on two vertices and the other one is $D'$.
	Therefore, \cref{lemma:01sumnoneven} yields that $D'$ is again non-even, alternatively, this follows from \cref{thm:noneven}.
	Moreover, as $\Abs{\Fkt{V}{D'}}=\Abs{\Fkt{V}{D}}-1$, $D'$ must admit a proper $2$-colouring $c'\colon\Fkt{V}{D'} \rightarrow \{0,1\}$ by the minimality of $D$.
	Similar to the first case we use $c'$ to define a $2$-colouring $c$ for the vertices $x\in\Fkt{V}{D}$.
	\begin{align*}
	\Fkt{c}{x}\coloneqq\ThreeCases{\Fkt{c'}{w}}{x=v_1}{1-\Fkt{c'}{v_2}}{x=v}{\Fkt{c'}{x}}{\text{otherwise}}
	\end{align*}
	Again, we assumed $D$ to not be $2$-colourable and thus there must be a monochromatic (with respect to $c$) directed cycle $C$ in $D$.
	If $C$ contains $v$, it cannot contain $v_2$ as $\Fkt{c}{v_2} \neq \Fkt{c}{v}$.
	Therefore, it must contain the edge $e_1$.
	Since $e_1$ is not contained in a digon we have $\Abs{\Fkt{V}{C}}\geq 3$ and thus there exists a cycle $C'$ in $D'$ with $\Fkt{V}{C'}\setminus{w}=\Fkt{V}{C}\setminus\Set{v,v_1}$.
	By definition of $c$, $C'$ must be monochromatic with respect to $c'$ which yields the desired contradiction in this case.
	Otherwise, $C$ does not contain $v$.
	Then, possibly after replacing $v_1$ with $w$, $C$ again corresponds to a directed cycle in $D'$ which, again, has to be monochromatic with respect to $c'$, contradicting our choice of $c'$.
\end{proof}

The proof of \cref{thm:mainthm} yields a polynomial time algorithm to find a proper $2$-colouring of a non-even digraph.
One first reduces a digraph $D$ into its strong components, then finds the cut vertices and decomposes $D$ into strongly $2$-connected digraphs.
Then, one either finds a out-degree $2$ vertex contained in two digons, which can be dealt
with by Case $1$ of the proof, or Case $2$ of the proof can be applied.
Afterwards, these reduction steps are reiterated until every such graph is reduced to a digraph on one or two vertices, which is trivially $2$-colourable.
Then, by reversing the reductions step by step, we can extend these $2$-colourings until all of $D$ is coloured.
Additionally, the work of Robertson et.\@ al.\@ and McCuaig \cite{robertson1999permanents,mccuaig2004polya} imply polynomial time algorithms to recognise non-even digraphs.
Hence, given a digraph $D$ we can decide whether it is non-even and then find a proper $2$-colouring in polynomial time.

\section{Computational Hardness}

Similar to the undirected case, the problem of deciding whether a given digraph $D$ has dichromatic number at most $k$ is \NP-complete for all $k \ge 2$ \cite{bojannp}, \cite{hochstattler2018complexity}.
For the chromatic number however, for example by using Courcelle's Theorem, one can approach colouring on undirected graphs by parametrising with treewidth \cite{courcelle1990monadic}.
While many problems become tractable for fixed parameters in the undirected case (see \cite{downey2012parameterized} for an introduction to the topic) directed width measures in general do not seem as capable \cite{ganian2010there}.
In this section we explore the computational complexity of deciding the colourability of digraphs regarding fixed parameters.

We show that the positive results for treewidth and colouring of graphs do not carry over to the world of digraphs.
More precisely and somewhat surprisingly, we show that deciding whether a digraph is $2$-colourable is \NP-hard even if $\VCNum{D} \leq 6$, where $D$ is the input digraph.
With directed treewidth being bounded in a function of $\VCNum{D}$ this implies the hardness for bounded width.
This strengthens the previous hardness reduction due to \cite{bokal2004circular}.
Formally, we consider the following decision problem.

\ProblemDef{\DigraphColouringProblem[$k$]}
{A digraph $D$.}
{Does there exist a proper $k$-colouring for $D$?}

Our hardness results bounds not only $\VCNum{D}$, but also the \emph{out-degeneracy} of $D$.

\begin{definition}
	Let $D$ be a digraph.
	The \emph{out-degeneracy} of $D$ (written $\Degeneracy{D}$) is the minimum $x$ such that a linear ordering $\preceq$ of $V(D)$ exists with the property that $\Abs{\Set{u \in \OutN{}{v} | u \preceq v}} \leq x$ for each $v \in V(D)$.
\end{definition}

The hardness result presented below is relatively tight with respect to $\VCNum{D}$ and $\Degeneracy{D}$:
If $\VCNum{D} \leq k-1$, we can find a feedback vertex set $S$ in time $f(k)n^{\BigO{1}}$ \cite{chen2008fixed}, assign each vertex of $S$ a different colour in $[k-1]$ and the remaining vertices the remaining colour $k$.
Further, one can easily find a proper $(\Degeneracy{D}+1)$-colouring of a digraph by greedily assigning each vertex a colour which does not appear in its smaller outneighbours.
Hence, if $\Degeneracy{D} \leq k-1$ or $\VCNum{D} \leq k-1$, finding a proper $k$-colouring for $D$ can be done in $f(k)n^{\BigO{1}}$ time.
In contrast, our hardness result excludes the existence of an $n^{f(k)}$-time algorithm if we only assume $\VCNum{D} \leq k+4$ and $\Degeneracy{D} \leq k+1$ instead, leaving only the cases $k \leq \VCNum{D} \leq k+3$ and $\Degeneracy{D} = k$ open.

In what follows, for a natural number $n$ we denote the set $\Set{1,\dots,n}$ by $[n]$.

\begin{lemma}
	\label{lemma:np-hard 2 fvs}
	\DigraphColouringProblem[2] is \NP-hard even if $\VCNum{D} \leq 6$ and $\Degeneracy{D} \leq 3$, where $D$ is the input digraph.
\end{lemma}
\begin{proof}
	We provide a reduction from \SAT\ to \DigraphColouringProblem[2].
	Let $C_1, C_2, \dots, C_m$ denote the clauses and $X_1, X_2, \dots, X_n$ the variables in the \SAT\ instance.
	We construct a digraph $D$ which is 2-colourable if and only if there is a satisfying assignment for the \SAT\ instance.
	For each clause $C_i$ we add the vertex $c_i$ to $D$, and for each literal $L_j \in C_i$ we add the vertex $l_{j,i}$.
	That is, we add the vertex $x_{j,i}$ if $X_j \in C_i$ and the vertex $\overline{x}_{j,i}$ if $\overline{X}_j \in C_i$.
	To simplify our notation, we assume that a literal $L_j$ is associated with the variable $X_j$, that is $L_j = X_j$ or $L_j = \overline{X}_j$, and that $l_j$ corresponds to the lower-case variant of $L_j$, that is $l_j = x_j$ if $L_j = X_j$ and $l_j = \overline{x}_j$ if $L_j = \overline{X}_j$.
	We want the colour of a vertex $x_{j,i}$ to correspond to an assignment of the variable $X_i$.
	To this end, we add a set $S = \{t_1, t_2, t_3, f_1, f_2, f_3\}$ of vertices which will correspond to a feedback vertex set in $D$.
	Furthermore, for each literal $L_j$ we add a vertex $l_j$.
	We now add cycles to $D$ in such a way that any proper colouring $c\colon \Fkt{V}{D} \rightarrow \Set{0,1}$ must have the following properties.
	\begin{enumerate}[label=(\roman*)]
		\item\label{item:lji = ljh} $\Fkt{c}{l_{j,i}} = \Fkt{c}{l_{j,h}}$ for all $j \in [n]$ and $i,h \in [m]$, and
		\item\label{item:nxji != xji} $\Fkt{c}{\overline{x}_{j,h}} \neq \Fkt{c}{x_{j,i}}$ for all $j \in [n]$ and $i,h \in [m]$.
	\end{enumerate}
	Clearly, these properties allow us to obtain a variable assignment from any proper $2$-colouring of $D$.
	
	To ensure \ref{item:lji = ljh}, we construct a \emph{literal gadget} (illustrated in \cref{subfig:literal gadget}).
	First, we add the cycle $t_1, f_1$.
	Then, for each literal $L_j$ and each clause $C_i$ with $L_j \in C_i$ we add the cycles $l_j, l_{j,i}, t_1$ and $l_j, l_{j,i}, f_1$.
	If there are $i,h \in [n]$ such that $l_{j,i}$ and $l_{j,h}$ have different colours, one of them, say, $l_{j,i}$, must have the same colour as $l_j$.
	Since $t_1, f_1$ forms a cycle, they must have different colours in any solution of the $2$-colouring problem.
	Hence, the cycle $l_j, l_{j,i}, t_1$ or the cycle $l_j, l_{j,i}, f_1$ is monochromatic if $l_{i,j}$ and $l_{i,h}$ have different colours.
	This proves \ref{item:lji = ljh}.
	
	For \ref{item:nxji != xji}, we construct a \emph{variable gadget} (illustrated in \cref{subfig:variable gadget}).
	First, we add the cycle $t_2, f_2$.
	Then, we add the cycles $x_{j}, \overline{x}_{j}, t_2$ and $x_{j}, \overline{x}_{j}, f_2$ for each $j \in [n]$ where both $X_j$ and $\overline{X}_j$ appear in the formula.
	If $x_{j}$ and $\overline{x}_{j}$ receive the same colour, then one of the added cycles is monochromatic as $t_2$ and $f_2$ must receive different colours.
	Because of the literal gadgets, we know that $l_j$ and $l_{j,i}$ have different colours for all $j \in [n]$ and $i \in [m]$.
	As $x_j$ and $\overline{x}_j$ have different colours, it follows from \ref{item:lji = ljh} that $x_{j,i}$ and $\overline{x}_{j,h}$ have different colours for all $j \in [n]$ and all $h,i \in [m]$.
	This implies \ref{item:nxji != xji}.
	\begin{figure}
		\centering
		\begin{subfigure}[t]{0.3\textwidth}
			\begin{center}
				\begin{tikzpicture}[scale=0.7]
				
				\pgfdeclarelayer{background}
				\pgfdeclarelayer{foreground}
				
				\pgfsetlayers{background,main,foreground}
				
				%%%%% Vertex Styles %%%%%
%				\tikzstyle{v:main} = [draw, circle, scale=0.5, thick,fill=black]
%				\tikzstyle{v:tree} = [draw, circle, scale=0.3, thick,fill=black]
%				\tikzstyle{v:border} = [draw, circle, scale=0.75, thick,minimum size=10.5mm]
%				\tikzstyle{v:mainfull} = [draw, circle, scale=1, thick,fill]
%				\tikzstyle{v:ghost} = [inner sep=0pt,scale=1]
%				\tikzstyle{v:marked} = [circle, scale=1.2, fill=CornflowerBlue,opacity=0.3]
%				%%%%% %%%%% %%%%%
%				
%				%%%%% Edge Styles %%%%%
%				\tikzset{>=latex} 
%				\tikzstyle{e:marker} = [line width=9pt,line cap=round,opacity=0.2,color=DarkGoldenrod]
%				\tikzstyle{e:colored} = [line width=1.2pt,color=BostonUniversityRed,cap=round,opacity=0.8]
%				\tikzstyle{e:coloredthin} = [line width=1.1pt,opacity=0.8]
%				\tikzstyle{e:coloredborder} = [line width=2pt]
%				\tikzstyle{e:main} = [line width=1pt]
%				\tikzstyle{e:extra} = [line width=1.3pt,color=LavenderGray]
%				%%%%% %%%%% %%%%%
				
				\begin{pgfonlayer}{main}
				
				%%%%% Centered Ghost Vertices %%%%%
				\node (C) [] {};
				
				%%%%% Left Center %%%%%
				\node (C1) [v:ghost, position=180:25mm from C] {};
				%\node (U1) [v:ghost, position=90:50mm from C1] {};
				%		\node (L1) [v:ghost, position=270:27mm from C1,align=center] {};
				
				%%%%% Center %%%%%
				\node (C2) [v:ghost, position=0:0mm from C] {};
				%\node (U2) [v:ghost, position=90:50mm from C2] {};
				%		\node (L2) [v:ghost, position=270:27mm from C2,align=center] {};
				
				%%%%% Right Center %%%%%
				\node (C3) [v:ghost, position=0:25mm from C] {};
				%\node (U3) [v:ghost, position=90:50mm from C3] {};
				%		\node (L3) [v:ghost, position=270:27mm from C3,align=center] {};
				%%%%% %%%%% %%%%%

				%%%%% Vertices %%%%%
				
				%%%%% Left Center %%%%%

				%%%%% %%%%% %%%%%
				
				%%%%% Center %%%%%
				
				\node (up) [v:ghost,position=90:9mm from C2] {};
				\node (x1) [v:main,position=270:9mm from C2] {};
				\node (x11) [v:main,position=180:14mm from x1] {};
				\node (x12) [v:main,position=0:14mm from x1] {};
				\node (t1) [v:main,position=180:14mm from up,fill=myGreen] {};
				\node (f1) [v:main,position=0:14mm from up,fill=BostonUniversityRed] {};
				
				\node (Lx1) [v:ghost,position=270:4.5mm from x1,font=\small] {$x_1$};
				\node (Lx11) [v:ghost,position=225:5.5mm from x11,font=\small] {$x_{1,1}$};
				\node (Lx12) [v:ghost,position=315:5.5mm from x12,font=\small] {$x_{1,2}$};
				\node (Lt1) [v:ghost,position=135:5.5mm from t1,font=\small] {$t_1$};
				\node (Lf1) [v:ghost,position=45:5.5mm from f1,font=\small] {$f_1$};
				
				%%%%% %%%%% %%%%%
				
				%%%%% Right Center %%%%%
				
				%%%%% %%%%% %%%%%
				
				%%%%% %%%%% %%%%%

				%%%%% Edges %%%%%
				
				%%%%% Left Center %%%%%

				%%%%% %%%%% %%%%%
				
				%%%%% Center %%%%%
				
				\draw (t1) [e:main,->,bend left=10] to (f1);
				\draw (f1) [e:main,->,bend left=10] to (t1);
				\draw (t1) [e:main,->] to (x1);
				\draw (f1) [e:main,->] to (x1);
				\draw (x1) [e:main,->] to (x11);
				\draw (x1) [e:main,->] to (x12);
				\draw (x11) [e:main,->] to (t1);
				\draw (x11) [e:main,->] to (f1);
				\draw (x12) [e:main,->] to (t1);
				\draw (x12) [e:main,->] to (f1);

				%%%%% %%%%% %%%%%
				
				%%%%% Right Center %%%%%

				%%%%% %%%%% %%%%%
				
				%%%%% %%%%% %%%%%
				
				\end{pgfonlayer}
				
				%%%%% %%%%% %%%%%

				%%%%% Background %%%%%
				\begin{pgfonlayer}{background}
				
				\end{pgfonlayer}	
				%%%%% %%%%% %%%%%
				
				%%%%% Foreground %%%%%
				\begin{pgfonlayer}{foreground}

				\end{pgfonlayer}
				%%%%% %%%%% %%%%%				
				\end{tikzpicture}
			\end{center}
			\subcaption{Literal gadget.}
			\label{subfig:literal gadget}
		\end{subfigure}
		~
		\begin{subfigure}[t]{0.3\textwidth}
			\begin{center}
				\begin{tikzpicture}[scale=0.7]
				
				\pgfdeclarelayer{background}
				\pgfdeclarelayer{foreground}
				
				\pgfsetlayers{background,main,foreground}
				
				%%%%% Vertex Styles %%%%%
%				\tikzstyle{v:main} = [draw, circle, scale=0.5, thick,fill=black]
%				\tikzstyle{v:tree} = [draw, circle, scale=0.3, thick,fill=black]
%				\tikzstyle{v:border} = [draw, circle, scale=0.75, thick,minimum size=10.5mm]
%				\tikzstyle{v:mainfull} = [draw, circle, scale=1, thick,fill]
%				\tikzstyle{v:ghost} = [inner sep=0pt,scale=1]
%				\tikzstyle{v:marked} = [circle, scale=1.2, fill=CornflowerBlue,opacity=0.3]
%				%%%%% %%%%% %%%%%
%				
%				%%%%% Edge Styles %%%%%
%				\tikzset{>=latex} 
%				\tikzstyle{e:marker} = [line width=9pt,line cap=round,opacity=0.2,color=DarkGoldenrod]
%				\tikzstyle{e:colored} = [line width=1.2pt,color=BostonUniversityRed,cap=round,opacity=0.8]
%				\tikzstyle{e:coloredthin} = [line width=1.1pt,opacity=0.8]
%				\tikzstyle{e:coloredborder} = [line width=2pt]
%				\tikzstyle{e:main} = [line width=1pt]
%				\tikzstyle{e:extra} = [line width=1.3pt,color=LavenderGray]
				%%%%% %%%%% %%%%%
				
				\begin{pgfonlayer}{main}
				
				%%%%% Centered Ghost Vertices %%%%%
				\node (C) [] {};
				
				%%%%% Left Center %%%%%
				\node (C1) [v:ghost, position=180:25mm from C] {};
				%\node (U1) [v:ghost, position=90:50mm from C1] {};
				%		\node (L1) [v:ghost, position=270:27mm from C1,align=center] {};
				
				%%%%% Center %%%%%
				\node (C2) [v:ghost, position=0:0mm from C] {};
				%\node (U2) [v:ghost, position=90:50mm from C2] {};
				%		\node (L2) [v:ghost, position=270:27mm from C2,align=center] {};
				
				%%%%% Right Center %%%%%
				\node (C3) [v:ghost, position=0:25mm from C] {};
				%\node (U3) [v:ghost, position=90:50mm from C3] {};
				%		\node (L3) [v:ghost, position=270:27mm from C3,align=center] {};
				%%%%% %%%%% %%%%%

				%%%%% Vertices %%%%%
				
				%%%%% Left Center %%%%%

				%%%%% %%%%% %%%%%
				
				%%%%% Center %%%%%
				
				\node (up) [v:ghost,position=90:9mm from C2] {};
				\node (down) [v:ghost,position=270:9mm from C2] {};
				\node (x1) [v:main,position=180:9mm from down] {};
				\node (x1B) [v:main,position=0:9mm from down] {};
				\node (t2) [v:main,position=180:9mm from up,fill=myGreen] {};
				\node (f2) [v:main,position=0:9mm from up,fill=BostonUniversityRed] {};
				
				\node (Lx1) [v:ghost,position=225:5.5mm from x1,font=\small] {$x_1$};
				\node (Lx1B) [v:ghost,position=315:5.5mm from x1B,font=\small] {$\overline{x}_1$};
				\node (Lt2) [v:ghost,position=135:5.5mm from t2,font=\small] {$t_2$};
				\node (Lf2) [v:ghost,position=45:5.5mm from f2,font=\small] {$f_2$};
				
				%%%%% %%%%% %%%%%
				
				%%%%% Right Center %%%%%
				
				%%%%% %%%%% %%%%%
				
				%%%%% %%%%% %%%%%

				%%%%% Edges %%%%%
				
				%%%%% Left Center %%%%%

				%%%%% %%%%% %%%%%
				
				%%%%% Center %%%%%
				
				\draw (t2) [e:main,->,bend left=10] to (f2);
				\draw (f2) [e:main,->,bend left=10] to (t2);
				\draw (t2) [e:main,->] to (x1);
				\draw (f2) [e:main,->] to (x1);
				\draw (x1) [e:main,->] to (x1B);
				\draw (x1B) [e:main,->] to (t2);
				\draw (x1B) [e:main,->] to (f2);

				%%%%% %%%%% %%%%%
				
				%%%%% Right Center %%%%%

				%%%%% %%%%% %%%%%
				
				%%%%% %%%%% %%%%%
				
				\end{pgfonlayer}
				
				%%%%% %%%%% %%%%%

				%%%%% Background %%%%%
				\begin{pgfonlayer}{background}
				
				\end{pgfonlayer}	
				%%%%% %%%%% %%%%%
				
				%%%%% Foreground %%%%%
				\begin{pgfonlayer}{foreground}

				\end{pgfonlayer}
				%%%%% %%%%% %%%%%
				\end{tikzpicture}
			\end{center}
			\subcaption{Variable gadget.}
			\label{subfig:variable gadget}
		\end{subfigure}
		~
		\begin{subfigure}[t]{0.3\textwidth}
			\begin{center}
			\begin{tikzpicture}[xscale=-0.7,yscale=0.7]
				
				\pgfdeclarelayer{background}
				\pgfdeclarelayer{foreground}
				
				\pgfsetlayers{background,main,foreground}
				
				%%%%% Vertex Styles %%%%%
%				\tikzstyle{v:main} = [draw, circle, scale=0.5, thick,fill=black]
%				\tikzstyle{v:tree} = [draw, circle, scale=0.3, thick,fill=black]
%				\tikzstyle{v:border} = [draw, circle, scale=0.75, thick,minimum size=10.5mm]
%				\tikzstyle{v:mainfull} = [draw, circle, scale=1, thick,fill]
%				\tikzstyle{v:ghost} = [inner sep=0pt,scale=1]
%				\tikzstyle{v:marked} = [circle, scale=1.2, fill=CornflowerBlue,opacity=0.3]
%				%%%%% %%%%% %%%%%
%				
%				%%%%% Edge Styles %%%%%
%				\tikzset{>=latex} 
%				\tikzstyle{e:marker} = [line width=9pt,line cap=round,opacity=0.2,color=DarkGoldenrod]
%				\tikzstyle{e:colored} = [line width=1.2pt,color=BostonUniversityRed,cap=round,opacity=0.8]
%				\tikzstyle{e:coloredthin} = [line width=1.1pt,opacity=0.8]
%				\tikzstyle{e:coloredborder} = [line width=2pt]
%				\tikzstyle{e:main} = [line width=1pt]
%				\tikzstyle{e:extra} = [line width=1.3pt,color=LavenderGray]
				%%%%% %%%%% %%%%%
				
				\begin{pgfonlayer}{main}
				
				%%%%% Centered Ghost Vertices %%%%%
				\node (C) [] {};
				
				%%%%% Left Center %%%%%
				\node (C1) [v:ghost, position=180:25mm from C] {};
				%\node (U1) [v:ghost, position=90:50mm from C1] {};
				%		\node (L1) [v:ghost, position=270:27mm from C1,align=center] {};
				
				%%%%% Center %%%%%
				\node (C2) [v:ghost, position=0:0mm from C] {};
				%\node (U2) [v:ghost, position=90:50mm from C2] {};
				%		\node (L2) [v:ghost, position=270:27mm from C2,align=center] {};
				
				%%%%% Right Center %%%%%
				\node (C3) [v:ghost, position=0:25mm from C] {};
				%\node (U3) [v:ghost, position=90:50mm from C3] {};
				%		\node (L3) [v:ghost, position=270:27mm from C3,align=center] {};
				%%%%% %%%%% %%%%%

				%%%%% Vertices %%%%%
				
				%%%%% Left Center %%%%%

				%%%%% %%%%% %%%%%
				
				%%%%% Center %%%%%
				
				\node (up) [v:ghost,position=90:9mm from C2] {};
				\node (x12) [v:main,position=270:9mm from C2] {};
				\node (x22B) [v:main,position=180:14mm from x12] {};
				\node (c2) [v:main,position=0:14mm from x12] {};
				\node (f3) [v:main,position=180:14mm from up,fill=BostonUniversityRed] {};
				\node (t3) [v:main,position=0:14mm from up,fill=myGreen] {};
				
				\node (Lx12) [v:ghost,position=270:4.5mm from x12,font=\small] {$x_{1,2}$};
				\node (Lx22B) [v:ghost,position=225:5.5mm from x22B,font=\small] {$\overline{x}_{2,2}$};
				\node (Lc2) [v:ghost,position=315:5.5mm from c2,font=\small] {$c_2$};
				\node (Lf3) [v:ghost,position=135:5.5mm from f3,font=\small] {$f_3$};
				\node (Lt3) [v:ghost,position=45:5.5mm from t3,font=\small] {$t_3$};
				
				%%%%% %%%%% %%%%%
				
				%%%%% Right Center %%%%%
				
				%%%%% %%%%% %%%%%
				
				%%%%% %%%%% %%%%%

				%%%%% Edges %%%%%
				
				%%%%% Left Center %%%%%

				%%%%% %%%%% %%%%%
				
				%%%%% Center %%%%%
				
				\draw (t3) [e:main,->,bend left=10] to (f3);
				\draw (f3) [e:main,->,bend left=10] to (t3);
				\draw (t3) [e:main,->,bend left=10] to (c2);
				\draw (c2) [e:main,->,bend left=10] to (t3);
				\draw (f3) [e:main,->] to (c2);
				\draw (c2) [e:main,->] to (x12);
				\draw (x12) [e:main,->] to (x22B);
				\draw (x22B) [e:main,->] to (f3);

				%%%%% %%%%% %%%%%
				
				%%%%% Right Center %%%%%

				%%%%% %%%%% %%%%%
				
				%%%%% %%%%% %%%%%
				
				\end{pgfonlayer}
				
				%%%%% %%%%% %%%%%

				%%%%% Background %%%%%
				\begin{pgfonlayer}{background}
				
				\end{pgfonlayer}	
				%%%%% %%%%% %%%%%
				
				%%%%% Foreground %%%%%
				\begin{pgfonlayer}{foreground}

				\end{pgfonlayer}
				%%%%% %%%%% %%%%%
				\end{tikzpicture}
			\end{center}
				\subcaption{Clause gadget.}
				\label{subfig:clause gadget}
		\end{subfigure}
		\caption{Variable, literal and clause gadgets of the proof of \cref{lemma:np-hard 2 fvs} for the variable $X_1$ and the clause $(X_1 \lor \overline{X}_2)$ in the \SAT\ formula $(X_1 \lor X_2) \land (X_1 \lor \overline{X_2}) \land (\overline{X_1} \lor X_2)$.}
	\end{figure}
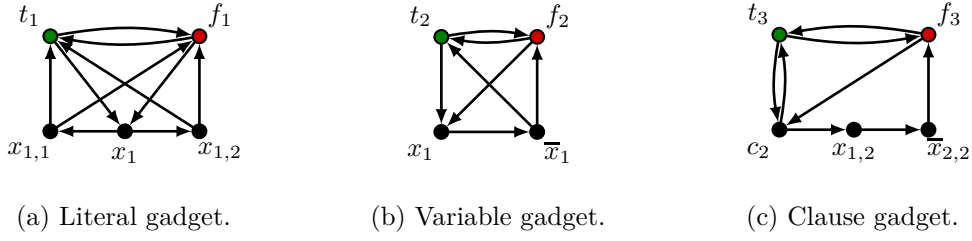
	
	We now construct a \emph{clause gadget} (illustrated in \cref{subfig:clause gadget}) that ensures that each clause is satisfied by at least one of its literals.
	We first add the cycle $t_3, f_3$.
	Then, for each clause $C_i$ we add the cycle $c_i, t_3$.
	Finally, we add the cycle $c_i, l_{j_1,i}, l_{j_2,i}, \dots, l_{j_h,i}, f_3$, where $l_{j_1,i}, l_{j_2,i}, \dots, l_{j_h,i}$ are the literals of $C_i$.
	We sort the literals in such a way that $j_1<j_2<\dots<j_h$ and such that $X_j$ comes before $\overline{X}_j$.
	This concludes the construction of $D$.
	
	We first show that $\VCNum{D} \leq 6$.
	We claim that the set $S = \{t_1, t_2, t_3, f_1, f_2, f_3\}$ is a feedback vertex set of $D$.
	We prove that $D - S$ is acyclic by finding a topological ordering of its vertices.
	We first take the positive literal vertices $x_j$ and the clause vertices $c_i$ into the ordering, as these are sources in $D-S$.
	Removing these vertices, all negative literal vertices $\overline{x}_j$ become sources, which we then add to the end of the current topological ordering.
	The only remaining vertices are the variable vertices $l_{j,i}$.
	It follows from the construction of the clause gadget that ordering the $l_{j,i}$ monotonically in $j$, with positive literals preceding corresponding negative literals, completes the topological ordering of $D - S$.
	
	To show that the degeneracy of $D$ is 3, we construct a linear ordering of the vertices as follows.	
	The first vertices of the ordering are $t_1, f_1, t_2, f_2, t_3$ and $f_3$.
	These have at most one outgoing arc to vertices which are smaller.
	Afterwards come all positive literal vertices $x_j$, then all negative literal vertices $\overline{x}_j$, followed by the variable vertices $l_{j,i}$.
	The vertices $\overline{x}_j$ have arcs to $t_2$ and $f_2$, and $x_j$ has no arc to smaller vertices.
	Hence, they have at most two arcs to smaller vertices.
	The vertices $l_{j,i}$ have arcs to $t_1$, $f_1$ and potentially to some other $l_{h,i}$ or to $f_3$, but never both.
	Hence, they have at most 3 arcs to smaller vertices.
	The last vertices in the ordering are the clause vertices $c_i$.
	These have an arc to $t_3$ and another to some $l_{j,i}$.
	Hence, the directed degeneracy of $D$ is at most $3$.

	We now prove that $D$ is 2-colourable if there is a truth assignment of the variables satisfying all clauses.
	
	Let $\beta:\{X_j \mid j \in [n]\} \rightarrow \{0,1\}$ be a satisfying truth assignment of the variables.
	We construct a colouring $c\colon\Fkt{V}{D} \rightarrow \Set{0,1}$ as follows.
	\begin{enumerate}
		\item $\Fkt{c}{f_i} \coloneqq 0$ and $\Fkt{c}{t_i} \coloneqq 1$ for $i \in [3]$.
		\item $\Fkt{c}{c_i} \coloneqq 0$ for $i \in [m]$.
		\item $\Fkt{c}{x_{j,i}} \coloneqq \Fkt{\beta}{X_j}$ for all $j \in [n]$ and $i \in [m]$ with $X_j \in C_i$.
		\item $\Fkt{c}{\overline{x}_{j,i}} \coloneqq 1 - \Fkt{\beta}{X_j}$ for all $j \in [n]$ and $i \in [m]$ with $\overline{X}_j \in C_i$.
		\item $\Fkt{c}{x_j} \coloneqq 1 - \Fkt{\beta}{X_j}$ and $\Fkt{c}{\overline{x}_j} \coloneqq \Fkt{\beta}{X_j}$ for all $j \in [n]$.
	\end{enumerate}
	This concludes the construction of $c$.
	We now argue that each colour class induces an acyclic digraph in $D$.
	
	Let $d \in \Set{0,1}$ be some colour.
	Note that either $t_1,t_2,t_3 \in \Fkt{c^{-1}}{d}$ or $f_1,f_2,f_3 \in \Fkt{c^{-1}}{d}$, as these vertices receive different colours.
	%Let $S_d = \Fkt{c^{-1}}{d} \cap S$.
	Since $S$ is a feedback vertex set of $D$, it suffices to show that there are no cycles using vertices of $S_d:=\Fkt{c^{-1}}{d} \cap S$ in $\InducedSubgraph{D}{\Fkt{c^{-1}}{d}}$.
	
	Assume, without loss of generality, that $t_1, t_2 \in \Fkt{c^{-1}}{d}$.
	The case $f_1, f_2 \in \Fkt{c^{-1}}{d}$ follows analogously.
	We prove that no cycle contains $t_1$ or $t_2$ by progressively identifying and removing sinks from $\InducedSubgraph{D}{c^{-1}(d)}$.
	As for all $j \in [n]$ and $i \in [m]$ we have $c(x_j) \neq c(\overline{x}_j) = c(x_{j,i})$, it follows that all $x_j$ are sinks in $\InducedSubgraph{D}{c^{-1}(d)}$.
	Removing all $x_j$, we can see that $t_2$ is now a sink.
	Hence, no directed cycle in $D[c^{-1}(d)]$ contains $t_2$.
	As $c(\overline{x}_j) \neq c(\overline{x}_{j,i})$, it follows that $\overline{x}_j$ is now a sink and we can remove it.
	Without literal vertices, $t_1$ becomes a sink, implying no cycle goes through $t_1$ in $\InducedSubgraph{D}{c^{-1}(d)}$, as desired.
	Consequently, for any $d \in \Set{0,1}$, no directed cycle in $D[c^{-1}(d)]$ can possibly use one of the vertices $t_1,t_2,f_1,f_2$ and therefore must either contain $t_3$ or $f_3$.
	
	If $t_3 \in \Fkt{c^{-1}}{d}$, then $c_i \not\in \Fkt{c^{-1}}{d}$ for all $i \in [m]$, as $\Fkt{c}{t_3} = 1$ and $\Fkt{c}{c_i} = 0$.
	Hence, $t_3$ has no neighbours in $\InducedSubgraph{D}{\Fkt{c^{-1}}{d}}$ and cannot be in any cycle.
	If $f_3 \in \Fkt{c^{-1}}{d}$, assume towards a contradiction that there is a cycle $C$ in $\InducedSubgraph{D}{\Fkt{c^{-1}}{d}}$ containing $f_3$.
	Note that this cycle must also contain $c_i$ for some $i \in [m]$, as these are the only out-neighbours of $f_3$ in $\InducedSubgraph{D}{\Fkt{c^{-1}}{d}}$.
	Furthermore, the out-neighbour of $c_i$ in $C$ is some $l_{j,i}$, and the only out-neighbours of $l_{j,i}$ are $t_1$ and potentially some $l_{h,i}$ or $f_3$, as these were the arcs added in the clause gadgets.
	The vertices $l_{j,i}$ in $C$ correspond to the literals in $c_i$.
	In order to form a cycle, all literals in $c_i$ must be in $C$.
	However, this means that $\Fkt{c}{x_{j,i}} = 0$ for all $X_j$ in clause $C_i$ and $\Fkt{c}{\overline{x}_{j,i}} = 0$ for all $\overline{X}_j$ in clause $C_i$.
	By construction of $c$, this implies that all literals in $C_i$ are set to false, which means that the clause is not satisfied, a contradiction to our initial assumption.
	Hence, the digraph $\InducedSubgraph{D}{\Fkt{c^{-1}}{d}}$ is acyclic, and $D$ is 2-colourable.
	
	We now show that the formula is satisfiable if $\Dichromatic{D} \leq 2$ by constructing a satisfying variable assignment $\beta$ from a proper 2-colouring of $D$.
	Let $c\colon\Fkt{V}{D}\rightarrow \Set{0,1}$ be a proper colouring of $D$.
	Without loss of generality, we assume that $\Fkt{c}{t_3} = 1$, which implies that $\Fkt{c}{f_3} = 0$.
	We set $\Fkt{\beta}{X_j}$ to true if $\Fkt{c}{x_j} = 0$ and to false if $\Fkt{c}{x_j} = 1$.
	
	Assume towards a contradiction that there is some clause $C_i$ which is not satisfied by $\beta$.
	By simply renaming the variables, we can assume without loss of generality that the literals of $C_i$ are $L_1, L_2, \dots, L_a$.
	As $C_i$ is not satisfied, it follows that all $L_j$ evaluate to false with $\beta$.
	By construction of the literal gadget, $\Fkt{c}{l_j} \neq \Fkt{c}{l_{j,i}}$ for all $i \in [m]$ with $L_j \in C_i$.
	From \ref{item:lji = ljh} and \ref{item:nxji != xji}, for all $j \in [n]$ it follows that $\Fkt{c}{l_{j,i}} = 1$ if the literal $L_j$ is true, and that $\Fkt{c}{l_{j,i}} = 0$ if the literal $L_j$ is false.
	As $C_i$ is not satisfied, $\Fkt{c}{c_i} = \Fkt{c}{f_3} = \Fkt{c}{l_{j,i}} = 0$ for all $j \in [a]$.
	Hence, the cycle $C = c_i, l_1, l_2, \dots l_a, f_3$ is monochromatic, contradicting our assumption that $c$ is a proper colouring.
	This implies that $\beta$ is a satisfying variable assignment, concluding our proof.
\end{proof}

With a simple self-reduction, we can extend the previous result to all $k \geq 2$.

\begin{theorem}\label{thm:np-hard k}
	For each $k \geq 2$, \DigraphColouringProblem[$k$] is \NP-hard even if $\VCNum{D} \leq k+4$ and $\Degeneracy{D} \leq k+1$, where $D$ is the input digraph.
\end{theorem}
\begin{proof}
	We prove the statement by induction on $k$.
	The case $k=2$ follows from \cref{lemma:np-hard 2 fvs}.
	We provide a reduction from \DigraphColouringProblem[$(k-1)$] to \DigraphColouringProblem[$k$] such that $\VCNum{D'} \leq \VCNum{D} + 1$ and $\Degeneracy{D'} \leq \Degeneracy{D} + 1$, where $D$ is the input instance and $D'$ is the reduced instance.
	We obtain $D'$ be adding a vertex $x$ to $D$, together with the edges $\CondSet{\Brace{x,v}, \Brace{v,x}}{ x \in \Fkt{V}{D}}$.
	If $D$ is $(k-1)$-colourable, then setting the colour of $x$ to $k$ gives a proper $k$-colouring for $D'$.
	If $D'$ is $k$-colourable, then no vertex in $D$ has the same colour as $x$.
	Hence, $D$ is $(k-1)$-colourable.
	Furthermore, all new cycles created by adding $x$ go through $x$.
	If $D - S$ is acyclic for some vertex set $S$, then $D' - \Brace{S \cup \Set{x}} = D - S$ is also acyclic.
	Hence, $\VCNum{D'} \leq \VCNum{D} + 1 = k+4$.
	To show that the degeneracy of $D'$ increased by at most one, we consider some ordering of $D$ with degeneracy $\Degeneracy{D} = k$.
	By placing $v$ as the smallest vertex with respect to the ordering, we increase the outdegree of the vertices in $D$ by one.
	Hence, the degeneracy of $D'$ is at most $\Degeneracy{D}+1 = k+1$, as desired.
\end{proof}

As an immediate consequence of the above theorem arises the following corollary.

\begin{corollary}
	There is no $n^{f(k,x,y)}$-time algorithm deciding \DigraphColouringProblem[$k$] where $x = \VCNum{D}$, $y = \Degeneracy{D}$ and $f$ is some function, unless \Poly$=$\NP.
\end{corollary}

A finer analysis of the reduction provided in \cref{lemma:np-hard 2 fvs} gives us stronger hardness results under a different assumption.
Similar to how no polynomial-time algorithms for \NP-complete problems are known, no $2^{\SmallO{n}}n^{\BigO{1}}$-time algorithm for \SAT[$k$] is known, where $n$ is the number of variables in the input formula (which contains at most $k$ literals in each clause).
An algorithm with such a running time is called a \emph{subexponential-time} algorithm.
Impagliazzo and Paturi \cite{impagliazzo2001complexity} provided evidence that no such algorithm for \SAT[$k$] exists, and formulated the following hypothesis (often referred to as \emph{ETH}).
\begin{hypothesis*}[Exponential Time Hypothesis \cite{impagliazzo2001complexity}]
	For each $k \geq 3$ there is some $s_k > 0$ such that no $2^{s_k n}n^{\BigO{1}}$-time algorithm for \SAT[$k$] exists.
\end{hypothesis*}

Note that the ETH only considers the running time with respect to the number of variables in the input formula, not the number of clauses.
In several reductions, however, it is difficult to ensure that the size of the reduced instance depends only on the number of variables.
For example, the reduction in \cref{lemma:np-hard 2 fvs} contains one vertex for each clause.
This would prevent us from directly applying the ETH.
Fortunately, \cite{impagliazzo2001problems} showed that it is possible to assume that $m \in \BigO{n}$, where $m$ is the number of clauses, by proving the following lemma.
\begin{lemma*}[Sparsification Lemma, Impagliazzo, Paturi and Zane \cite{impagliazzo2001problems}]
	For all $\epsilon > 0$ and $k > 0$ there is a constant $C$ so that any \SAT[$k$] formula $\Phi$ with $n$ variables can be expressed as $\Phi' = \bigvee_{i=1}^t\Psi_i$, where $t \leq 2^{\epsilon n}$ and each $\Psi_i$ is a \SAT[$k$] formula with at most $Cn$ clauses such that each variable appears in constantly many clauses.
	Moreover, this disjunction can be computed by an algorithm running in time $2^{\epsilon n}n^{\BigO{1}}$.
\end{lemma*}
By first applying the sparsification lemma to the input formula and then the reduction from \cref{thm:np-hard k}, we can show the following.
\begin{theorem}
	For each $k \geq 2$ there is some $\epsilon > 0$ such that no $2^{\epsilon n} n^{f(x,y)}$ algorithm for \DigraphColouringProblem[$k$] exists, where $D$ is the input digraph, $x = \VCNum{D}$, $y = \Degeneracy{D}$ and $f$ is some function, unless the ETH is false.
\end{theorem}
\begin{proof}
	First note that the reduction from \DigraphColouringProblem[$(k-1)$] to \DigraphColouringProblem[$k)$] from \cref{thm:np-hard k} increases the input instance by one vertex.
	Hence, it suffices to show the statement for $k=2$, as the remaining cases follow by induction.
	We first use the sparsification lemma to obtain at most $2^{\epsilon n}$ many \SAT[$3$] instances where each variable appears in constantly many clauses.
	Applying the reduction from \cref{lemma:np-hard 2 fvs} to each instance, we obtain at most $2^{\epsilon n}$ many digraphs where for each variable we have constantly many vertices and for each clause we have one vertex.
	This means that the number of vertices on the reduced instances is linear in the number of variables of the formula.
	Hence, a subexponential-time algorithm for \DigraphColouringProblem[$2$] implies a subexponential-time algorithm for \SAT[$3$], which would contradict the ETH.
\end{proof}

Note that an algorithm with running time $\BigO{k^n\cdot(n+m})$ is trivial: test all $k^n$ colourings of the vertices of $D$, and then check if each colour class is a DAG in linear time by computing a topological ordering.

\section{Polychromatic Colourings and Cycle Packings of Strongly Planar Digraphs} \label{sec:polychromatic}

In this section, we study colouring properties of so-called \emph{strongly planar digraphs}.
These form a canonical class of planar non-even digraphs (however, there are many others).
To motivate their definition, consider an arbitrary bipartite, matching-covered planar graph $G$ with bipartition $\Brace{A,B}$.
Because $G$ is planar, it must be Pfaffian.
Choose some perfect matching $M$ of $G$.
Considering the orientation $\vec{G}$ of $G$ orienting all edges from $A$ to $B$, we can view $\DirM{G}{M}$ as being obtained from $\vec{G}$ by contraction of all edges in $M$.
It is now clear that the digraph $\DirM{G}{M}$ inherits a natural plane-embedding from $G$ in which for each vertex, the incident incoming and outgoing edges are separated into two intervals in the cyclic ordering.
It is not hard to reverse the described relationship to see that any digraph $D$ admitting such an embedding is isomorphic to $\DirM{G}{M}$ for some planar bipartite graph and a perfect matching $M$.

\begin{definition}
	A digraph $D$ is called \emph{strongly planar} if there is a simple, non-crossing topological plane-embedding of $D$ such that for each $x \in \Fkt{V}{D}$ the incoming (resp. outgoing) edges incident to $x$ form a consecutive interval in the cyclic ordering around $x$.
	Equivalently, $D \cong \mathcal{D}(G,M)$ for a planar bipartite graph $G$ and a perfect matching $M$.
\end{definition}

An example of a strongly planar digraph is given in \cref{fig:GridDiGraph}.

\begin{figure}[h!]
	\begin{center}
		\begin{tikzpicture}[scale=0.7]
		
		\pgfdeclarelayer{background}
		\pgfdeclarelayer{foreground}
		
		\pgfsetlayers{background,main,foreground}
		
		%%%%% Vertex Styles %%%%%
%		\tikzstyle{v:main} = [draw, circle, scale=0.5, thick,fill=black]
%		\tikzstyle{v:tree} = [draw, circle, scale=0.3, thick,fill=black]
%		\tikzstyle{v:border} = [draw, circle, scale=0.75, thick,minimum size=10.5mm]
%		\tikzstyle{v:mainfull} = [draw, circle, scale=1, thick,fill]
%		\tikzstyle{v:ghost} = [inner sep=0pt,scale=1]
%		\tikzstyle{v:marked} = [circle, scale=1.2, fill=CornflowerBlue,opacity=0.3]
%		%%%%% %%%%% %%%%%
%		
%		%%%%% Edge Styles %%%%%
%		\tikzset{>=latex} 
%		\tikzstyle{e:marker} = [line width=9pt,line cap=round,opacity=0.2,color=DarkGoldenrod]
%		\tikzstyle{e:colored} = [line width=1.2pt,color=BostonUniversityRed,cap=round,opacity=0.8]
%		\tikzstyle{e:coloredthin} = [line width=0.9pt,color=myGreen]
%		\tikzstyle{e:coloredborder} = [line width=2pt]
%		\tikzstyle{e:main} = [line width=1pt]
%		\tikzstyle{e:extra} = [line width=1.3pt,color=LavenderGray]
		%%%%% %%%%% %%%%%
		
		\begin{pgfonlayer}{main}
		
		%%%%% Centered Ghost Vertices %%%%%
		\node (C) [] {};
		
		%%%%% Left Center %%%%%
		\node (C1) [v:ghost, position=180:40mm from C] {};
		%\node (U1) [v:ghost, position=90:50mm from C1] {};
%		\node (L1) [v:ghost, position=270:32mm from C1,align=center] {$R$};
		
		%%%%% Center %%%%%
		\node (C2) [v:ghost, position=0:0mm from C] {};
		%\node (U2) [v:ghost, position=90:50mm from C2] {};
		%		\node (L2) [v:ghost, position=270:27mm from C2,align=center] {};
		
		%%%%% Right Center %%%%%
		\node (C3) [v:ghost, position=0:40mm from C] {};
		%\node (U3) [v:ghost, position=90:50mm from C3] {};
%		\node (L3) [v:ghost, position=270:32mm from C3,align=center] {$\Bidirected{C_5}$};
		%%%%% %%%%% %%%%%

		%%%%% Vertices %%%%%
		
		%%%%% Left Center %%%%%
		
		\node (a) [v:main,position=0:0mm from C1] {};
		\node (ain) [v:ghost,position=0:2.25mm from a] {};
		
		\node (b) [v:main,position=0:13mm from a,fill=white] {};
		\node (e) [v:main,position=270:13mm from a,fill=white] {};
		
		\node (c) [v:main,position=0:13mm from b] {};
		\node (cin) [v:ghost,position=270:2.25mm from c] {};
		\node (f) [v:main,position=270:13mm from b] {};
		\node (fin) [v:ghost,position=270:2.25mm from f] {};
		\node (i) [v:main,position=270:13mm from e] {};
		\node (iin) [v:ghost,position=90:2.25mm from i] {};
		
		\node (d) [v:main,position=0:13mm from c,fill=white] {};
		\node (g) [v:main,position=270:13mm from c,fill=white] {};
		\node (j) [v:main,position=0:13mm from i,fill=white] {};
		\node (m) [v:main,position=270:13mm from i,fill=white] {};
		
		\node (h) [v:main,position=0:13mm from g] {};
		\node (hin) [v:ghost,position=90:2.25mm from h] {};
		\node (k) [v:main,position=270:13mm from g] {};
		\node (kin) [v:ghost,position=0:2.25mm from k] {};
		\node (n) [v:main,position=270:13mm from j] {};
		\node (nin) [v:ghost,position=180:2.25mm from n] {};
		
		\node (l) [v:main,position=0:13mm from k,fill=white] {};
		\node (o) [v:main,position=270:13mm from k,fill=white] {};
		
		\node (p) [v:main,position=0:13mm from o] {};
		\node (pin) [v:ghost,position=180:2.25mm from p] {};
		
		%%%%% %%%%% %%%%%
		
		%%%%% Center %%%%%

		%%%%% %%%%% %%%%%
		
		%%%%% Right Center %%%%%
		
		\node (ab) [v:main,fill=myGreen,position=270:4mm from C3] {};
		\node (ei) [v:main,fill=myGreen,position=247.5:16mm from ab] {};
		\node (fj) [v:main,fill=myGreen,position=292.5:16mm from ab] {};
		\node (mn) [v:main,fill=myGreen,position=270:32mm from ab] {};
		\node (cg) [v:main,fill=myGreen,position=45:15mm from fj] {};
		\node (dh) [v:main,fill=myGreen,position=0:15mm from cg] {};
		\node (kl) [v:main,fill=myGreen,position=300:15mm from cg] {};
		\node (op) [v:main,fill=myGreen,position=270:14.5mm from kl] {};
		
		%%%%% %%%%% %%%%%
		
		%%%%% %%%%% %%%%%

		%%%%% Edges %%%%%
		
		%%%%% Left Center %%%%%
		
		\draw (b) [e:main,->] to (c);
		\draw (b) [e:main,->] to (f);
		
		\draw (e) [e:main,->] to (a);
		\draw (e) [e:main,->] to (f);
		
		\draw (d) [e:main,->] to (c);
		
		\draw (g) [e:main,->] to (f);
		\draw (g) [e:main,->] to (h);
		\draw (g) [e:main,->] to (k);
		
		\draw (j) [e:main,->] to (i);
		\draw (j) [e:main,->] to (k);
		\draw (j) [e:main,->] to (n);
		
		\draw (m) [e:main,->] to (i);
		
		\draw (l) [e:main,->] to (h);
		\draw (l) [e:main,->] to (p);
		
		\draw (o) [e:main,->] to (k);
		\draw (o) [e:main,->] to (n);

		%%%%% %%%%% %%%%%
		
		%%%%% Center %%%%%
		
		%%%%% %%%%% %%%%%
		
		%%%%% Right Center %%%%%
		
		\draw (ab) [e:main,->] to (fj);
		\draw (ab) [e:main,->] to (cg);
		
		\draw (ei) [e:main,->,bend left=15] to (fj);
		\draw (ei) [e:main,->] to (ab);
		
		\draw (fj) [e:main,->,bend left=15] to (ei);
		\draw (fj) [e:main,->] to (mn);
		\draw (fj) [e:main,->] to (kl);
		
		\draw (mn) [e:main,->] to (ei);
		
		\draw (cg) [e:main,->,bend right=15] to (dh);
		\draw (cg) [e:main,->] to (fj);
		\draw (cg) [e:main,->] to (kl);
		
		\draw (dh) [e:main,->,bend right=15] to (cg);
		
		\draw (kl) [e:main,->] to (dh);
		\draw (kl) [e:main,->,bend left=15] to (op);
		
		\draw (op) [e:main,->] to (mn);
		\draw (op) [e:main,->,bend left=15] to (kl);
		
		%%%%% %%%%% %%%%%
		
		%%%%% %%%%% %%%%%
		
		\end{pgfonlayer}
		
		%%%%% %%%%% %%%%%

		%%%%% Background %%%%%
		\begin{pgfonlayer}{background}
		
		\draw (b) [e:matching1,->] to (a);
		\draw (e) [e:matching1,->] to (i);
		\draw (d) [e:matching1,->] to (h);
		\draw (g) [e:matching1,->] to (c);
		\draw (j) [e:matching1,->] to (f);
		\draw (m) [e:matching1,->] to (n);
		\draw (l) [e:matching1,->] to (k);
		\draw (o) [e:matching1,->] to (p);
		
%		\draw (b) [e:coloredthin,->] to (ain);
%		\draw (e) [e:coloredthin,->] to (iin);
%		\draw (d) [e:coloredthin,->] to (hin);
%		\draw (g) [e:coloredthin,->] to (cin);
%		\draw (j) [e:coloredthin,->] to (fin);
%		\draw (m) [e:coloredthin,->] to (nin);
%		\draw (l) [e:coloredthin,->] to (kin);
%		\draw (o) [e:coloredthin,->] to (pin);
		
		\end{pgfonlayer}	
		%%%%% %%%%% %%%%%
		
		%%%%% Foreground %%%%%
		\begin{pgfonlayer}{foreground}

		\end{pgfonlayer}
		%%%%% %%%%% %%%%%
		\end{tikzpicture}
	\end{center}
	\caption{Left: An oriented grid equipped with a perfect matching. Right: The arising $M$-direction, a strongly planar digraph.}
	\label{fig:GridDiGraph}
\end{figure}
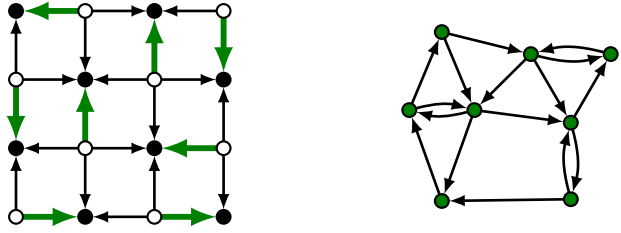
By \cref{thm:pfaffian}, every strongly planar digraph is non-even and so, according to \cref{thm:mainthm}, it is $2$-colourable.

In this section, we seek a strengthening of $2$-colourability for strongly planar digraphs of large girth.
While $\Dichromatic{D}\leq 2$ for all strongly planar digraphs can be rephrased as the existence of a packing of two disjoint feedback vertex sets in any strongly planar digraph, we conjecture the following generalisation.

\begin{conjecture} \label{conj:highergirth}
	For any strongly planar digraph $D$ of girth $g$, there exists a packing of $g$ pairwise disjoint feedback vertex sets.
	In other words, $D$ can be vertex $g$-coloured such that every directed cycle uses each colour at least once.
\end{conjecture}

Clearly, the directed cycle $\vec{C}_g$ of length $g$ admits a packing of $g$ and no more disjoint feedback vertex sets, and consequently, this conjecture, if true, is best-possible. 

For an arbitrary bipartite planar graph $G$ with a perfect matching $M$, a feedback vertex set in $\mathcal{D}(G,M)$ corresponds to a partial matching $S \subseteq M$ with the property that every $M$-alternating cycle uses an edge in $S$, which is the same as saying that $S$ is forcing. Consequently, in the language of perfect matchings, the above translates to:

\begin{conjecture}\label{conj:highergirthmatching}
	Let $G$ be a bipartite planar graph with a perfect matching $M$ and let $2g$ be the length of a shortest $M$-alternating cycle.
	Then $M$ can be decomposed into $g$ pairwise disjoint forcing sets. 
\end{conjecture}

The type of colouring as described for cycles in digraphs was investigated more generally for hypergraphs by Bollob\'{a}s et al. (\cite{bollobas}).
Given a hypergraph $H$, a \emph{polychromatic $k$-colouring} of $\mathcal{H}$ is defined to be a vertex-colouring $c\colon \Fkt{V}{H} \rightarrow \Set{0,\ldots,k-1}$ such that every hyperedge $e \in \Fkt{E}{H}$ contains at least one vertex of each colour.
The \emph{polychromatic number} of $H$ then is defined as the maximal $k$ for which a polychromatic $k$-colouring of $H$ exists.
Clearly, the polychromatic number of a hypergraph $H$ is upper bounded by its \emph{rank}, that is, the size of a smallest hyperedge. 

Given a digraph $D$, we may associate with it the \emph{cycle hypergraph} $\CycleHypergraph{D}$ having $\Fkt{V}{D}$ as vertex set and containing the vertex sets of all directed cycles in $D$ as hyperedges.
It is now clear that \cref{conj:highergirth} claims that the cycle hypergraph $\CycleHypergraph{D}$ of any strongly planar digraph $D$ has the very special property that the polychromatic number matches its rank.

To the best of our knowledge, polychromatic colourings of digraphs in the above sense have not been investigated before, and we hope that this conjecture might initiate research in this direction.
Looking at general planar digraphs, for any $g \ge 2$, there are examples of planar digraphs with girth $g$ which do not admit a packing of $g$ disjoint feedback vertex sets (cf. \cite{hochsteiner}).
However, the following statement, which contains the 2-Colour-Conjecture (\Cref{con:twocolours}) as the subcase $g=3$, might still be true.

\begin{conjecture}[Hochstättler and S. \cite{hochsteiner}]
	For any planar digraph of girth $g \geq 3$, there exists a packing of $g-1$ disjoint feedback vertex sets. 
\end{conjecture} 

The rest of this section is devoted to partial results towards \cref{conj:highergirth} using the concept of fractional colourings.

Given a fixed natural number $b \geq 1$ and some $k \in \mathbb{N}, k \geq b$, a \emph{$b$-tuple $k$-colouring} of a digraph $D$ is defined to be an assignment of subsets of $\Set{0,\ldots,k-1}$ of size $b$ to the vertices of $D$ in such a way that for any $i \in \Set{0,\ldots,k-1}$, the subdigraph of $D$ induced by those vertices whose colour-set contains $i$ is acyclic.
The \emph{$b$-dichromatic number} $\BDichromatic{D}$ of a digraph is then defined to be the least $k$ for which a $b$-tuple $k$-colouring of $D$ exists.
It is easy to see that $\BDichromatic{D} \leq b \cdot \Dichromatic{D}$ for any digraph.
Thus, the \emph{fractional dichromatic number} of a digraph defined as $\FracDichromatic{D}\coloneqq\inf_{b \geq 1}{\frac{\BDichromatic{D}}{b}} \in [1,\infty)$ is always a lower bound for the dichromatic number.

It has been proved in \cite{doct}, Chapter 5 that $\FracDichromatic{D}$ is always a rational number and can be alternatively represented as the optimal value of the following linear relaxation of a natural integer program formulation of the dichromatic number:

\begin{theorem}[Severino \cite{doct}] \label{doct}
	Let $D$ be a digraph.
	Then there is an integer $b \ge 1$ such that $\FracDichromatic{D}=\frac{\BDichromatic{D}}{b}$.
	Denote the collection of acyclic vertex sets in $D$ by $\Acyclic{D}$ and for any $v \in \Fkt{V}{D}$ let $\Acyclic{D,v} \subseteq \Acyclic{D}$ consist of only those acyclic sets containing $v$.
	Then $\FracDichromatic{D}$ is the optimal value of
	\begin{align} \label{primalfract}
	&\min \sum_{A \in \mathcal{A}(D)}{x_A}\\
	\text{  subj.\ to }&
	\sum_{A \in \mathcal{A}(D,v)}{x_A} \ge 1, \text{ for all } v \in V(D) \cr
	& x \ge 0.\nonumber
	\end{align}
\end{theorem}

The fractional dichromatic number has turned out to be a useful concept.
For instance, it was used in \cite{frdichr} to prove a fractional version of the so-called \emph{Erd\H{o}s-Neumann-Lara-Conjecture}.

To make the statement of our results clearer, we reformulate \cref{conj:highergirth} in the setting of \emph{circular colourings} of digraphs.
The \emph{star dichromatic number} $\StarDichromatic{D}$ of a digraph was recently introduced in \cite{hochsteiner} as a refined measure of the dichromatic number of a digraph which, similar to the circular or fractional chromatic number of a graph (cf. \cite{vince} and \cite{kneser}), can take on rational values.
Instead of a finite colour set, for any $p \in \mathbb{R}, p \geq 1$, in an \emph{acyclic $p$-colouring} of a digraph $D$, vertices are coloured with points on a plane circle $S_p$ with perimeter $p$ such that for any open cyclic subinterval $I \subseteq S_p$ of length $1$, the set of vertices mapped to this interval is acyclic.
The star dichromatic number $\StarDichromatic{D}$ is now defined as the minimal value of $p$ for which an acyclic $p$-colouring of $D$ exists. 

Intuitively, having fractional or star dichromatic number close to $1$ captures the property of a digraph being ``close'' to acyclic.

We restate the following basics.

\begin{proposition} [Hochstättler and S. \cite{hochsteiner}] \noindent
	Let $D$ be a digraph, then the following statements hold.
	\begin{enumerate}
		\item The star dichromatic number $\StarDichromatic{D}$ is a fraction with numerator at most $\Abs{\Fkt{V}{D}}$ satisfying $\lceil \StarDichromatic{D} \rceil=\Dichromatic{D}$. 
		\item For any pair of integers $k \ge d \ge 1$, we have $\StarDichromatic{D} \leq \frac{k}{d}$ if and only if there is a colouring $c\colon\Fkt{V}{D} \rightarrow \mathbb{Z}_k \simeq \{0,1,\ldots,k-1\}$ of the vertices of $D$ such that $\Fkt{c^{-1}}{\Set{i,i+1,\ldots,i+d-1}}$ (sums taken modulo $k$) is an acyclic vertex set for every $i \in \mathbb{Z}_k$.
		\item $\FracDichromatic{D} \leq \StarDichromatic{D}$.
		\item For every $n \in \mathbb{N}, n \ge 2$ we have $\FracDichromatic{\vec{C}_n}=\StarDichromatic{\vec{C}_n}=\frac{n}{n-1}$.
	\end{enumerate}
\end{proposition}

As a consequence of the second point, we obtain that, for a digraph $D$ of girth $g$, $\StarDichromatic{D} \leq \frac{g}{g-1}$ if and only if $\Fkt{V}{D}$ can be coloured with the elements of $\mathbb{Z}_g$ such that for any $i \in \mathbb{Z}_g$, the vertices mapped to $\mathbb{Z}_g \setminus \{i\}$ form an acyclic set.
However, this is the same as saying that $D$ can be vertex-coloured with $g$ colours such that each colour class is a feedback vertex set of $D$. Therefore, the following is an equivalent reformulation of \cref{conj:highergirth}.

\begin{conjecture}
	For any strongly planar digraph $D$ of girth $g \geq 2$, we have $\StarDichromatic{D}=\frac{g}{g-1}$.
\end{conjecture} 

For planar digraphs, the fractional and the star dichromatic number often coincide or are closely tied to each other.
Thus, the following result can be seen as a source of evidence for \cref{conj:highergirth}.

\begin{theorem}\label{thm:mainfractional}
	For any strongly planar digraph $D$ of girth $g \geq 2$, we have $\FracDichromatic{D}=\frac{g}{g-1}$. 
\end{theorem}

To prove this result, we use insights from the theory of so-called clutters.
A \emph{clutter} is defined to be a collection $\mathcal{C}$ of subsets of a finite ground set $S$ such that $C_1 \nsubseteq C_2$ for any $C_1 \neq C_2 \in \mathcal{C}$.
We refer to the first chapter of \cite{corn} for a short and comprehensible introduction to the topic.

Associated with any clutter $\mathcal{C}$ over the ground set $S$ we have a \emph{clutter matrix} $\ClutterMatrix{C}$ whose columns are indexed by the elements of $S$ and whose rows correspond to the characteristic vectors of the members of $\mathcal{C}$ with respect to $S$.
The following primal-dual pair of linear optimisation programs resembles natural covering and packing problems related to clutters.
Here, $w \geq 0$ denotes a row vector whose entries are non-negative real numbers or possibly $\infty$, and $\textbf{1}$ denotes the vector with all entries equal to $1$. Vector-inequalities are to be understood component-wise.

\begin{align}
\min\CondSet{wx}{ x \geq 0,~ \ClutterMatrix{C}x \geq \textbf{1}} \label{primal} \\
=\max\CondSet{y\textbf{1}}{y \geq 0,~ y\ClutterMatrix{C} \leq w} \label{dual} 
\end{align}

In the following, we introduce a number of important notions for clutters related to integral solutions of the linear programs (\ref{primal}) and (\ref{dual}).

Given a clutter $\mathcal{C}$, we will say that it admits the \emph{Max-Flow-Min-Cut-Property} (MFMC for short) if, for any non-negative $w$ with integral entries, there exists a primal-dual pair of integral optimal solutions to the linear programs (\ref{primal}) and (\ref{dual}).

We say that $\mathcal{C}$ \emph{packs} if the same holds true at least for $w=\textbf{1}$.
If such an integral primal-dual solution exists for all vectors $w$ with entries $0$, $1$ or $\infty$, we call the clutter \emph{packing}.

It is not hard to see that if a clutter has the MFMC-property, it is packing, and, clearly, any packing clutter also packs.
While there are examples of clutters that pack but do not have the packing property, it is a famous open problem due to Conforti and Cornuejols to show that in fact, the packing property and the MFMC-property are equivalent.

\begin{conjecture}[Conforti and Cornuejols]
	A clutter has the packing property if and only if it has the MFMC property.
\end{conjecture} 

For the following, we will furthermore need the notion of \emph{idealness} for clutters.
A clutter is said to be \emph{ideal} if, for any real-valued vector $w \ge 0$, the primal linear program (\ref{primal}) has an integral optimal solution vector $x$.
It is not hard to show that the MFMC-property implies idealness of a clutter.

A famous example of a clutter related to digraphs is the clutter of all minimal directed cuts of a fixed directed graph $D$.
The following well-known result of Lucchesi and Younger can be rephrased as the fact that the clutter of minimal directed cuts of any digraph has the MFMC-property.
To formulate the theorem, we need the following terminology: A \emph{dijoin} of a digraph $D$ is a subset of $E(D)$ intersecting every directed cut in at least one edge.

\begin{theorem}[Lucchesi and Younger \cite{lucchesiyounger}]
	Let $D$ be a digraph and $w\colon\Fkt{E}{D} \rightarrow \mathbb{N}_0$ a non-negative integral edge-weighting.
	Then the minimal weight of a dijoin in $D$ equals the maximal size of a collection of (minimal) directed cuts in $D$ so that any edge $e \in \Fkt{E}{D}$ is contained in at most $w(e)$ of them.
\end{theorem}

Using planar duality of digraphs, the above theorem restricted to planar digraphs reformulates as follows.

\begin{corollary}\label{duallucchesi}
	Let $D$ be a planar digraph and $w\colon\Fkt{E}{D} \rightarrow \mathbb{N}_0$ a non-negative integral edge-weighting.
	Then the minimal weight of a directed cycle in $D$ equals the maximal number of directed cycles containing any edge $e \in\Fkt{E}{D}$ at most $\Fkt{w}{e}$ times. 
\end{corollary}

We now use the above result to prove that given a strongly planar digraph $D$, the associated clutter containing the vertex sets of all induced directed cycles in $D$ admits the MFMC-property.
This result has already been observed for instance in \cite{mengeriandigraphs}, we provide its proof for completeness.

\begin{theorem}
	Let $D$ be strongly planar.
	Then for any non-negative integral vertex-weighting $w\colon\Fkt{V}{D} \rightarrow \mathbb{N}_0$, the minimal weight of a feedback vertex set in $D$ equals the maximal number of (induced) directed cycles in $D$ which together contain any vertex $x \in \Fkt{V}{D}$ at most $\Fkt{w}{x}$ times.
\end{theorem}

\begin{proof}
	We construct an auxiliary splitting-digraph $D'$ by replacing each vertex $x \in \Fkt{V}{D}$ by a directed edge $e_x \in \Fkt{E}{D'}$ in such a way that all the incoming edges incident to $x$ in $D$ are now incident to $\Tail{e_x}$ while all the outgoing edges of $x$ in $D$ are now emanating from $\Head{e_x}$.
	By contracting the edge $e_x$ for each $x \in \Fkt{V}{D}$, it is clear that the directed cycles in $D'$ are in one-to-one correspondence with the directed cycles of $D$.
	Moreover, the vertex-intersection of a pair of directed cycles in $D$ yields a subset of the edge-intersection of the corresponding directed cycles in $D'$.
	It is furthermore easy to see from the fact that the outgoing and incoming edges incident to any vertex in $D$ are separated in the cyclic ordering, that $D'$ indeed admits a planar embedding.
	We now define a corresponding weighting of the edges of $D'$ by setting $\Fkt{w'}{e_x}\coloneqq\Fkt{w}{x}$ for any $x \in \Fkt{V}{D}$ and $\Fkt{w'}{e}\coloneqq M$ for a large natural number $M \in \mathbb{N}$ for any other edge of $D'$.
	If we choose $M$ large enough, we find that the minimal edge-weight of a feedback edge set in $D'$ is exactly the minimal vertex-weight of a feedback vertex set in $D$.
	\Cref{duallucchesi} now tells us that the latter is the same as the maximal size of a collection of directed cycles in $D'$ in which any $e_x$ is contained at most $\Fkt{w}{x}$ times while any other edge is contained at most $M$ times.
	As the latter condition becomes redundant for $M$ large enough, this again is the same as the maximal size of a collection of directed cycles in $D$ in which any vertex $x \in \Fkt{V}{D}$ is contained at most $\Fkt{w}{x}$ times.
	As we may assume all the directed cycles in an optimal collection to be induced, this implies the claim.
\end{proof}

As a consequence, the clutter of vertex sets of induced directed cycles of a strongly planar digraph $D$ is MFMC and, thus, also ideal.

Given any clutter $\Brace{S,\mathcal{C}}$, we may define a corresponding dual clutter (called \emph{blocking clutter} and denoted by $\Brace{S,\mathcal{C}^\ast}$) which contains all the inclusion-wise minimal subsets $X \subseteq S$ with the property that $X \cap C \neq \emptyset$ for all $C \in \mathcal{C}$.
It is clear that the blocking clutter of the clutter of vertex sets of induced directed cycles of a digraph is just the clutter of inclusion-wise minimal feedback vertex sets.
To proceed, we will need the following theorem of Lehman.

\begin{theorem}[Lehman, \cite{lehman}, and \cite{corn}, Theorem 1.17]
	A clutter is ideal if and only if its blocking clutter is.
\end{theorem}

In our case, this implies that, for strongly planar digraphs, the clutter of minimal feedback vertex sets is ideal and, consequently, the corresponding linear optimisation problem (\ref{primal}) admits an integer optimal solution $x \geq 0$ for any real-valued vector $w \geq 0$.
By setting $w\coloneqq\textbf{1}^T$, we obtain the following.

\begin{lemma} \label{hilf}
	Let $D$ be strongly planar and let $g$ be the girth of $D$.
	Then there is a collection $F_1,\ldots,F_m$ of feedback vertex sets of $D$ equipped with a weighting $y_1,\ldots,y_m \in \mathbb{R}_{\geq 0}$ such that $y_1+\cdots+y_m=g$ and for any vertex $v \in \Fkt{V}{D}$, we have $\sum_{\CondSet{j}{ v \in F_j}}{y_j} \leq 1$.
\end{lemma}

\begin{proof}
	Let $x \geq 0$ be an integer-valued optimal solution of the linear program (\ref{primal}) corresponding to the clutter of inclusion-wise minimal feedback vertex sets of $D$ and $w=\textbf{1}^T$.
	It is easy to see from the definition of the linear program (\ref{primal}) that, in any optimal solution, we have $x \leq \textbf{1}$ (component-wise), as otherwise one could replace $x$ with $\min\Set{x,\textbf{1}}$, obtaining a better solution to the linear program, contradicting the optimality.
	Consequently, we know that $x$ has only $0$ and $1$ as entries and is thus determined by its support $X\coloneqq\FktO{supp}{x} \subseteq \Fkt{V}{D}$.
	From the conditions in the program (\ref{primal}) we derive that $X$ has a common intersection with any feedback vertex set of $D$ and thus must contain a directed cycle (as $\Fkt{V}{D} \setminus X$ cannot be a feedback vertex set).
	Hence $wx=\Abs{X} \geq g$.
	On the other hand, the $(0,1)$-vector whose support is given by the vertex set of some directed cycle of length of $g$ clearly has value $g$ and also satisfies the conditions of the program and thus is an optimal solution.
	Consequently, also the optimal value of the dual program (\ref{dual}) is $g$ and thus there is an optimal solution vector $y \ge 0$ with $y\textbf{1}=g$. This implies the claim.
\end{proof}

We are now ready to give a proof of \cref{thm:mainfractional} which will conclude this section.

\begin{proof}[Proof (of \Cref{thm:mainfractional}).]
	Let $D$ be strongly planar and let $g \geq 2$ denote the directed girth of $D$.
	We show that $\FracDichromatic{D}=\frac{g}{g-1}$.
	First of all, the fractional dichromatic number cannot increase by taking subdigraphs, and so we have $\FracDichromatic{D} \ge \FracDichromatic{\vec{C}_g}=\frac{g}{g-1}$.
	It remains to prove $\FracDichromatic{D} \leq \frac{g}{g-1}$.
	For this purpose we construct a feasible instance of the linear optimisation program (\ref{primalfract}) with value at most $\frac{g}{g-1}$.
	To do so, let $F_1,\ldots,F_m$ be a collection of feedback vertex sets as given by \cref{hilf} with a corresponding weighting $y_1,\ldots,y_m \ge 0$. The complements $\Fkt{V}{D} \setminus F_i$ are clearly acyclic for any $j \in \{1,\ldots,m\}$.
	For any acyclic vertex $A \in \Acyclic{D}$ we now define the value of the corresponding variable to be
	$$x_A\coloneqq\frac{1}{g-1}\sum_{\CondSet{j}{A=V(D) \setminus F_j}}{y_j} \geq 0.$$
	We then have for any vertex $v \in \Fkt{V}{D}$:
	$$\sum_{A \in \mathcal{A}(D,v)}{x_A}=\frac{1}{g-1}\sum_{\CondSet{j}{v \notin F_j}}{y_j}=\frac{1}{g-1}\left(\underbrace{\sum_{j=1}^{m}{y_j}}_{=g}-\underbrace{\sum_{\CondSet{j}{ v \in F_j}}{y_j}}_{\leq 1}\right) \ge \frac{g-1}{g-1}=1,$$
	so this is indeed a feasible instance of the program (\ref{primalfract}) and we obtain
	$$\FracDichromatic{D} \leq \sum_{A \in \mathcal{A}(D)}{x_A}=\frac{1}{g-1}\sum_{j=1}^{m}{y_j}=\frac{g}{g-1}$$ as desired.
\end{proof}

\section{Non-Bipartite Graphs} \label{sec:pfaff}

In the previous sections we were concerned with digraphs, which correspond exactly to the bipartite graphs with perfect matchings.
However, a matching covered graph does not need to be bipartite.
In fact, most parts of (bipartite) matching theory directly translate into the world of general matching covered graphs.
This includes, especially, tight cuts, their contractions, and pfaffian orientations.

In particular, the $M$-chromatic number is defined on all matching covered graphs.
By \cref{cor:matchinghadwiger} every bipartite Pfaffian graph has $M$-chromatic number at most $2$ for every perfect matching.
A natural question to ask would be whether this generalises to all Pfaffian graphs.
To this question there exists a rather easy negative answer.
The \emph{triangular prism} is the complement $\Complement{C_6}$ of the $6$-cycle.
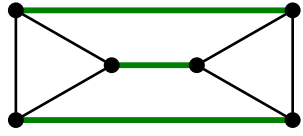
\begin{figure}[h!]
	\begin{center}
		\begin{tikzpicture}[scale=0.7]
		
		\pgfdeclarelayer{background}
		\pgfdeclarelayer{foreground}
		
		\pgfsetlayers{background,main,foreground}
		
		%%%%% Vertex Styles %%%%%
%		\tikzstyle{v:main} = [draw, circle, scale=0.5, thick,fill=black]
%		\tikzstyle{v:tree} = [draw, circle, scale=0.3, thick,fill=black]
%		\tikzstyle{v:border} = [draw, circle, scale=0.75, thick,minimum size=10.5mm]
%		\tikzstyle{v:mainfull} = [draw, circle, scale=1, thick,fill]
%		\tikzstyle{v:ghost} = [inner sep=0pt,scale=1]
%		\tikzstyle{v:marked} = [circle, scale=1.2, fill=CornflowerBlue,opacity=0.3]
%		%%%%% %%%%% %%%%%
%		
%		%%%%% Edge Styles %%%%%
%		\tikzset{>=latex} 
%		\tikzstyle{e:marker} = [line width=9pt,line cap=round,opacity=0.2,color=DarkGoldenrod]
%		\tikzstyle{e:colored} = [line width=1.2pt,color=BostonUniversityRed,cap=round,opacity=0.8]
%		\tikzstyle{e:coloredthin} = [line width=1.1pt,opacity=1,color=myGreen]
%		\tikzstyle{e:coloredborder} = [line width=2pt]
%		\tikzstyle{e:main} = [line width=1pt]
%		\tikzstyle{e:extra} = [line width=1.3pt,color=LavenderGray]
		%%%%% %%%%% %%%%%
		
		\begin{pgfonlayer}{main}
		
		%%%%% Centered Ghost Vertices %%%%%
		\node (C) [] {};
		
		%%%%% Left Center %%%%%
		\node (C1) [v:ghost, position=180:25mm from C] {};
		%\node (U1) [v:ghost, position=90:50mm from C1] {};
		%		\node (L1) [v:ghost, position=270:27mm from C1,align=center] {};
		
		%%%%% Center %%%%%
		\node (C2) [v:ghost, position=0:0mm from C] {};
		%\node (U2) [v:ghost, position=90:50mm from C2] {};
		%		\node (L2) [v:ghost, position=270:27mm from C2,align=center] {};
		
		%%%%% Right Center %%%%%
		\node (C3) [v:ghost, position=0:25mm from C] {};
		%\node (U3) [v:ghost, position=90:50mm from C3] {};
		%		\node (L3) [v:ghost, position=270:27mm from C3,align=center] {};
		%%%%% %%%%% %%%%%

		%%%%% Vertices %%%%%
		
		%%%%% Left Center %%%%%

		%%%%% %%%%% %%%%%
		
		%%%%% Center %%%%%
		
		\node (L) [v:ghost,position=180:20mm from C2] {};
		\node (R) [v:ghost,position=0:20mm from C2] {};
		
		\node(l1) [v:main,position=0:12mm from L] {};
		\node(l2) [v:main,position=120:12mm from L] {};
		\node(l3) [v:main,position=240:12mm from L] {};
		
		\node(r1) [v:main,position=180:12mm from R] {};
		\node(r2) [v:main,position=60:12mm from R] {};
		\node(r3) [v:main,position=300:12mm from R] {};
		
		%%%%% %%%%% %%%%%
		
		%%%%% Right Center %%%%%
		
		%%%%% %%%%% %%%%%
		
		%%%%% %%%%% %%%%%

		%%%%% Edges %%%%%
		
		%%%%% Left Center %%%%%

		%%%%% %%%%% %%%%%
		
		%%%%% Center %%%%%
		
		\draw (l1) [e:main] to (l2);
		\draw (l2) [e:main] to (l3);
		\draw (l3) [e:main] to (l1);
		
		\draw (r1) [e:main] to (r2);
		\draw (r2) [e:main] to (r3);
		\draw (r3) [e:main] to (r1);
		
		%%%%% %%%%% %%%%%
		
		%%%%% Right Center %%%%%

		%%%%% %%%%% %%%%%
		
		%%%%% %%%%% %%%%%
		
		\end{pgfonlayer}
		
		%%%%% %%%%% %%%%%

		%%%%% Background %%%%%
		\begin{pgfonlayer}{background}
		
		\draw (l1) [e:matching1] to (r1);
		\draw (l2) [e:matching1] to (r2);
		\draw (l3) [e:matching1] to (r3);
		
%		\draw (l1) [e:coloredthin] to (r1);
%		\draw (l2) [e:coloredthin] to (r2);
%		\draw (l3) [e:coloredthin] to (r3);
		
		\end{pgfonlayer}	
		%%%%% %%%%% %%%%%
		
		%%%%% Foreground %%%%%
		\begin{pgfonlayer}{foreground}

		\end{pgfonlayer}
		%%%%% %%%%% %%%%%
		\end{tikzpicture}
	\end{center}
	\caption{The triangular prism $\Complement{C_6}$ together with a perfect matching $M$.}
	\label{fig:prism}
\end{figure}

It is planar and therefore Pfaffian, but when considering the perfect matching $M$ from \cref{fig:prism}, one can see that any two of the three edges in $M$ lie together on a $4$-cycle.
Hence no two of the three edges may receive the same colour and therefore $\Mchromatic{\Complement{C_6}}{M}=3$.

In \cref{cor:matchinghadwiger} we went for a class closed under matching minors, so a next step would be to consider a subclass of the $\Complement{C_6}$-matching minor-free graphs.
The triangular prism is one of two graphs appearing in a fundamental theorem by Lov{\'a}sz on non-bipartite matching covered graphs.

\begin{theorem}[Lov{\'a}sz \cite{lovasz1987matching}]\label{thm:K4andprism}
	Every non-bipartite matching covered graph contains a conformal bisubdivision of $K_4$ or $\Complement{C_6}$.
\end{theorem}

A matching covered graph without a non-trivial tight cut is called a \emph{brace} if it is bipartite and a \emph{brick} otherwise. In his seminal paper \cite{lovasz1987matching}, Lov{\'a}sz introduced a decomposition procedure, known under the name \emph{tight cut decomposition}, which, given a matching covered graph, searches for non-trivial tight cuts, computes both tight cut contractions, and iterates this for both reduced matching covered graphs, until a list of bricks and braces, which are not reducible any more, is obtained. Among many other things, Lov{\'a}sz proved that the list of bricks and braces does not depend on the chosen order in which the tight cuts are contracted.
As the following theorem shows, braces correspond exactly to the strongly $2$-connected digraphs.

\begin{theorem}[Lov{\'a}sz and Plummer \cite{lovasz1986matching}]
	A bipartite graph $G$ is a brace if and only if it is $2$-extendable.
\end{theorem}

Bricks have a more complicated structure and although every $2$-extendable graph is either a brick or a brace as seen in \cref{thm:2extnotightcut}, there are bricks that are not $2$-extendable.
For an example of such a brick consider the triangular prism.

\begin{theorem}[Plummer \cite{plummer1980n}]\label{thm:2extnotightcut}
	Let $G$ be a $2$-extendable graph.
	Then, $G$ is either a brace or a brick.
\end{theorem}

There exists a generalisation of tight cuts that crosses the border towards bricks.
Given a matching covered graph $G$ and a set $X\subseteq\Fkt{V}{G}$ we call the graph $G_X$ obtained from $G$ by identifying $X$ into a single vertex the \emph{$X$-contraction} of $G$.
Now a cut $\Cut{}{X}$ is called \emph{separating} if both $G_X$ and $G_{\Complement{X}}$ are matching covered.

\begin{theorem}[de Carvalho, Lucchesi and Murty \cite{de2002conjecture}]\label{thm:separatingcuts}
	Let $G$ be a matching covered graph and $X\subseteq\Fkt{V}{G}$.
	The cut $\Cut{}{X}$ is separating if and only if for every edge $e\in\Fkt{E}{G}$ there is a perfect matching $M_e$ of $G$ containing $e$ such that $\Abs{\Cut{}{X}\cap M_e}=1$.
\end{theorem}

We call a matching covered graph \emph{solid} if every non-trivial separating cut is already tight.

One can easily check the following lemma on bipartite graphs, showing that any bipartite matching covered graph is solid.

\begin{lemma}[de Carvalho, Lucchesi, Kothari and Murty \cite{lucchesi2018two}]\label{lemma:bipartitenosepcuts}
Let $G$ be a bipartite matching covered graph. Then $\Cut{}{X}$ is separating if and only if it is tight.
\end{lemma}

Moreover, being solid is preserved by tight cut contractions (cf.\@ \cite{de2002conjecture}) and thus a matching covered graph is solid of and only if all of its bricks are solid.

Please note that even bricks may contain non-trivial separating cuts.
Again consider the triangular prism from \cref{fig:prism} and take a cut around one of the two triangles.
Such a cut is separating.
In fact, the existence of a prism as a conformal bisubdivision immediately implies the existence of a non-trivial and non-tight separating cut.

\begin{lemma}[de Carvalho, Lucchesi, Kothari and Murty \cite{lucchesi2018two}]\label{lemma:solidprismfree}
	Every solid graph is $\Complement{C_6}$-free.
\end{lemma}

The goal of this section is to establish an extension of \cref{cor:matchinghadwiger} to non-bipartite matching covered graphs in the form of a conjecture.

\begin{conjecture}\label{con:solidandpfaffian}
Let $G$ be a solid and Pfaffian graph and $M$ a perfect matching of $G$. Then $\Mchromatic{G}{M} \leq 2$.
\end{conjecture}

To provide some evidence towards \cref{con:solidandpfaffian}, the remainder of this section is dedicated to settle the planar case.
For this we first establish a more general version of \cref{lemma:splitting} by proving it directly for tight cut contractions.
We will need a bit of notation here.
If $G$ is matching covered, $M$ a perfect matching, and $G_X$ is a tight cut contraction of $\Cut{}{X}$ with contraction vertex $v_X$, we denote by $M_X$ the perfect matching $\CondSet{e\in M}{e\subseteq\Fkt{V}{G_X}}\cup\Set{uv_X}$ where $u$ is the unique vertex of $X$ covered by the edge of $M$ in $\Cut{}{X}$.

\begin{lemma}\label{lemma:tightcutsplitting}
	Let $G$ be a matching covered graph, $\Cut{}{X}$ a non-trivial tight cut in $G$ and $M$ a perfect matching.
	If $\Mchromatic{G_X}{M_X}\leq 2$ and $\Mchromatic{G_{\Complement{X}}}{M_{\Complement{X}}}\leq 2$, then $\Mchromatic{G}{M}\leq 2$.
\end{lemma}

\begin{proof}
	For $Y\in\Set{X,\Complement{X}}$ let $c_Y$ be a proper $2$-colouring of $M_Y$ in $G_Y$.
	Let $e_Y\in M_Y$ be the edge covering the contraction vertex. Then we can rename the colours for $c_X$ and $c_{\Complement{X}}$ such that $\Fkt{c_X}{e_X}=\Fkt{c_{\Complement{X}}}{e_{\Complement{X}}}$ and we define a colouring for $M$ as follows.
	\begin{align*}
	\Fkt{c}{e}\coloneqq\ThreeCases{\Fkt{c_X}{e}}{e\in M_X}{\Fkt{c_X}{e_X}=\Fkt{c_{\Complement{X}}}{e_{\Complement{X}}}}{e\in\Cut{}{X}\cap M}{\Fkt{c_{\Complement{X}}}{e}}{e\in M_{\Complement{X}}}
	\end{align*}
	Suppose $G$ contains an $M$-alternating cycle $C$ that is monochromatic with respect to $c$.
	If $\Fkt{V}{C}$ is a subset of either $X$ or $\Complement{X}$, by definition of $c$, $C$ must be a monochromatic cycle in either $G_X$ or $G_{\Complement{X}}$ and, thus,  $C$ must cross $\Cut{}{X}$.
	Since $\Cut{}{X}$ is tight, $C-\Brace{\Cut{}{X}\cap\Fkt{E}{C}}$ contains exactly $2$ components.
	Each of them is a path of even length and $M$ covers all vertices but exactly one endpoint.
	Moreover, each of these paths forms, together with the corresponding edges in $\Cut{}{X}$, an $M_Y$-alternating cycle in their respective contraction $G_Y$.
	By definition of $c$, these two cycles must also be monochromatic which ultimately contradicts the choice of the $c_Y$ and completes the proof.
\end{proof}

Using the tight cut decomposition and the above Theorem, it suffices to show that every perfect matching of a solid planar brick or planar brace is $2$-colourable.
The brace case is of course taken care of by \cref{cor:matchinghadwiger} and thus our only concern are the solid planar bricks.
By \cref{lemma:solidprismfree} we only have to consider $\Complement{C_6}$-free planar bricks.
A theorem of Kothari and Murty (cf.\@ \cite{kothari2016k4}) gives a precise description of these bricks.

A graph $W_k$ consisting of a cycle of length $k$ and a single vertex adjacent to every vertex on the cycle is called a \emph{wheel}.
If $k$ is odd, we call $W_k$ an odd wheel; every odd wheel is a brick.

Let $\Brace{u_1,u_2,\dots,u_k}$ and $\Brace{v_1,v_2,\dots,v_k}$ be two disjoint paths with $k\geq2$.
The graph $S_k$ obtained from the union of these paths by adding the edges $u_iv_i$ for all $i\in\Set{1,\dots,k}$, two new vertices $x$ and $y$ joined by an edge and the edges $xu_1$, $xv_1$, $yu_k$, $yv_k$, is called a \emph{staircase of order $2k+2$}.
Every $S_{2k+2}$ is a brick and $S_6$ is isomorphic to the triangular prism.

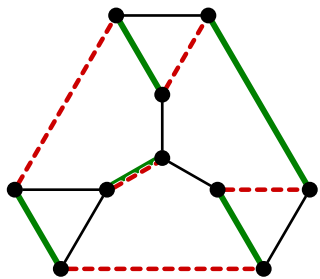
\begin{figure}[h!]
	\begin{center}
		\begin{tikzpicture}[scale=0.7]
		
		\pgfdeclarelayer{background}
		\pgfdeclarelayer{foreground}
		
		\pgfsetlayers{background,main,foreground}
		
%		%%%%% Vertex Styles %%%%%
%		\tikzstyle{v:main} = [draw, circle, scale=0.5, thick,fill=black]
%		\tikzstyle{v:tree} = [draw, circle, scale=0.3, thick,fill=black]
%		\tikzstyle{v:border} = [draw, circle, scale=0.75, thick,minimum size=10.5mm]
%		\tikzstyle{v:mainfull} = [draw, circle, scale=1, thick,fill]
%		\tikzstyle{v:ghost} = [inner sep=0pt,scale=1]
%		\tikzstyle{v:marked} = [circle, scale=1.2, fill=CornflowerBlue,opacity=0.3]
%		%%%%% %%%%% %%%%%
%		
%		%%%%% Edge Styles %%%%%
%		\tikzset{>=latex} 
%		\tikzstyle{e:marker} = [line width=9pt,line cap=round,opacity=0.2,color=DarkGoldenrod]
%		\tikzstyle{e:colored} = [line width=1.2pt,color=BostonUniversityRed,cap=round,opacity=0.8]
%		\tikzstyle{e:coloredthin} = [line width=1.1pt,opacity=1]
%		\tikzstyle{e:coloredborder} = [line width=2pt]
%		\tikzstyle{e:main} = [line width=1pt]
%		\tikzstyle{e:extra} = [line width=1.3pt,color=LavenderGray]
%		%%%%% %%%%% %%%%%
		
		\begin{pgfonlayer}{main}
		
		%%%%% Centered Ghost Vertices %%%%%
		\node (C) [] {};
		
		%%%%% Left Center %%%%%
		\node (C1) [v:ghost, position=180:25mm from C] {};
		%\node (U1) [v:ghost, position=90:50mm from C1] {};
		%		\node (L1) [v:ghost, position=270:27mm from C1,align=center] {};
		
		%%%%% Center %%%%%
		\node (C2) [v:ghost, position=0:0mm from C] {};
		%\node (U2) [v:ghost, position=90:50mm from C2] {};
		%		\node (L2) [v:ghost, position=270:27mm from C2,align=center] {};
		
		%%%%% Right Center %%%%%
		\node (C3) [v:ghost, position=0:25mm from C] {};
		%\node (U3) [v:ghost, position=90:50mm from C3] {};
		%		\node (L3) [v:ghost, position=270:27mm from C3,align=center] {};
		%%%%% %%%%% %%%%%

		%%%%% Vertices %%%%%
		
		%%%%% Left Center %%%%%

		%%%%% %%%%% %%%%%
		
		%%%%% Center %%%%%
		
		\node (T) [v:ghost,position=90:22mm from C2] {};
		\node (L) [v:ghost,position=210:22mm from C2] {};
		\node (R) [v:ghost,position=330:22mm from C2] {};
		
		\node(l1) [v:main,position=30:10mm from L] {};
		\node(l2) [v:main,position=150:10mm from L] {};
		\node(l3) [v:main,position=270:10mm from L] {};
		
		\node(r1) [v:main,position=150:10mm from R] {};
		\node(r2) [v:main,position=30:10mm from R] {};
		\node(r3) [v:main,position=270:10mm from R] {};
		
		\node(t1) [v:main,position=270:10mm from T] {};
		\node(t2) [v:main,position=150:10mm from T] {};
		\node(t3) [v:main,position=30:10mm from T] {};
		
		\node(M) [v:main,position=0:0mm from C2] {};
		
		%%%%% %%%%% %%%%%
		
		%%%%% Right Center %%%%%
		
		%%%%% %%%%% %%%%%
		
		%%%%% %%%%% %%%%%

		%%%%% Edges %%%%%
		
		%%%%% Left Center %%%%%

		%%%%% %%%%% %%%%%
		
		%%%%% Center %%%%%
		
		\draw (l1) [e:main] to (l2);
%		\draw (l2) [e:main] to (l3);
		\draw (l3) [e:main] to (l1);
		
%		\draw (r1) [e:main] to (r2);
		\draw (r2) [e:main] to (r3);
%		\draw (r3) [e:main] to (r1);
		
%		\draw (t1) [e:main] to (t2);
		\draw (t2) [e:main] to (t3);
%		\draw (t3) [e:main] to (t1);
		
		\draw (M) [e:main] to (t1);
		\draw (M) [e:main] to (r1);
%		\draw (M) [e:main] to (l1);
		
%		\draw (t2) [e:main] to (l2);
%		\draw (t3) [e:main] to (r2);
%		\draw (r3) [e:main] to (l3);

		%%%%% %%%%% %%%%%
		
		%%%%% Right Center %%%%%

		%%%%% %%%%% %%%%%
		
		%%%%% %%%%% %%%%%
		
		\end{pgfonlayer}
		
		%%%%% %%%%% %%%%%

		%%%%% Background %%%%%
		\begin{pgfonlayer}{background}
		
		\draw (l1) [e:matching1,e:shiftedright] to (M);
%		\draw (l1) [line width=3.9pt,cap=round,color=BostonUniversityRed] to (M);
		\draw (l1) [e:matching2border,e:shiftedleft] to (M);
%		\draw (l1) [e:coloredthin,color=myGreen] to (M);
		
		\draw (l2) [e:matching1] to (l3);
		\draw (r1) [e:matching1] to (r3);
		\draw (t1) [e:matching1] to (t2);
		\draw (t3) [e:matching1] to (r2);
		
%		\draw (l2) [e:coloredthin,color=myGreen] to (l3);
%		\draw (r1) [e:coloredthin,color=myGreen] to (r3);
%		\draw (t1) [e:coloredthin,color=myGreen] to (t2);
%		\draw (t3) [e:coloredthin,color=myGreen] to (r2);
%		
		\draw (t1) [e:matching2] to (t3);
		\draw (l2) [e:matching2] to (t2);
		\draw (r1) [e:matching2] to (r2);
		\draw (l3) [e:matching2] to (r3);
		
%		\draw (t1) [e:coloredthin,color=BostonUniversityRed] to (t3);
%		\draw (l2) [e:coloredthin,color=BostonUniversityRed] to (t2);
%		\draw (r1) [e:coloredthin,color=BostonUniversityRed] to (r2);
%		\draw (l3) [e:coloredthin,color=BostonUniversityRed] to (r3);
		
		\end{pgfonlayer}	
		%%%%% %%%%% %%%%%
		
		%%%%% Foreground %%%%%
		\begin{pgfonlayer}{foreground}

		\end{pgfonlayer}
		%%%%% %%%%% %%%%%
		\end{tikzpicture}
	\end{center}
	\caption{The \emph{tricorn} together with a perfect matching of \textcolor{myGreen}{type I} and a perfect matching of \textcolor{BostonUniversityRed}{type II}.}
	\label{fig:tricorn}
\end{figure}

\begin{theorem}[Kothari and Murty \cite{kothari2016k4}]\label{thm:planarprismfree}
	\noindent
	\begin{enumerate}
		\item A matching-covered graph is $\Complement{C_6}$-free if and only if all the bricks and braces in its tight cut decomposition are $\Complement{C_6}$-free.
		\item The only planar $\Complement{C_6}$-free bricks are the odd wheels, the staircases of order $4k$ and the tricorn (see \cref{fig:tricorn}).
	\end{enumerate}
\end{theorem}

If we have a planar and matching covered graph $G$ that does not contain a conformal bisubdivision of $\Complement{C_6}$, by \cref{thm:planarprismfree} the only bricks $G$ can have are odd wheels, staircases of orders divisible by $4$ and tricorns.
Along with these bricks, $G$ can have any planar brace.
Planar braces are Pfaffian and thus by \cref{cor:matchinghadwiger} $2$-colourable.
While it is our goal to provide evidence towards the $2$-colourability of solid Pfaffian graphs, for the planar case we can prove a stronger statement, namely \cref{thm:planarprismfree2colours}.

\begin{proof}[Proof (of \Cref{thm:planarprismfree2colours})]
	As we have seen, by \cref{lemma:tightcutsplitting} it suffices to consider planar and $\Complement{C_6}$-free bricks and planar braces.
	Since planar braces $G$ all satisfy $\Mchromatic{G}{M}\leq 2$ for all perfect matchings $M$ by \cref{cor:matchinghadwiger} the only case left is where $G$ is a planar and $\Complement{C_6}$-free brick.
	So with \cref{thm:planarprismfree} we have to show that the perfect matchings of the odd wheels, staircases of order $4k$ and the tricorn are $2$-colourable.
	
	\emph{Odd Wheels}
	
	For $K_4=W_3$ we have exactly two edges in every perfect matching and thus are done.
	Let $k\geq 4$ be any odd number.
	For the odd wheel $W_k$ on $k+1$ vertices, let $x$ be the unique vertex of degree $k$.
	Clearly every perfect matching $M$ has to cover $x$ with an edge, say, $e_x^M$, and every other matching edge lies on the cycle induced by the neighbourhood of $x$.
	Consider the graph induced by $\bigcup M\setminus\Set{e_x^M}$.
	Since $\Fkt{N}{x}$ induces a cycle, this graph is a path and thus every $M$-alternating cycle in $W_k$ must contain $e_x^M$.
	Hence, by colouring $e_x^M$ with $0$ and every other edge of $M$ with $1$ we have found a proper $2$-colouring for $M$ in $W_k$.
	
	\emph{Staircases of Order $4k$}
	
	For the staircases $S_{4k}$ we give a $2$-colouring $c\colon\Fkt{E}{S_{4k}}\rightarrow\Set{0,1}$ of the edges that induces a proper $2$-colouring for every perfect matching.
	Let $xy$ be the unique edge with endpoints in two disjoint triangles.
	Let $\Brace{u_1,\dots,u_{2k-1}}$ be the path from the construction of $S_{4k}$ not on the outer face and assume $xu_1$ to be an edge of $S_{4k}$.
	We colour $xy$ with $0$.
	Then, going counter-clockwise around the outer face, we assign $0$ as the colour of the edges $xv_1$ and $v_1v_2$, the next two edges receive the colour $1$, then two times colour $0$ and so forth until the edge $v_{2k-1}y$ is coloured.
	Since $S_{4k}$ is of order $4k$ we colour $2k-2$ edges this way and the last two edges receive colour $1$.
	With this the path $(x,v_1,\ldots,v_{2k-1},y)$ on the outer face is coloured.
	We set $\Fkt{c}{u_iu_{i+1}}\coloneqq1-\Fkt{c}{v_iv_{i+1}}$ for $i=1,\ldots,2k-2$ and $\Fkt{c}{xu_1}\coloneqq1$ while $\Fkt{c}{yu_{2k-1}}=0$.
	At last we need to colour the spokes.
	Let $\Fkt{c}{v_iu_i}\coloneqq i\mod2$ for $i=1,\ldots,2k-1$.
	For an illustration consider \cref{fig:staircase}.
	To show that $c$ induces a proper $2$-colouring for every perfect matching, we must show that there is no conformal cycle $C$ such that every second edge has the same colour.
	Assume for a contradiction that $\InducedSubgraph{S_{4k}}{\Fkt{c^{-1}}{0}}$ contains $C$.
	However, this graph contains a single even length cycle and this cycle contains exactly the vertices incident with at most one edge of colour $1$ in $G$.
	Therefore $V(G)\setminus V(C)$ is a stable set and thus $C$ is not conformal, a contradiction.
	Thus $C$ must contain an edge of colour $1$ and therefore, by construction, also two consecutive such edges. Consequently, we must have $i=1$, and every second edge must be of colour $1$.
	There does not exist a path of length $5$ in $S_{4k}$ such that the first, third, and fifth edge are coloured with $1$, hence $C$ must have length $4$.
	Clearly none of the $4$-cycles contains two disjoint edges of the same colour and thus $C$ cannot exist.
	
\begin{figure}[h!]
	\begin{center}
		\begin{tikzpicture}[scale=0.7]
		
		\pgfdeclarelayer{background}
		\pgfdeclarelayer{foreground}
		
		\pgfsetlayers{background,main,foreground}
		
		%%%%% Vertex Styles %%%%%
		\tikzstyle{v:main} = [draw, circle, scale=0.5, thick,fill=black]
		\tikzstyle{v:tree} = [draw, circle, scale=0.3, thick,fill=black]
		\tikzstyle{v:border} = [draw, circle, scale=0.75, thick,minimum size=10.5mm]
		\tikzstyle{v:mainfull} = [draw, circle, scale=1, thick,fill]
		\tikzstyle{v:ghost} = [inner sep=0pt,scale=1]
		\tikzstyle{v:marked} = [circle, scale=1.2, fill=CornflowerBlue,opacity=0.3]
		%%%%% %%%%% %%%%%
		
		%%%%% Edge Styles %%%%%
		\tikzset{>=latex} 
		\tikzstyle{e:marker} = [line width=9pt,line cap=round,opacity=0.2,color=DarkGoldenrod]
		\tikzstyle{e:colored} = [line width=1.2pt,color=BostonUniversityRed,cap=round,opacity=0.8]
		\tikzstyle{e:coloredthin} = [line width=1.1pt,opacity=1]
		\tikzstyle{e:coloredborder} = [line width=2pt]
		\tikzstyle{e:main} = [line width=1.5pt]
		\tikzstyle{e:extra} = [line width=1.3pt,color=LavenderGray]
		%%%%% %%%%% %%%%%
		
		\begin{pgfonlayer}{main}
		
		%%%%% Centered Ghost Vertices %%%%%
		\node (C) [] {};
		
		%%%%% Left Center %%%%%
		\node (C1) [v:ghost, position=180:25mm from C] {};
		%\node (U1) [v:ghost, position=90:50mm from C1] {};
		%		\node (L1) [v:ghost, position=270:27mm from C1,align=center] {};
		
		%%%%% Center %%%%%
		\node (C2) [v:ghost, position=0:0mm from C] {};
		%\node (U2) [v:ghost, position=90:50mm from C2] {};
		%		\node (L2) [v:ghost, position=270:27mm from C2,align=center] {};
		
		%%%%% Right Center %%%%%
		\node (C3) [v:ghost, position=0:25mm from C] {};
		%\node (U3) [v:ghost, position=90:50mm from C3] {};
		%		\node (L3) [v:ghost, position=270:27mm from C3,align=center] {};
		%%%%% %%%%% %%%%%

		%%%%% Vertices %%%%%
		
		%%%%% Left Center %%%%%

		%%%%% %%%%% %%%%%
		
		%%%%% Center %%%%%
		
		\node(u1) [v:main,position=115.7:21mm from C2] {};
		\node(u2) [v:main,position=167.13:21mm from C2] {};
		\node(u3) [v:main,position=218.56:21mm from C2] {};
		\node(u4) [v:main,position=269.99:21mm from C2] {};
		\node(u5) [v:main,position=321.42:21mm from C2] {};
		\node(u6) [v:main,position=372.85:21mm from C2] {};
		\node(u7) [v:main,position=424.28:21mm from C2] {};
%		\node(u8) [v:main,position=67.5:21mm from C2] {};
		
		\node(v1) [v:main,position=115.7:10mm from u1] {};
		\node(v2) [v:main,position=167.13:10mm from u2] {};
		\node(v3) [v:main,position=218.56:10mm from u3] {};
		\node(v4) [v:main,position=269.99:10mm from u4] {};
		\node(v5) [v:main,position=321.42:10mm from u5] {};
		\node(v6) [v:main,position=372.85:10mm from u6] {};
		\node(v7) [v:main,position=424.28:10mm from u7] {};
%		\node(v8) [v:main,position=67.5:10mm from u8] {};
		
		\node(x) [v:main,position=101.25:26mm from C2] {};
		\node(y) [v:main,position=78.75:26mm from C2] {};
		
		%%%%% %%%%% %%%%%
		
		%%%%% Right Center %%%%%
		
		%%%%% %%%%% %%%%%
		
		%%%%% %%%%% %%%%%

		%%%%% Edges %%%%%
		
		%%%%% Left Center %%%%%

		%%%%% %%%%% %%%%%
		
		%%%%% Center %%%%%
		
		\draw (u1) [e:main,dashed,color=BostonUniversityRed] to (u2);
		\draw (u2) [e:main,color=myGreen] to (u3);
		\draw (u3) [e:main,color=myGreen] to (u4);
		\draw (u4) [e:main,dashed,color=BostonUniversityRed] to (u5);
		\draw (u5) [e:main,dashed,color=BostonUniversityRed] to (u6);
		\draw (u6) [e:main,color=myGreen] to (u7);
%		\draw (u7) [e:main,color=myGreen] to (u8);
%		
		\draw (v1) [e:main,color=myGreen] to (v2);
		\draw (v2) [e:main,dashed,color=BostonUniversityRed] to (v3);
		\draw (v3) [e:main,dashed,color=BostonUniversityRed] to (v4);
		\draw (v4) [e:main,color=myGreen] to (v5);
		\draw (v5) [e:main,color=myGreen] to (v6);
		\draw (v6) [e:main,dashed,color=BostonUniversityRed] to (v7);
%		\draw (v7) [e:main,dashed,color=BostonUniversityRed] to (v8);
%		
		\draw (u1) [e:main,dashed,color=BostonUniversityRed] to (v1);
		\draw (u2) [e:main,color=myGreen] to (v2);
		\draw (u3) [e:main,dashed,color=BostonUniversityRed] to (v3);
		\draw (u4) [e:main,color=myGreen] to (v4);
		\draw (u5) [e:main,dashed,color=BostonUniversityRed] to (v5);
		\draw (u6) [e:main,color=myGreen] to (v6);
		\draw (u7) [e:main,dashed,color=BostonUniversityRed] to (v7);
%		\draw (u8) [e:main,color=myGreen] to (v8);
		
		\draw (x) [e:main,color=myGreen] to (v1);
		\draw (x) [e:main,color=BostonUniversityRed,dashed] to (u1);
		
		\draw (y) [e:main,color=BostonUniversityRed,dashed] to (v7);
		\draw (y) [e:main,color=myGreen] to (u7);
		
		\draw (x) [e:main,color=myGreen] to (y);
		
		%%%%% %%%%% %%%%%
		
		%%%%% Right Center %%%%%

		%%%%% %%%%% %%%%%
		
		%%%%% %%%%% %%%%%
		
		\end{pgfonlayer}
		
		%%%%% %%%%% %%%%%

		%%%%% Background %%%%%
		\begin{pgfonlayer}{background}

		\end{pgfonlayer}	
		%%%%% %%%%% %%%%%
		
		%%%%% Foreground %%%%%
		\begin{pgfonlayer}{foreground}

		\end{pgfonlayer}
		%%%%% %%%%% %%%%%
		\end{tikzpicture}
	\end{center}
	\caption{The staircase of order $16$ together with a $2$-colouring of the edges inducing a proper $2$-colouring for every perfect matching.
	The solid edges are considered to be of colour $0$, while the dashed ones are of colour $1$.}
	\label{fig:staircase}
\end{figure}
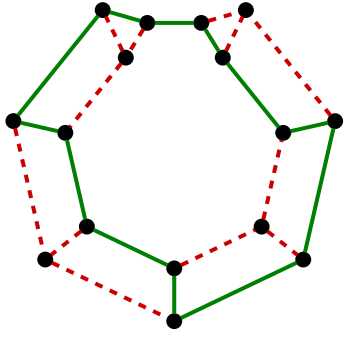

\emph{Tricorn}

For the tricorn we first observe that we can classify its perfect matchings into two types.
Any perfect matching either contains exactly one edge on the outer face (compare \cref{fig:tricorn}) that belongs to a triangle or none.
If we fix such an edge $e$ on the outer face belonging to a triangle for our perfect matching $M_1$, the remaining edges of $M_1$ are uniquely determined.
This can be seen as follows: Taking an edge from one of the triangles forces us to match the remaining vertex of said triangle to the middle vertex.
Then the remaining neighbours of the middle vertex have to be matched within their respective triangles in such a way that the remaining two vertices are adjacent.
There is only one way to do this after $e$ has been chosen and thus $\Set{e}$ is in fact a forcing set for $M_1$.
Hence colouring $e$ with $0$ and all other edges of $M_1$ with $1$ yields the desired colouring.
We call such a matching \emph{type I}.

A matching of \emph{type II} is a matching not containing any edge on the outer face belonging to a triangle.
Note that any perfect matching must contain two edges of the outer face.
So let $e_1$ and $e_2$ be these two edges.
One of the three triangles contains an endpoint from both $e_1$ and $e_2$, and its third vertex has to be matched to the middle one.
This is already enough to determine the last two edges and we obtain $M_2$.
Hence $\Set{e_1,e_2}$ is a forcing set of $M_2$ and, since the tricorn contains no $4$-cycle, by colouring $e_1$ and $e_2$ with $0$ and the rest of $M_2$ with $1$ we are done.

By the above discussion it is clear that any perfect matching of the tricorn is either of type I or II and this concludes the proof.
\end{proof}

One can easily see that any cut around a triangle in the tricorn or a staircase is separating.
Moreover, one can check that the odd wheels are indeed solid.
Hence we have the following corollary.

\begin{corollary}[de Carvalho, Lucchesi and Murty \cite{de2006build}]\label{cor:planarsolidbricks}
The only planar solid bricks are the odd wheels.
\end{corollary}

With this, the planar and solid case follows immediately.

\begin{corollary}\label{thm:planarsolid2colours}
Let $G$ be a planar solid graph and $M$ a perfect matching of $G$, then $\Mchromatic{G}{M}\leq2$.
\end{corollary}

Please note that the proof of \cref{thm:planarprismfree2colours} works for every Pfaffian matching covered graph whose bricks are planar and $\Complement{C_6}$-free.
If one was able to show that the number of edges in a solid Pfaffian brick is linearly bounded in the number of vertices, an approach similar to the one for \cref{thm:mainthm} would likely be successful.
It does not seem very likely that solid bricks in general can be very dense, as they cannot contain conformal bisubdivisions of the triangular prism, however, no linear bound on the number of edges is known.
\section{List Colourings of Non-Even Digraphs}
List colourings naturally generalise several types of colourings of graphs and have been widely investigated. While a lot of progress has been made in the last decades, many important questions, such as the \emph{list colouring conjecture}, still remain open. 

It is natural to apply the concept of list colouring also to colourings of digraphs. Indeed, such a notion was investigated in \cite{lists}. Therein, for a given digraph $D$ equipped with an assignment of finite \emph{colour lists} $\mathcal{L}=\{L(v)|v \in V(D)\}$ to the vertices, an \emph{$\mathcal{L}$-list-colouring} of $D$ is defined to be a choice function $c:V(D) \rightarrow \bigcup \mathcal{L}$ such that for any vertex $v \in V(D)$, we have $c(v) \in L(v)$, and moreover, $c$ defines a proper digraph colouring, that is, $D[c^{-1}(i)]$ is acyclic for all $i \in \bigcup \mathcal{L}$. 

Putting $L(v):=\{1,\ldots,k\}$ for each vertex simply yields the definition of a usual digraph $k$-colouring. In \cite{lists}, a digraph $D$ is called \emph{$k$-list colourable} (also \emph{$k$-choosable}) if for any list assignment $\mathcal{L}$, where $|L(v)| \ge k$ for every $v \in V(D)$, there is an $\mathcal{L}$-list colouring of $D$. The smallest integer $k \ge 1$ for which a digraph $D$ is $k$-choosable now is defined to be the \emph{list dichromatic number} (also \emph{choice number}) $\vec{\chi}_\ell(D)$. Clearly, we have $\vec{\chi}(D) \leq \vec{\chi}_\ell(D)$ for every digraph. However, as pointed out in \cite{lists}, this estimate can be arbitrarily bad in general. 

It is therefore desirable to identify classes of digraphs with bounded choice number. In the context of \cref{con:twocolours}, the authors of \cite{lists} observed that every oriented planar digraph is $3$-choosable and posed the question whether all oriented planar digraphs are $2$-choosable. 

We have shown in \cref{sec:noneven} that all non-even digraphs are $2$-colourable, and so it is natural to ask whether they are even $2$-choosable. This question can rather easily be answered in the negative, see \cref{fig:listexample} for an example of a strongly planar digraph with choice number $3$. In the remainder of this section, we show that $3$ is the (best possible) upper bound for the choice number of non-even digraphs.

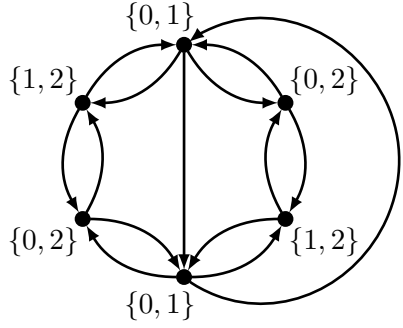
\begin{figure}[h!]
	\begin{center}
		\begin{tikzpicture}[scale=0.7]
		
		\pgfdeclarelayer{background}
		\pgfdeclarelayer{foreground}
		
		\pgfsetlayers{background,main,foreground}
		
		%%%%% Vertex Styles %%%%%
		\tikzstyle{v:main} = [draw, circle, scale=0.5, thick,fill=black]
		\tikzstyle{v:tree} = [draw, circle, scale=0.3, thick,fill=black]
		\tikzstyle{v:border} = [draw, circle, scale=0.75, thick,minimum size=10.5mm]
		\tikzstyle{v:mainfull} = [draw, circle, scale=1, thick,fill]
		\tikzstyle{v:ghost} = [inner sep=0pt,scale=1]
		\tikzstyle{v:marked} = [circle, scale=1.2, fill=CornflowerBlue,opacity=0.3]
		%%%%% %%%%% %%%%%
		
		%%%%% Edge Styles %%%%%
		\tikzset{>=latex} 
		\tikzstyle{e:marker} = [line width=9pt,line cap=round,opacity=0.2,color=DarkGoldenrod]
		\tikzstyle{e:colored} = [line width=1.2pt,color=BostonUniversityRed,cap=round,opacity=0.8]
		\tikzstyle{e:coloredthin} = [line width=1.1pt,opacity=0.8]
		\tikzstyle{e:coloredborder} = [line width=2pt]
		\tikzstyle{e:main} = [line width=1pt]
		\tikzstyle{e:extra} = [line width=1.3pt,color=LavenderGray]
		%%%%% %%%%% %%%%%
		
		\begin{pgfonlayer}{main}
		
		%%%%% Centered Ghost Vertices %%%%%
		\node (C) [] {};
		
		%%%%% Left Center %%%%%
		\node (C1) [v:ghost, position=180:25mm from C] {};
		%\node (U1) [v:ghost, position=90:50mm from C1] {};
		%		\node (L1) [v:ghost, position=270:27mm from C1,align=center] {};
		
		%%%%% Center %%%%%
		\node (C2) [v:ghost, position=0:0mm from C] {};
		%\node (U2) [v:ghost, position=90:50mm from C2] {};
		%		\node (L2) [v:ghost, position=270:27mm from C2,align=center] {};
		
		%%%%% Right Center %%%%%
		\node (C3) [v:ghost, position=0:25mm from C] {};
		%\node (U3) [v:ghost, position=90:50mm from C3] {};
		%		\node (L3) [v:ghost, position=270:27mm from C3,align=center] {};
		%%%%% %%%%% %%%%%

		%%%%% Vertices %%%%%
		
		%%%%% Left Center %%%%%

		%%%%% %%%%% %%%%%
		
		%%%%% Center %%%%%
		
		\node (v1) [v:main,position=90:22mm from C2] {};
		\node (v2) [v:main,position=150:22mm from C2] {};
		\node (v3) [v:main,position=210:22mm from C2] {};
		\node (v4) [v:main,position=270:22mm from C2] {};
		\node (v5) [v:main,position=330:22mm from C2] {};
		\node (v6) [v:main,position=30:22mm from C2] {};
		
		\node (g1) [v:ghost,position=330:7mm from v4] {};
		\node (g2) [v:ghost,position=0:19mm from v5] {};
		\node (g3) [v:ghost,position=0:19mm from v6] {};
		\node (g4) [v:ghost,position=30:7mm from v1] {};
		
		\node (v1g) [v:ghost,position=30:1mm from v1] {};
		
		\node (l1) [v:ghost,position=130:7mm from v1] {$\Set{0,1}$};
		\node (l2) [v:ghost,position=150:8.5mm from v2] {$\Set{1,2}$};
		\node (l3) [v:ghost,position=210:8.5mm from v3] {$\Set{0,2}$};
		\node (l4) [v:ghost,position=230:7mm from v4] {$\Set{0,1}$};
		\node (l5) [v:ghost,position=330:8.5mm from v5] {$\Set{1,2}$};
		\node (l6) [v:ghost,position=30:8.5mm from v6] {$\Set{0,2}$};
		
		%%%%% %%%%% %%%%%
		
		%%%%% Right Center %%%%%
		
		%%%%% %%%%% %%%%%
		
		%%%%% %%%%% %%%%%

		%%%%% Edges %%%%%
		
		%%%%% Left Center %%%%%

		%%%%% %%%%% %%%%%
		
		%%%%% Center %%%%%
		
		\draw (v1) [e:main,bend left=30,->] to (v2);
		\draw (v2) [e:main,bend left=30,->] to (v1);
		
		\draw (v2) [e:main,bend right=30,->] to (v3);
		\draw (v3) [e:main,bend right=30,->] to (v2);
		
		\draw (v3) [e:main,bend left=30,->] to (v4);
		\draw (v4) [e:main,bend left=30,->] to (v3);
		
		\draw (v4) [e:main,bend right=30,->] to (v5);
		\draw (v5) [e:main,bend right=30,->] to (v4);
		
		\draw (v5) [e:main,bend left=30,->] to (v6);
		\draw (v6) [e:main,bend left=30,->] to (v5);
		
		\draw (v1) [e:main,bend right=30,->] to (v6);
		\draw (v6) [e:main,bend right=30,->] to (v1);
		
		\draw (v1) [e:main,->] to (v4);
		
		\draw (v4) edge[e:main,->,quick curve through={(g1) (g2) (g3) (g4)}] (v1g);
		
%		\draw (v4) [e:main,->,out=330,in=30] to (v1);
		
		%%%%% %%%%% %%%%%
		
		%%%%% Right Center %%%%%

		%%%%% %%%%% %%%%%
		
		%%%%% %%%%% %%%%%
		
		\end{pgfonlayer}
		
		%%%%% %%%%% %%%%%

		%%%%% Background %%%%%
		\begin{pgfonlayer}{background}
		
		\end{pgfonlayer}	
		%%%%% %%%%% %%%%%
		
		%%%%% Foreground %%%%%
		\begin{pgfonlayer}{foreground}

		\end{pgfonlayer}
		%%%%% %%%%% %%%%%
		\end{tikzpicture}
	\end{center}
	\caption{A non-2-choosable strongly planar digraph.}
	\label{fig:listexample}
\end{figure}

\begin{theorem}
Let $D$ be a non-even digraph. Then $\vec{\chi}_\ell(D) \leq 3$. 
Moreover, for any choice of a designated vertex $v_0 \in V(D)$, $D$ is $\mathcal{L}$-list colourable for every list assignment $\mathcal{L}=\{L(v)|v \in V(D)\}$ fulfilling $|L(v_0)|=1$ and $|L(v)| \ge 3$ for all $v \in V(D) \setminus \{v_0\}$.
\end{theorem}
\begin{proof}
We show the second (stronger) assertion. Assume towards a contradiction that there is a non-even digraph $D$ which does not satisfy the assertion, and assume $D$ to be chosen minimal with respect to the number of vertices. Let in the following $\mathcal{L}$ be a fixed list assignment for $D$, where $|L(v_0)|=1$ for some designated $v_0 \in V(D)$, $|L(v)| \ge 3$ for all $v \in V(D) \setminus \{v_0\}$, and such that $D$ is not $\mathcal{L}$-choosable. Clearly, we have $|V(D)| \ge 3$. 

We first show that $D$ must be strongly $2$-connected: Assume for a contradiction that there is a directed separation of order $i\in\Set{0,1}$ in $D$. By \cref{lemma:01sumnoneven}, we find that there are non-even digraphs $D_1$ and $D_2$ with fewer vertices than $D$ such that $D$ is the $i$-sum of $D_1$ and $D_2$. By the assumed minimality of $D$, we know that $D_1$ and $D_2$ both satisfy the assertion. 

If $i=0$, consider a partition $(X,Y)$ of $V(D)$ such that $D_1=D[X], D_2=D[X]$ and the edges with exactly one endpoint in $X$ and exactly one endpoint in $Y$ form a directed cut in $D$. Restricting $\mathcal{L}$ to $X$ resp. $Y$ defines list assignments for $D_1$ and $D_2$ (each with at most one list of size less than $3$), and we find that $D_j$ admits a choice function $c_j$ for $j=1,2$ that defines a valid digraph colouring and satisfies $c_j(x) \in L(x)$ for all $x \in V(D_j)$. Putting 
\begin{align*}
	\Fkt{c}{x}\coloneqq\TwoCases{\Fkt{c_1}{x}}{x\in X}{\Fkt{c_2}{x}}{x\in Y}
\end{align*} now defines a valid choice of colours for $D$ without a monochromatic directed cycle, proving that $D$ is $\mathcal{L}$-choosable. This is a contradiction to our initial assumption. 

If $i=1$, let $w \in V(D)$ be such that $D$ is the $1$-sum of $D_1$ and $D_2$ along $w$. Consider a partition $(X,Y)$ of $V(D) \setminus \{w\}$ such that no edge in $D$ has its head in $X$ and its tail in $Y$, and such that $D_1$ arises from $D$ by identification of $Y \cup \{w\}$ into a single vertex $v_1$, and $D_2$ by identification of $X \cup \{w\}$ into a vertex $v_2$. 

We have that $v_0 \in X \cup \{w\}$ or $v_0 \in Y \cup \{w\}$. Assume for the following that $v_0 \in X \cup \{w\}$, the other case works symmetrically. Define an assignment $\mathcal{L}_1$ of lists to the vertices of $D_1$ according to $L_1(x):=L(x)$ for all $x \in X$ and $L_1(v_1):=L(w)$. Because $D_1$ satisfies the assertion, we find a choice function $c_1$ which defines a proper digraph colouring of $D_1$ while satisfying $c_1(x) \in L(x), x \in X$, and $\tilde{c}:=c_1(v_1) \in L(w)$.
Now define a list assignment $\mathcal{L}_2$ for $D_2$ according to $L_2(x):=L(x)$ for $x \in Y$ and $L_2(v_2):=\{\tilde{c}\}$. Because we have $|L_2(x)|=|L(x)| \ge 3$ for all $x \in Y=V(D_2) \setminus \{v_2\}$, we can apply the assertion to $D_2$ and thus find a choice function $c_2$ on $V(D_2)$ satisfying $c_2(x) \in L(x)$ for all $x \in Y$ and $c_2(v_2)=\tilde{c}=c_1(v_1)$. Now define a choice function $c$ on $V(D)$ by
\begin{align*}
	\Fkt{c}{x}\coloneqq\ThreeCases{\Fkt{c_1}{x}}{x\in X}{\tilde{c}}{x=w}{\Fkt{c_2}{x}}{x\in Y}
\end{align*}
By the above it is clear that we have $c(x) \in L(x)$ for all $x \in V(D)$. Because $D$ is not $\mathcal{L}$-choosable, this implies that there is a directed cycle $C$ in $D$ which is monochromatic under $c$. Because $c_1$ and $c_2$ are valid digraph colourings of $D_1$ and $D_2$, $C$ must contain vertices of both $X$ and $Y$ and therefore must visit $w$ as well as exactly one edge with tail in $X$ and head in $Y$. Therefore, identifying all vertices in $Y \cup \{v\}$ on $C$ into a single vertex results in a directed cycle in $D_1$, which has to be monochromatic as well. This finally is a contradiction to the definition of $c_1$.

As both cases led to a contradiction, for the rest of the proof we may assume that $|V(D)| \ge 3$ and $D$ is strongly $2$-connected. Applying \cref{cor:degreetwo} we find that there is a vertex $u \in V(D) \setminus \{v_0\}$ of out-degree two. Clearly, $D-u$ is non-even as well and has less vertices, so the minimality of $D$ implies that for the induced assignment $\mathcal{L}':=\{L(x)|x \in V(D) \setminus \{u\}\}$ of lists, there is a choice function $c'$ which defines a valid digraph colouring of $D-u$. Let $u_1,u_2$ be the two out-neighbours of $u$. Since $|L(u) \setminus \{c'(u_1),c'(u_2)\}| \ge 1$, we can extend $c'$ to a choice function $c$ on $V(D)$ such that $c(x)=c'(x) \in L(x)$ for all $x \in V(D) \setminus \{u\}$ and $c(u) \in L(u) \setminus \{c(u_1),c(u_2)\}$. Because $D$ is by initial assumption not $\mathcal{L}$-choosable, this implies that there is a directed cycle in $D$ which is monochromatic with respect to $c$. Since $c'$ defined a valid digraph colouring, this is only possible if the cycle traverses $u$ and thus one of the edges $(u,u_1)$ or $(u,u_2)$. However, this gives a contradiction to the fact that both of these edges are bi-coloured. 

This final contradiction shows that our initial assumption was false and concludes the proof of the Theorem.
\end{proof}
\section{Conclusive Remarks}

In this paper, we initiated the study of relationships between butterfly-minor closed classes of digraphs and the dichromatic number by characterising the largest butterfly-minor closed class of $2$-colourable digraphs.
Since odd bicycles have dichromatic number $3$, one direction of the following is \cref{thm:noneven}, the reverse follows from \cref{thm:mainthm}.

\begin{corollary} \label{cor:minorclosedclass}
The non-even digraphs form the unique inclusion-wise largest class $\mathcal{D}_2$ of 2-colourable digraphs which is closed under butterfly-minors.
\end{corollary}

In the undirected case, Hadwiger's Conjecture claims a characterisation of the largest minor-closed class of $k$-colourable graphs. 
In view of \cref{cor:minorclosedclass}, the following is a natural directed analogue.

\begin{question}
Given a natural $k \ge 3$, what is the largest butterfly-minor closed subclass $\mathcal{D}_k$ of the $k$-colourable digraphs?
\end{question}

Due to the existence of infinte antichains (such as the odd bicycles) in the butterfly-minor order of digraphs, we believe that for larger values of $k$, possibly no very simple description of the forbidden butterfly minors for $\mathcal{D}_k$ can be obtained.
Looking at the case $k=2$, this drastically changed when moving from digraphs to the corresponding bipartite graphs, where we only needed to exclude $K_{3,3}$ as a matching minor.
While by now the $K_{3,3}$-matching minor-free bipartite graphs (that is, the Pfaffian bipartite graphs) have many equivalent characterisations and can be recognised in polynomial time, not much is known about the classes of $K_{k,k}$-matching minor-free graphs with $k \ge 4$.
Clearly, the complete bipartite graph $K_{k,k}$ has $M$-chromatic number $k$ for any perfect matching.
Concerning \cref{cor:matchinghadwiger}, we think that the following analogue of Hadwiger's Conjecture for $M$-colourings of bipartite graphs could be true.

\begin{conjecture} \label{con:biphadwiger}
Let $k \in \mathbb{N}$, $G$ be a bipartite graph and $M$ an arbitrary perfect matching of $G$, such that $\chi(G,M) \ge k$. Then $G$ contains $K_{k,k}$ as a matching minor.
\end{conjecture}

While for $k=1,2$, the statement is trivial, the case $k=3$ amounts to \cref{cor:matchinghadwiger}. At the current state, we do not have a good approach for proving this conjecture even in the first open case of $k=4$, which
is mostly due to the fact that our proof for $k=3$ relied on a certain sparsity of Pfaffian bipartite graphs, which has not yet been established for classes excluding larger complete bipartite graphs as matching minors.

\begin{question}
Is there a function $f\colon\mathbb{N} \rightarrow \mathbb{N}$ such that every $(k-1)$-extendable bipartite graph $G$ without a $K_{k,k}$-matching minor on $n$ vertices has at most $\Fkt{f}{k}n$ edges?
In other words, is the average degree of these graphs bounded in terms of $k$?
\end{question}

The following observation, which is a direct consequence of a result of Aboulker et al. \cite{largesubdivisions}, provides some evidence towards \cref{con:biphadwiger}.
\begin{theorem}[Theorem 32 in \cite{largesubdivisions}] \label{thm:largesubdivision}
Let $D$ and $F$ be digraphs, $m:=|E(F)|, n:=|V(F)|$. If $\vec{\chi}(D) \ge 4^m(n-1)+1$, then $D$ contains a subdivision of $F$ as a subdigraph.
\end{theorem}

\begin{corollary}
There is a function $f: \mathbb{N} \rightarrow \mathbb{N}$ such that for any $k \in \mathbb{N},$ every bipartite graph $G$ with a perfect matching $M$ satisfying $\chi(G,M) \ge f(k)$ contains $K_{k,k}$ as a matching minor. 
\end{corollary}
\begin{proof}
Set $f(k):=4^{k^2-k}(k-1)+1$ and let $G$ be a bipartite graph with a perfect matching $M$ such that $\chi(G,M) \ge f(k)$. As the complete bioriented digraph $\Bidirected{K_k}$ has $k^2-k$ edges and $k$ vertices, we deduce from \cref{thm:largesubdivision} that $\vec{\chi}(\DirM{G}{M})=\chi(G,M) \ge f(k)$ implies the existence of a subdivision of $\Bidirected{K_k}$ as a subdigraph of $\DirM{G}{M}$. Clearly, this implies that $\Bidirected{K_k}$, which is the unique perfect matching-direction of $K_{k,k}$, is a butterfly minor of $\DirM{G}{M}$. The claim now follows from \cref{lemma:mcguigmatminors}.
\end{proof}

It would be interesting to see whether \cref{con:biphadwiger} would already imply Hadwiger's Conjecture for graphs.
While we do not have a proof of this implication yet, it does seem quite likely that a relation exists.
For this, note that the chromatic number of a graph can be expressed as the dichromatic number of its bidirection, and that the matching minors of the corresponding bipartite graph to some extent resemble the ordinary minors of the original graph.
Here, the complete graph $K_k$ yields the bidirected $k$-clique, which in the matching context corresponds to $K_{k,k}$.

An additional line of future research could be to investigate colouring properties of classes of digraphs which are closed under different notions of digraph minors. One such candidate are the \emph{topological minors}, which are defined similarly to the undirected case: A digraph $D_1$ is called a directed topological minor of another digraph $D_2$ if $D_2$ contains a subdivision of $D_1$ (that is, replacing directed edges by directed paths of positive length) as a subdigraph. It is easily seen that topological minors are always butterfly-minors, but that the converse fails in general. In any class of $2$-colourable digraphs which is closed under topological minors, the odd bicycles must form a set of forbidden minors. So far, we have been unable to decide the following question. If true, this statement would be a proper generalisation of \cref{thm:mainthm}. 
\begin{question}
Let $D$ be a digraph with $\vec{\chi}(D) \ge 3$. Must $D$ contain a subdivision of an odd bicycle?
\end{question}

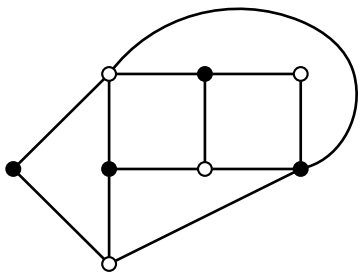
\begin{figure}[h]
	\begin{center}
	\begin{tikzpicture}[scale=0.7]
	
	\pgfdeclarelayer{background}
	\pgfdeclarelayer{foreground}
	
	\pgfsetlayers{background,main,foreground}
	
	%%%%% Vertex Styles %%%%%
	\tikzstyle{v:main} = [draw, circle, scale=0.5, thick,fill=black]
	\tikzstyle{v:tree} = [draw, circle, scale=0.3, thick,fill=black]
	\tikzstyle{v:border} = [draw, circle, scale=0.75, thick,minimum size=10.5mm]
	\tikzstyle{v:mainfull} = [draw, circle, scale=1, thick,fill]
	\tikzstyle{v:ghost} = [inner sep=0pt,scale=1]
	\tikzstyle{v:marked} = [circle, scale=1.2, fill=CornflowerBlue,opacity=0.3]
	%%%%% %%%%% %%%%%
	
	%%%%% Edge Styles %%%%%
	\tikzset{>=latex} 
	\tikzstyle{e:marker} = [line width=9pt,line cap=round,opacity=0.2,color=DarkGoldenrod]
	\tikzstyle{e:colored} = [line width=1.2pt,color=BostonUniversityRed,cap=round,opacity=0.8]
	\tikzstyle{e:coloredthin} = [line width=1.1pt,opacity=1]
	\tikzstyle{e:coloredborder} = [line width=2pt]
	\tikzstyle{e:main} = [line width=1pt]
	\tikzstyle{e:extra} = [line width=1.3pt,color=LavenderGray]
	%%%%% %%%%% %%%%%
	
	\begin{pgfonlayer}{main}
	
	%%%%% Centered Ghost Vertices %%%%%
	\node (C) [] {};
	
	%%%%% Left Center %%%%%
	\node (C1) [v:ghost, position=180:25mm from C] {};
	%\node (U1) [v:ghost, position=90:50mm from C1] {};
	%		\node (L1) [v:ghost, position=270:27mm from C1,align=center] {};
	
	%%%%% Center %%%%%
	\node (C2) [v:ghost, position=0:0mm from C] {};
	%\node (U2) [v:ghost, position=90:50mm from C2] {};
	%		\node (L2) [v:ghost, position=270:27mm from C2,align=center] {};
	
	%%%%% Right Center %%%%%
	\node (C3) [v:ghost, position=0:25mm from C] {};
	%\node (U3) [v:ghost, position=90:50mm from C3] {};
	%		\node (L3) [v:ghost, position=270:27mm from C3,align=center] {};
	%%%%% %%%%% %%%%%

	%%%%% Vertices %%%%%
	
	%%%%% Left Center %%%%%

	%%%%% %%%%% %%%%%
	
	%%%%% Center %%%%%

	\node (a1) [v:main,position=0:0mm from C2] {};
	\node (a4) [v:main,position=180:18mm from a1] {};
	\node (b2) [v:main,position=90:18mm from a1, fill=white] {};
	\node (b3) [v:main,position=0:18mm from a1, fill=white] {};
	\node (b1) [v:main,position=270:18mm from a1, fill=white] {};
	\node (a2) [v:main,position=90:18mm from b3] {};
	\node (a3) [v:main,position=0:18mm from b3] {};
	\node (b4) [v:main,position=90:18mm from a3, fill=white] {};
	
	\node (x) [v:ghost,position=0:10mm from b4] {};
	\node (y) [v:ghost,position=90:10mm from b4] {};
	
	%%%%% %%%%% %%%%%
	
	%%%%% Right Center %%%%%
	
	%%%%% %%%%% %%%%%
	
	%%%%% %%%%% %%%%%

	%%%%% Edges %%%%%
	
	%%%%% Left Center %%%%%

	%%%%% %%%%% %%%%%
	
	%%%%% Center %%%%%

	\draw (a1) [e:main] to (b1);
	\draw (a1) [e:main] to (b2);
	\draw (a1) [e:main] to (b3);
	
	\draw (a2) [e:main] to (b2);
	\draw (a2) [e:main] to (b3);
	\draw (a2) [e:main] to (b4);

%	\draw (a3) [e:main,bend right=120] to (b2);
	\draw (a3) [e:main] to (b3);
	\draw (a3) [e:main] to (b4);
	
	\draw (a4) [e:main] to (b1);
	\draw (a4) [e:main] to (b2);
	
	\draw (a3) [e:main] to (b1);
	
	%%%%% %%%%% %%%%%
	
	%%%%% Right Center %%%%%

	%%%%% %%%%% %%%%%
	
	%%%%% %%%%% %%%%%
	
	\end{pgfonlayer}
	
	%%%%% %%%%% %%%%%

	%%%%% Background %%%%%
	\begin{pgfonlayer}{background}
	
	\draw (a3) edge[line width=1pt,quick curve through={(x) (y)}] (b2);
	
	\end{pgfonlayer}	
	%%%%% %%%%% %%%%%
	
	%%%%% Foreground %%%%%
	\begin{pgfonlayer}{foreground}

	\end{pgfonlayer}
	%%%%% %%%%% %%%%%
	\end{tikzpicture}
\end{center}
\caption{A planar bipartite graph such that for any $2$-colouring of its edges, there is a perfect matching with a monochromatic alternating cycle.} \label{fig:counterex}
\end{figure}

Considering the notion of $M$-colourings, it is natural to ask whether it is necessary to have different colourings of the matching edges for every perfect matching, or whether one might strengthen \cref{cor:matchinghadwiger} by finding a single $2$-colouring of \emph{all} edges in a bipartite Pfaffian graph, such that for any perfect matching $M$ the induced $2$-colouring on the matching edges yields a proper $M$-colouring.
For an example consider the $2$-colouring of the staircase in \cref{fig:staircase}.
Although it seems to be possible to find such a ``super''-colouring for many bipartite Pfaffian graphs such as the Heawood graph or square grids, there are small examples of (even planar) Pfaffian bipartite graphs without such a colouring (cf. \cref{fig:counterex}).

Furthermore, all stated results and Conjectures are worthy to consider in the more general setting of solid graphs.
However, here, even more fundamental questions concerning the structure of these graphs are left widely open, see \cref{sec:pfaff}.

The questions raised in \cref{sec:polychromatic} concerning the relationship of directed girth and disjoint packings of feedback vertex sets might also apply to non-planar digraphs excluding certain butterfly-minors; in fact, \cref{thm:mainfractional} easily extends to the class of so-called \emph{mengerian digraphs} generalising the strongly planar digraphs (\cite{mengeriandigraphs}), with very similar properties.
To conclude, we want to mention the similarity of the treated problems with the following open subcase of a Conjecture of Woodall.

\begin{conjecture}[\cite{egres}]
	In every planar digraph $D$ of girth $g \geq 3$, there exists a packing of $g$ disjoint feedback arc sets.
\end{conjecture}

\bibliographystyle{alphaurl}
\bibliography{literature}

\end{document}